\documentclass[leqno]{siamltex}

\RequirePackage{yfonts}
\RequirePackage[active]{srcltx}
\RequirePackage{stmaryrd}
\RequirePackage{xcolor,graphicx}

\RequirePackage{amsmath,amsfonts,amssymb}
\RequirePackage{mathtools}
\RequirePackage{psfrag}
\RequirePackage{fancyheadings}
\RequirePackage{bm}

\RequirePackage{subfig}

\RequirePackage[hyphens]{url}

\RequirePackage[nodayofweek,dmyyyy]{datetime}

\RequirePackage[a4paper]{geometry}

\RequirePackage{bropd}

\graphicspath%
{%
    {./Parts/Images/}%
    {./Parts/Images/Test1/}%
    {./Parts/Images/Test2/}%
    {./Parts/Images/Test3/}%
    {./Parts/Images/Test4/}%
}

\newcommand{\trapici}[1]{``{#1}''}

\newcommand{\ie}{\textit{i.e.~}}

\newcommand{\norm}[1]{{\Vert {#1} \Vert}}

\newcommand{\Order}[1]{\mathcal{O}\left( #1 \right)}

\newcommand{\R}{\mathbb{R}}
\newcommand{\mesh}{\mathcal{M}}
\newcommand{\edges}{\mathcal{E}}
\newcommand{\intEdges}{\edges_{\rm int}}
\newcommand{\extEdges}{\edges_{\rm ext}}
\newcommand{\points}{\mathcal{P}}
\newcommand{\abs}[1]{| {#1} |}
\newcommand{\vect}[1]{\bm{{#1}}}
\newcommand{\normal}{\vect{n}}
\newcommand{\family}{{\mathcal{F}}}
\newcommand{\disc}{{\mathcal{D}}}

\newcommand{\bForm}[2]{A{\left( {#1}, {#2} \right)}}
\newcommand{\bFormDisc}[2]{A_{\disc}{\left( {#1}, {#2} \right)}}

\newcommand{\dNabla}{\nabla_{\disc}}

\newtheorem{remark}{Remark}

\newtheorem{hypotheses}{Hypotheses}

\newcommand{\defeq}{\vcentcolon=}

\newcommand{\eig}{\mathop{\rm eig}}

\newcommand{\Hdiv}{{\bm{H}}_{\rm div}}

\newcommand{\eye}{\vect{I}}

\newcommand{\ttau}{\vect{\tau}}

\newcommand{\n}{\vect{n}}

\newcommand{\dual}{\bm{u}}

\newcommand{\NN}{\vect{N}}

\newcommand{\TT}{\vect{T}}

\newcommand{\jjj}{\jump{\dual \cdot \n}{\gamma_j}}

\newcommand{\dualv}{\hat{u}_{\n}}

\newcommand{\V}{{\mathcal V }}
\newcommand{\Vhat}{\hat{\V}}

\newcommand{\X}{\V_{\disc}}
\newcommand{\Xz}{\V_{\disc,0}}
\newcommand{\Xhat}{\Vhat_{\disc}}
\newcommand{\Xhatz}{\Vhat_{\disc,0}}
\newcommand{\Y}{{\mathcal Y}_{\disc}}
\newcommand{\Yz}{{\mathcal Y}_{\disc,0}}

\newcommand{\mean}[2]{\left\{\!\!\left\{ {#1} \right\}\!\!\right\}_{#2}}

\newcommand{\jump}[2]{\left\llbracket {#1} \right\rrbracket_{#2}}

\newcommand{\fSet}{T}

\newcommand{\gamhat}{{\hat{\gamma}}}

\newcommand{\gamhati}{\gamhat_1}

\newcommand{\gamhatii}{\gamhat_2}

\newcommand{\gamhatiii}{\gamhat_{1,2}}

\newcommand{\gamhatj}{\gamhat_j}

\newcommand{\phat}{\hat{p}}

\newcommand{\phati}{\phat_1}

\newcommand{\phatii}{\phat_2}

\newcommand{\phatj}{\phat_j}

\newcommand{\dualhat}{\hat{\dual}}

\newcommand{\dualhatj}{\dualhat_j}

\newcommand{\qhat}{\hat{q}}

\newcommand{\qhatj}{\qhat_j}

\newcommand{\ccform}[2]{cc\big({#1}, {#2}\big)}
\newcommand{\aform}[2]{a\big({#1}, {#2}\big)}
\newcommand{\avform}[2]{av\big({#1}, {#2}\big)}
\newcommand{\jform}[2]{j\big({#1}, {#2}\big)}

\title
{
    A double-layer reduced model for fault flow on slipping domains
    with hybrid finite volume scheme
}

\author
{
    Alessio Fumagalli
    \and
    Isabelle Faille
}

\begin{document}


\maketitle


\begin{abstract}
    In this work we are interested in dealing with single-phase flows in
    fractured porous media for underground processes. We focus our attention on
    domains where the presence of faults, with thickness several orders of
    magnitude smaller than other characteristic sizes, can allow one part of the
    domain to slide past to the other. We propose a mathematical scheme
    where a reduced model for the fault flows is employed yielding a problem of
    co-dimension one. The hybrid finite volume method is used to obtain the
    discretized problem, which employs two different meshes for each fault, one
    associated with the porous-medium domain on each side of the fault. These
    two meshes can move with the corresponding domain, resulting in non-matching
    grids between the two parts of the fault. In an earlier paper a mathematical
    scheme was proposed where the numerical discretization considers the hybrid
    finite volume method. In this paper we focus on the well-posedness of the
    continuous problem, the convergence of the discretized problem, and with several
    numerical tests we support the theoretical findings.
\end{abstract}

\begin{keywords}
    Porous media, reduced model, faults, finite volume, non-matching grids
\end{keywords}

\begin{AMS}
    76S05, 65N08, 86A60
\end{AMS}


\pagestyle{myheadings}
\thispagestyle{plain}
\markboth{TEX PRODUCTION AND V. A. U. THORS}{SIAM MACRO EXAMPLES}



\section{Introduction}


Subsurface multi-phase flows in porous medium are strongly influenced by the
presence of heterogeneities and in particular by the effect of faults, in which
the flow can move differently in the surrounding medium both across and along
the fault.
Depending on the geophysical data, in particular the permeability,
the faults can act as barriers or preferential paths for the flow.
This behaviour is due to several factors: further fracturation
of the fault zones, chemical reactions or generation at different geological
times. The effect of the faults is extremely important for several applications,
like fractured aquifers, $CO_2$ injection and sequestration or oil and gas
reservoirs exploitation. See \cite{bear1993flow,Karimi-Fard2004,Gong2011}
for applications in real a context.

One of the most important aspects of faults is the difference between their
characteristic sizes. We call the fault aperture the portion of rock containing
the fault core
and the surrounding damaged zone. Its typical thickness ranges from meters to a few
tens of meters, while its length is generally of the same order as the size of
the domain of interest. Normally the latter has extension of hundreds of kilo-meters
with depth of tens of kilo-meters. Considering a conforming discretization of a
real sedimentary basin with several faults, a standard numerical approximation
can easily make the simulation unaffordable.
Even if the literature on flows in fractured porous media is
extensively developed, see for example
\cite{adler1999fractures,Berkowitz2002,adler2012fractured}, a general method is
not yet available which can handle all the difficulties of this particular
problem.

We focus our attention on the family of mathematical models which replace the
fine description of the fault with an approximate one. The main idea of these
models is to substitute the $N$-dimensional description of the fault, in an
$N$-dimensional domain, by a new object of codimension one (an $N-1$-dimensional
object embedded in the $N$-dimensional domain). New differential
equations and suitable interface conditions are derived to couple the new
problem. The firsts contributions were \cite{Alboin2000,Alboin2002}, where a
first \textit{reduced model} (RM) is derived for only conductive faults, which
cut entirely the domain. The fault mesh is composed of a set of contiguous edges
of cells from the porous medium mesh, the method limits in this way the
computational cost. Numerical experiments and theoretical results show the
good behaviour of the proposed method. The authors in
\cite{Faille2002,Jaffre2002,Angot2003} consider a more general model where
low permeable faults can also be taken into account. Finite element and finite
volume approximations are considered with different numerical experiments. Three
dimensional experiments, with realistic geometry and intersecting faults are
presented in \cite{Amir2005}. In the work \cite{Martin2005} the authors consider
a further generalization of the interface conditions, where different a-priori
assumptions of fault pressure shapes in the normal direction are considered
yielding a new
RM with a model parameter.  Theoretical analyses and numerical experiments show
the robustness of the results in different situations. In this article we will
refer to such a model as a \textit{single-layer reduced model} (SLRM).  The
authors in \cite{Angot2009} consider a partially immersed fault with new
coupling conditions at the fault tips.  Two-phase flow in porous media are
considered in \cite{Jaffre2011,Elyes2015} where a RM was introduced for this
problem.

Based on the aforementioned mathematical model a different coupling approach was
introduced in \cite{DAngelo2011}. In this article the fault discretization is
completely independent of the mesh of the porous medium. The extended finite
element method (XFEM) is used to handle this geometrical non-conformity,
yielding a very flexible tool for real simulations. Again with the same type of
approximation we mention \cite{Fumagalli2012a} for a description of convection
and diffusion of a passive scalar in a porous media.  In \cite{Fumagalli2012d}
the two-phase flow problem is considered with different approximation of the
hyperbolic fluxes: upstream mobility and exact Riemann solver. In
\cite{Formaggia2012,Fumagalli2012g} a general RM is presented for a
network of faults where suitable interface conditions are considered in the
intersecting regions.

Finally in \cite{Tunc2012, Faille2014a} the authors assume that one part of the domain can
slip, because of the fault, with respect to another part. To handle this new
feature, a new model is consider with a two layers approximation. In contrast to
the SLRM, we will refer to the method proposed in \cite{Tunc2012} as a
\textit{double-layer reduced model} (DLRM) or simply (DL). Each part of the
domain, situated along the fault, has its own fault approximation. Suitable
interface conditions are considered for the layer-layer coupling.

In this work we continue the analyses of the mathematical scheme proposed in
\cite{Faille2014a}, where an approximation using the hybrid finite
volume scheme \cite{Eymard2010} is considered for both the rock matrix, the
fault, and their coupling. Furthermore the method can handle generic
permeability fields as well as enforce local mass conservation for each cell.
We present the DLRM, introducing its weak formulation and
showing its well posedness. Numerical discretization with different theoretical
results, including the convergence and model error, are presented in detail. A
complex example with a sliding domain shows the effectiveness of the proposed
approach also in such a situation.

This paper is organized as follow: in Section \ref{sec:math_model} the notations
and the governing equations for the RM are presented as well as the
analysis in the continuous spaces. Section \ref{sec:numerical_approximation} is
devoted to the presentation of the discretization of the proposed schemes along with some
important theoretical results.  In Section \ref{sec:examples} a collection of
examples highlights the potential of the proposed methods.  Finally, Section
\ref{sec:conclusion} contains the conclusions.



\section{Mathematical problem} \label{sec:math_model}


\def\pathImages{./Parts/Images}


To ease the presentation we consider only one single fault that cuts
entirely through
the domain. The method can be generalized without any additional difficulties if
we consider several non-intersecting faults.


\subsection{Physical equations}


Let us set, from now on, $i$ and $j$ indices with values $i \in
\left\{1,2,f\right\}$ and $j \in \left\{1,2\right\}$. We consider a regular
domain $\Omega \subset \R^N$, $N=2$ or 3, with Lipschitz-continuous boundary
denoted by $\Gamma \defeq \overline{\Omega} \setminus \Omega$. We suppose that
$\Omega$ is divided into three disjoint subsets, such that $\overline{\Omega} =
\cup_i \overline{\Omega_i}$, where $\Omega_f$ represents the fault. Moreover the
boundary is divided into $\Gamma_i \defeq \Gamma \cap \partial \Omega_i$. Figure
\ref{fig:domain} shows an example.
\begin{figure}[!htp]
    \centering
    \includegraphics{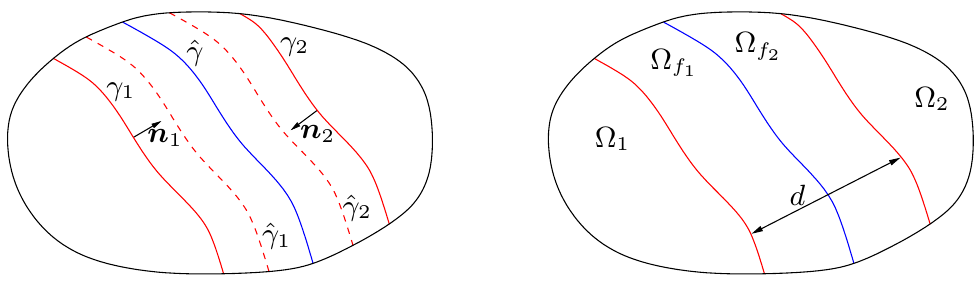}
    \caption{Representation of each sub-domain where the thickness of the fault
        is exaggerated for visualization purpose.}%
    \label{fig:domain}
\end{figure}
The interfaces, of codimension one, between the domain $\Omega_j$ and $\Omega_f$
are denoted as $\gamma_j \subset \R^{N}$ with unit normal $\n_j$, pointing
outwards from $\Omega_j$. Since $\Omega_1$ can slide along $\Omega_2$, or vice
versa, due to the fault we subdivide the latter into two disjoint layers
$\Omega_{f_j}$, such that $\overline{\Omega_f} = \cup_j
\overline{\Omega_{f_j}}$, defined in the sequel. Then, following
\cite{Martin2005}, we suppose that there exists a manifold $\gamhat \subset
\R^N$ of co-dimension one and of class piecewise-$C^2$ such that $\hat{\gamma}$
represents the centre of the fault and $\Omega_f$ may be defined as
\begin{gather} \label{domainFdefinition}
    \Omega_{f_j} = \left\{ \vect{x} \in \R^N: \vect{x} = \vect{s} + r \n, \,
    \vect{s} \in \gamhat, \, r \in \fSet_j \right\}
    \quad \text{with} \quad
    \fSet_1 \defeq (-d/2,0), \,\,
    \fSet_2 \defeq (0,d/2).
\end{gather}
In \eqref{domainFdefinition} we have denoted by $d \in C^2 ( \gamhat )$
the thickness of $\Omega_f$ and $\n$ the unit normal of
$\gamhat$, pointing from $\Omega_1$ to $\Omega_2$. We assume that $|\gamhat| \gg d$ and there
exist $c_1, c_2 \in \R^+$, with $c_2$ \trapici{small}, such that $d \left(
\vect{s} \right) > c_1$ and $\abs{d^\prime \left( \vect{s} \right)} < c_2$ for all
$\vect{s} \in \gamhat$, \ie the thickness of $\Omega_f$ is small and varies
slowly compared to its other dimensions. Moreover we introduce the centre
line $\gamhatj$ of the fault layer $\Omega_{f_j}$, translating
$\gamhat$ to the middle of $\fSet_j$. We indicate with a lower case subscript the
restriction of data and unknowns to the corresponding sub-domain of $\Omega$.
Finally we define the surrounding domain as $\overline{\Omega_{1,2}} \defeq \cup_j
\overline{\Omega_j}$ and the fault centre line as $\overline{\gamhatiii} \defeq
\cup_j \overline{\gamhatj}$.

We are interested in computing the steady pressure field $p$ and the velocity
field, or Darcy velocity, $\dual$ in the whole domain $\Omega$,
governed by the following Darcy problems, with the classical interface conditions,
formulated in $\Omega_i$. For simplicity we assume homogeneous boundary
conditions for the pressure on $\Gamma$. The problem is: find $(p, \dual)$
such that
\begin{gather} \label{eq:darcy_problem}
    \begin{aligned}
        &\!\!\!\begin{array}{ll}
            \nabla \cdot \dual_i = q_i \\
            \dual_i + \Lambda_{i} \nabla p_i = \vect{0}
        \end{array}
        &&\text{in } \Omega_i\\
        &p_i=0 &&\text{on } \Gamma_i
    \end{aligned}
    \qquad \text{with} \qquad
    \begin{aligned}
        &p_j = p_f \\
        &\dual_j \cdot \n_j = \dual_f \cdot \n_j
    \end{aligned}
    \text{ on } \gamma_j.
\end{gather}
Here $\Lambda_{i} \in \left[ L^\infty \left( \Omega_i \right) \right]^{N\times
N}$ denotes the permeability tensor, such that for almost every $\vect{x} \in
\Omega_i$ is symmetric and positive definite. More specifically we require that
its eigenvalues are included in $0 < \underline{\lambda}_i \leq \eig \Lambda_i
\leq \overline{\lambda}_i$, with $\underline{\lambda}_i , \overline{\lambda}_i
\in \R^+$. In \eqref{eq:darcy_problem} $q_i \in L^2 \left( \Omega_i \right)$ is
a scalar source term which may represents a possible volume source or sink.

We have the following standard result for the Darcy problem, see \cite{Brezzi1991,Quarteroni1994,Ern2004}.
\begin{theorem}
    Under the given hypothesis on the data, problem \eqref{eq:darcy_problem} is well posed.
    In particular, we have $\left( \dual, p \right) \in \Hdiv \left( \Omega \right)
    \times L^2 \left( \Omega \right)$.
\end{theorem}


\subsection{The reduced model}


For readers convenience we recall the main results and a brief derivation of the
DLRM, a more detailed derivation can be found in
\cite{Tunc2012,Faille2014a}.  We introduce the projection matrices in the normal
and tangential directions of $\gamhat$ as $\NN \defeq \n \otimes \n$ and $\TT
\defeq \eye - \NN$, respectively. The Darcy velocity in the fault can be
decomposed into its normal and tangential parts as $\dual_f = \NN \dual_f + \TT
\dual_f = \dual_{f,\n} + \dual_{f,\ttau}$, with $\dual_{f,\n} \defeq \NN
\dual_f$ and $\dual_{f,\ttau} \defeq \TT \dual_f$.  Moreover we introduce also
the normal and tangential divergence and gradient on $\gamhat$, given $\bm{v}$
and $v$ two regular functions we define
\begin{gather*}
    \nabla \cdot \bm{v} = \nabla_{\n} \cdot \bm{v} + \nabla_{\ttau} \cdot \bm{v}
    \quad \text{with} \quad
    \nabla_{\n} \cdot \bm{v} \defeq \NN : \nabla \bm{v}
    \quad \text{and} \quad \nabla_{\ttau}
    \cdot \bm{v} \defeq \TT : \nabla \bm{v}, \\
    \nabla v = \nabla_{\n} v + \nabla_{\ttau} v
    \quad \text{with} \quad
    \nabla_{\n} v \defeq \NN \nabla v
    \quad \text{and} \quad \nabla_{\ttau} v \defeq \TT \nabla v.
\end{gather*}
The conservation equation, for each side of the fault, is integrated along its
normal direction on $\fSet_j$ to obtain a conservation equation written in the
tangential space of $\gamma_j$
\begin{gather} \label{eq:red_conserv}
    \nabla_{\ttau} \cdot \dualhatj = \qhatj + \jjj
    \quad \text{in } \gamhatj.
\end{gather}
In the latter equation we have indicated with $\dualhatj$ the reduced flux
for each layer of the fault, defined as $\dualhatj \defeq \int_{\fSet_j}
\dual_{f,\ttau}$, and the reduced source term as $\qhatj \defeq \int_{\fSet_j}
q_f$. Moreover $\jjj$ indicates the jump of the flux across the corresponding
layer of the fault, defined as
\begin{gather*}
    \jjj \defeq
    (-1)^j(\left. \dual_f \cdot \n \right|_\gamhat - \left. \dual_f \cdot \n \right|_{\gamma_j})
    = (-1)^j\left(\left. \dualv - \dual_j  \cdot \n \right|_{\gamma_j}\right),
\end{gather*}
where $\dualv$ stands for $\left.
\dual_f \cdot \n \right|_\gamhat$. The Darcy equation require that the permeability in the
fault can be written as $\Lambda_f = \lambda_{f,\n} \NN + \lambda_{f,\ttau}
\TT$, with $\lambda_{f,\n}$ and $\lambda_{f,\ttau}$ strictly positive for almost
every $\vect{x} \in \Omega_f$. For a more general case refer to \cite{Angot2009}.
Considering the projected Darcy equation on the tangential space of $\gamhatj$,
integrated in the normal direction of the latter, we obtain
\begin{gather} \label{eq:red_darcy}
    \dualhatj + \hat{\lambda} \nabla_{\ttau} \phatj = \vect{0}
    \quad \text{in } \gamhatj,
\end{gather}
where $\phatj$ is the reduced pressure in each part of the fault, defined as
$\phatj \defeq \frac{2}{d} \int_{\fSet_j} p_f$, and $\hat{\lambda}$ is the
effective permeability in the tangential direction, defined as $\hat{\lambda}
\defeq d \lambda_{f,\ttau} / 2$. We can consider a different value of
$\hat{\lambda}$ for each layer but, for easy of the presentation, we avoid to
specify it. In Section \ref{subsec:slipping_domain} we present an example with
different value of $\hat{\lambda}$ for each layer of the fault. Projecting the
Darcy equation on the normal space of the fault and integrating in the normal
direction on the first half of $\fSet_1$, and on the second half of $\fSet_2$
respectively, we end up with the coupling conditions
\begin{gather*} \label{eq:red_cc1}
    \dual_1 \cdot \n = 2 \lambda_\gamhat
    \left( p_1 - \phati \right)
    \quad \text{and} \quad
    \dual_2 \cdot \n = 2 \lambda_\gamhat
    \left( \phatii - p_2 \right)
\end{gather*}
where $\lambda_\gamhat$ is the effective permeability in the normal direction of
the fault, defined as $\lambda_\gamhat \defeq 2 \lambda_{f,\n} / d$. In the
latter equations we have used a suitable approximation of the integral of
$\dual_i \cdot \n$. We need to introduce an additional equation to express the
coupling of the velocity between the two sides of the fault.  We consider again the
projection of the Darcy equation on the normal space of $\gamhat$ and integrating, in the
normal direction, between the second half of $\fSet_1$ and the first half of
$\fSet_2$ we obtain
\begin{gather} \label{eq:red_cc2}
    \dualv = \lambda_\gamhat \jump{\phat}{\gamhat},
\end{gather}
where, in this case, the jump operator is defined as $\jump{\phat}{\gamhat}
\defeq \phati - \phatii$.
Considering \eqref{eq:darcy_problem} for $i=j$ coupled with
\eqref{eq:red_conserv}, \eqref{eq:red_darcy}, \eqref{eq:red_cc1} and
\eqref{eq:red_cc2} we end up with the following problem: find $(\dual_j,p_j)$
and $(\dualhatj, \phatj)$ such that
\begin{subequations} \label{eq:red_problem}
\begin{gather} \label{eq:red_problem1}
    \begin{aligned}
        &\!\!\!\begin{array}{ll}
            \nabla \cdot \dual_i = q_i \\
            \dual_i + \Lambda_{i} \nabla p_i = \vect{0}
        \end{array}
        &&\text{in } \Omega_i\\
        &p_i=0 &&\text{on } \Gamma_i
    \end{aligned}
    \qquad \text{and} \qquad
    \begin{aligned}
        &\!\!\!\begin{array}{ll}
            \nabla_{\ttau} \cdot \dualhatj = \qhatj + \jjj \\
            \dualhatj + \hat{\lambda} \nabla_{\ttau} \phatj = \vect{0}
        \end{array}
        &&\text{in } \gamhatj\\
        &\phatj = 0 &&\text{on } \partial \gamhatj,
    \end{aligned}
\end{gather}
with the coupling conditions
\begin{align} \label{eq:red_problem2}
    \begin{alignedat}{2}
        &\dual_1 \cdot \n = 2 \lambda_\gamhat
        \left( p_1 - \phati \right) & \quad& \text{on } \gamhati\\
        &\dual_2 \cdot \n = 2 \lambda_\gamhat
        \left( \phatii - p_2 \right) && \text{on } \gamhatii\\
        &\dualv = \lambda_\gamhat \jump{\phat}{\gamhat} &&
        \text{on } \gamhat
    \end{alignedat}
\end{align}
\end{subequations}
Summing and subtracting the first two equations of \eqref{eq:red_problem2} we
end up with an equivalent set of coupling conditions
\begin{gather} \label{eq:red_problem3}
    \tag{\ref{eq:red_problem2}-bis}
    \begin{aligned}
        &\mean{\dual\cdot \n}{\gamhat} = \lambda_\gamhat \left( \jump{p}{\gamhat} -
        \jump{\phat}{\gamhat} \right)\\
        &\jump{\dual \cdot \n}{\gamhat} = 4 \lambda_\gamhat \left(
        \mean{p}{\gamhat} - \mean{\phat}{\gamhat} \right)\\
        &\dualv = \lambda_\gamhat \jump{\phat}{\gamhat}
    \end{aligned}
    \quad \text{on } \gamhat,
\end{gather}
where we have indicated by $\mean{\phat}{\gamhat} \defeq \frac{1}{2} \left(
\phati - \phatii \right)$ and, with an abuse of notations, by
$\mean{\dual\cdot \n}{\gamhat} \defeq \frac{1}{2} \left( \dual_1 \cdot \n +
\dual_2 \cdot \n \right)$,
$\jump{\dual \cdot \n}{\gamhat} \defeq \left( \dual_1 \cdot \n - \dual_2 \cdot
\n \right)$, $\jump{p}{\gamhat} \defeq p_1 - p_2$ and $\mean{p}{\gamhat} \defeq
\frac{1}{2} \left( p_1 + p_2 \right)$.

\begin{remark}
    In the sequel we will use a numerical scheme based on the primal formulation
    of \eqref{eq:red_problem}, since it is a trivial derivation we will refer to this problem
    for both its dual or primal formulation.
\end{remark}


\subsection{Weak formulation}


In the sequel we will use the symbols $a \lesssim b \Leftrightarrow a \leq c_1
b$ and $a \gtrsim b \Leftrightarrow a \geq c_2 b$ for some $c_1, c_2 \in \R^+$
dependent only on the data problem of \eqref{eq:red_problem} or on data which
are not important for the analyses. The constants are independent from the
grid size. First of all we
introduce the functional setting for problem \eqref{eq:red_problem}. We consider
the functional spaces $\V_j \defeq H_{\Gamma_j}^1\left( \Omega_j \right)$, endowed with the
usual norms, and the global functional space for the domain $\V \defeq \prod_j
\V_j$.  Moreover we define
\begin{gather*}
    \Vhat_j \defeq \left\{ \hat{v}_j:\, \hat{v}_j|_{\partial \gamhatj} = 0,
    \hat{v}_j \in L^2\left( \gamhatj \right) \text{ and }
    \nabla_{\ttau} \hat{v}_j \in \left[ L^2\left( \gamhatj \right) \right]^{N-1}
    \right\}
    \quad \text{and} \quad
    \Vhat \defeq \prod_j \Vhat_j,
\end{gather*}
with norms
\begin{gather*}
    \norm{\hat{v}_j}_{\Vhat_j}^2 \defeq \norm{\hat{v}_j}_{L^2(\gamhatj)}^2 +
    \norm{\nabla_{\ttau} \hat{v}_j}_{L^2(\gamhatj)}^2, \quad
    \norm{v}_{\V}^2 \defeq \sum_{j} \norm{v}_{\V_j}^2, \\
    \norm{\hat{v}}_{\Vhat}^2 \defeq \sum_{j}
    \norm{\hat{v}_j}_{\Vhat_j}^2
    \quad \text{and} \quad
    \norm{\left( v, \hat{v} \right)}_{\V\times \Vhat}^2 \defeq \norm{v}_{\V}^2 +
    \norm{\hat{v}}_{\Vhat}^2
\end{gather*}
Considering $\left( \cdot, \cdot \right)_A : L^2(A) \times L^2(A)
\rightarrow \R$ the scalar product in $L^2(A)$, with $A \subset \Omega$, we introduce
the bilinear forms for the diffusive parts as
\begin{gather*}
    a_\Omega \left(p,v\right) \defeq \sum_{j} \left( \Lambda_j \nabla p_j,
    \nabla v_j \right)_{\Omega_j }
    \quad \text{and} \quad
    a_\gamhat \left(\phat,\hat{v}\right) \defeq  \sum_{j} \left( \hat{\lambda}
    \nabla_{\ttau} \phatj, \nabla_{\ttau} \hat{v}_j \right)_{\gamhatj},\\
    a \left( \left( p, \phat \right), \left( v, \hat{v} \right) \right) \defeq
    a_\Omega \left(p,v\right)+ a_\gamhat \left(\phat,\hat{v}\right).
\end{gather*}
Moreover we consider also the bilinear forms for the coupling conditions,
considering \eqref{eq:red_problem3}, for the jumps and averages as
\begin{gather*}
    \avform{\left( p, \phat \right)}{\left( v, \hat{v} \right)}
    \defeq 4 \left( \lambda_\gamhat \mean{p}{\gamhat} - \lambda_\gamhat
    \mean{\phat}{\gamhat},
    \mean{v}{\gamhat} - \mean{\hat{v}}{\gamhat} \right)_\gamhat,\\
    \jform{\left( p, \phat \right)}{\left( v, \hat{v} \right)}
    \defeq \left( \lambda_\gamhat \jump{p}{\gamhat} -
    \lambda_\gamhat \jump{\phat}{\gamhat},  \jump{v}{\gamhat} -
    \jump{\hat{v}}{\gamhat} \right)_\gamhat + \left( \lambda_\gamhat
    \jump{\phat}{\gamhat} , \jump{\hat{v}}{\gamhat} \right)_\gamhat,\\
    \ccform{\left( p, \phat \right)}{\left( v, \hat{v} \right)}
    \defeq \avform{\left( p, \phat \right)}{\left( v, \hat{v} \right)
    } + \jform{\left( p, \phat \right)}{\left( v, \hat{v} \right)
    }
\end{gather*}
or considering the equivalent form \eqref{eq:red_problem2} we introduce
\begin{gather*}
    \ccform{ \left( p, \phat \right)}{ \left( v, \hat{v} \right)}
    \defeq 2 \left( \lambda_\gamhat p_1 - \lambda_\gamhat \phati,
    v_1 - \hat{v}_1 \right)_{\gamhati} + 2  \left(  \lambda_\gamhat p_2 -
   \lambda_\gamhat \phatii, v_2 - \hat{v}_2
   \right)_{\gamhatii} + \\ + \left( \lambda_\gamhat
    \jump{\phat}{\gamhat}, \jump{\hat{v}}{\gamhat} \right)_{\gamhat}.
\end{gather*}
The global bilinear form is defined as
\begin{gather*}
    \bForm{\left( p, \phat \right)}{\left( v, \hat{v}\right)}
    \defeq \aform{\left( p, \phat \right)}{\left( v, \hat{v}
    \right)} + \ccform{\left( p, \phat \right)}{ \left( v, \hat{v}
    \right)}.
\end{gather*}
Finally we introduce the functional for the right-hand side
\begin{gather*}
    F\left(v, \hat{v} \right) \defeq \sum_{j} \left( q_j,  v_j
    \right)_{\Omega_j} + \left( \qhatj, \hat{v}_j
    \right)_{\gamhatj}.
\end{gather*}
We present the weak formulation for problem \eqref{eq:red_problem}:
find $\left(p, \phat \right) \in \V \times \Vhat$ such that
\begin{gather} \label{eq:weak_form}
    \bForm{\left( p, \phat \right)}{\left( v, \hat{v}\right)}
    = F\left(v, \hat{v} \right) \qquad \forall \left( v,
    \hat{v} \right) \in \V \times \Vhat.
\end{gather}
\begin{lemma}[Well posedness]
    Problem \ref{eq:weak_form} is well posed, moreover
    $\norm{\left(p,\phat\right)}_{\V \times \Vhat} \lesssim 1$.
\end{lemma}
\begin{proof}
    Clearly all the bilinear forms and the functional introduced are linear,
    we are going to apply the Lax-Milgram theorem and obtain
    the existence and uniqueness of the solution. We prove the continuity of
    the bilinear forms, introducing $c_0 = \displaystyle \max_{i=1,2}
    \norm{\Lambda_i}_{L^\infty(\Omega_i)}$, we have
    \begin{gather*}
        \abs{\aform{\left( p, \phat \right)}{\left( v, \hat{v} \right)}}
        \leq \abs{a_\Omega \left(p,v\right)} + \abs{a_\gamhat
        \left(\phat,\hat{v}\right)} \leq
        c_0\norm{p}_{\V} \norm{v}_{\V} + \norm{\hat{\lambda}}_{L^\infty(\gamhat)}
        \norm{\phat}_{\Vhat} \norm{\hat{v}}_{\Vhat},
    \end{gather*}
    considering the maximum between $c_0$ and the norm of $\hat{\lambda}$ we
    obtain the bound for the bilinear form: $\abs{\aform{\left( p, \phat
    \right)}{\left( v, \hat{v} \right)}} \lesssim
    \norm{\left(p,\phat\right)}_{\V\times \Vhat
    }\norm{\left(v,\hat{v}\right)}_{\V\times\Vhat}$.
    We consider now the bilinear forms associated to the coupling conditions
    \begin{gather*}
        \abs{\ccform{\left( p, \phat \right)}{\left( v, \hat{v} \right)
        }} \leq \abs{\avform{\left( p, \phat \right)}{\left( v,
        \hat{v} \right)}} + \abs{\jform{\left( p, \phat \right)}{
        \left( v, \hat{v} \right)}},
    \end{gather*}
    using the inequality for the averages and the jumps operators introduced in
    \cite{Angot2003}, \ie
    \begin{gather*}
        \norm{\mean{v}{\gamhat}}_{L^2(\gamhat)} \lesssim \norm{v}_{\V}
        \quad \text{and} \quad
        \norm{\jump{v}{\gamhat}}_{L^2(\gamhat)} \lesssim \norm{v}_{\V},
    \end{gather*}
    then we have
    \begin{gather*}
       \abs{\avform{\left( p, \phat \right)}{\left( v,\hat{v}
       \right)}} \lesssim \abs{\left(\mean{p}{\gamhat} -
       \mean{\phat}{\gamhat},\mean{v}{\gamhat} - \mean{\hat{v}}{\gamhat}
       \right)_\gamhat} \lesssim
       \norm{\left(p,\phat\right)}_{\V\times\Vhat}\norm{\left(v,\hat{v}\right)
       }_{\V\times\Vhat},
       \\
       \abs{\jform{\left( p, \phat \right)}{\left( v, \hat{v}
       \right)}} \lesssim \abs{\left(\jump{p}{\gamhat} -
       \jump{\phat}{\gamhat} ,\jump{v}{\gamhat}  -
       \jump{\hat{v}}{\gamhat}\right)_\gamhat} + \abs{\left(
       \jump{\phat}{\gamhat} , \jump{\hat{v}}{\gamhat}\right)_\gamhat} \lesssim
       \norm{\left(p,\phat\right)}_{\V\times\Vhat}\norm{\left(v,\hat{v}\right)
       }_{\V\times\Vhat}
    \end{gather*}
    The functional in \eqref{eq:weak_form} is clearly continuous, while the
    coercivity of the global bilinear form is proved given the positivity of
    $\ccform{\cdot}{\cdot}$ and the coercivity of the stiffness bilinear form. The bound on the
    solution is obtained considering the coercivity of $a$ and the continuity of
    $F$.
\end{proof}




\section{Numerical approximation} \label{sec:numerical_approximation}


We introduce the definition of discretization for $\Omega_{1,2}$, the porous
medium domain. The discretization of the fault follows from the discretization of
the domain.  We report in Figure \ref{fig:cone_full} a schematic representation
of some notations we introduce in Definition \ref{definition:discretization}.
\begin{figure}[!htp]
    \centering
    \includegraphics{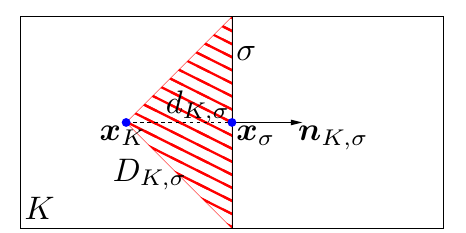}
    \caption{Notation useful for the numerical scheme, given a cell $K$.}%
    \label{fig:cone_full}
\end{figure}
\begin{definition}[Discretization of $\Omega_{1,2}$] \label{definition:discretization}
    A discretization of $\Omega_{1,2}$, denoted by $\disc$, is defined as the triplet
    $\disc \defeq (\mesh, \edges, \points)$ where

    \begin{itemize}

        \item $\mesh$ is the set of control volumes. The control volumes are non-empty
            connected and disjoint subset of $\Omega_{1,2}$ such that
            $\overline{\Omega_{1,2}} = \cup_{K \in \mesh} \overline{K}$. Let $\abs{K}
            > 0$ the measure of $K$ and $h_K \in \R^+$ its diameter. We indicate
            by $\displaystyle h_\disc \defeq \sup \left\{ h_K, K \in \mesh \right\}$ the
            diameter of the discretization;

        \item $\edges$ is the set of the edges, divided into the set of external
            edges $\extEdges = \partial \Omega$, the set of internal edges
            $\intEdges$ and the set of fault edges $\edges_\gamma = \partial
            \Omega_1 \cap \partial \Omega_2$;
            we have $\edges=\intEdges \cup \extEdges \cup \edges_\gamma$. Let
            $\abs{\sigma} > 0$ the measure of $\sigma$. We denote by
            $\edges_K \subset \edges$ the set of all edges of a control volume
            $K$ and by
            $\mesh_\sigma \defeq \left\{ K \in \mesh: \, \sigma \in \edges_K
            \right\}$ the set of all elements facing a given edge $\sigma$;

        \item $\points$ is the set of points, defined by $\points \defeq \left(
            \vect{x}_K \right)_{K \in \mesh} \cup \left( \vect{x}_\sigma \right)_{\sigma \in
            \mesh}$, where $\bm{x}_K$ is the centre of mass for the cell $K \in
            \mesh$ and $\bm{x}_\sigma$ is the barycentre of the face
            $\sigma \in \edges$;

        \item for any cell $K \in \mesh$ and face $\sigma \in \edges_K$ we
            indicate by $\normal_{K,\sigma}$ the unit vector normal to $\sigma$ outward to
            $K$;

        \item $D_{K,\sigma} \in K$ is the cone with vertex $\vect{x}_K$ and basis $\sigma
            \in \edges_K$. We indicate with $d_{K,\sigma} \in \R^+$ the orthogonal
            distance between $\bm{x}_K$ and $\sigma$.

    \end{itemize}

    The set, or family, of all the discretizations $\disc$ is denoted by $\family$.

\end{definition}
We introduce also a parameter that measure the quality of the mesh
\begin{gather} \label{eq:quality_mesh}
    \theta_\disc \defeq \max \br{ \max_{\sigma \in \intEdges, \mesh_\sigma =
    \left\{ K, L\right\} } \dfrac{d_{K,\sigma}}{d_{L,\sigma}},
    \max_{K \in \mesh, \sigma \in \edges_K}
    \dfrac{h_K}{d_{K,\sigma}} }
\end{gather}
For the discretization of the fault, in problem \eqref{eq:red_problem}, we
suppose that $\disc$ is conforming with the fault,
\ie the fault is represented by a set of continuous edges of $\edges_\gamma$. However
we allow a non-matching approximation of $\gamhati$ and $\gamhatii$. We
indicate with $\hat{\disc} = \left( \hat{\mesh}, \hat{\edges}, \hat{\points}
\right) \in \hat{\family}$ the discretization of the fault, where $\hat{\mesh}$
is the set of control volumes of the approximation of $\gamhatiii$. We
consider the same notation of Definition \ref{definition:discretization} where
$\hat{\family}$ is considered instead of $\family$. Thanks to Definition
\ref{definition:discretization} we have $\n_{K, \sigma} = - \n_{L, \sigma}$ for
each $\mesh_{\sigma} = \left\{K, L \right\}$, while for the fault we assume the
following statement.
\begin{hypotheses}[Normal discrepancy]\label{hyp:normal}
    For each $\sigma \in \hat{\edges}$, with $\mesh_\sigma = \left\{K, L \right\}$, we suppose that
    \begin{gather*}
        \n_{K, \sigma} + \n_{L, \sigma} = \Order{h_\disc}
        \quad \text{as} \quad h_\disc \rightarrow 0.
    \end{gather*}
\end{hypotheses}
We consider also the mesh quality parameter $\theta_{\hat{\disc}}$ for $\hat{\disc}$, defined as
\eqref{eq:quality_mesh} where $\hat{\mesh}$ and $\hat{\edges}_{\rm int}$ are
used instead of $\mesh$ and $\intEdges$, respectively. We assume that exists
$\theta \in \R^+$ such that
\begin{gather*}
    \theta \defeq \max \left\{ \sup \left\{ \theta_\disc,
    \disc \in \family \right\} , \sup \left\{ \theta_{\hat{\disc}}, \hat{\disc} \in
    \hat{\family} \right\} \right\}.
\end{gather*}
We introduce the following discrete spaces for both the domain and the fault
discretization: one degree of freedom for each element and one for each face,
namely for the porous domain
\begin{gather*}
    \X \defeq \left\{ v = \left( \left(v_K \right)_{K \in \mesh}, \left( v_\sigma
    \right)_{\sigma \in \edges}\right): v_K \in \R, v_\sigma \in \R \right\}
    \text{ and }
    \Xz \defeq \left\{ v \in \X: v_\sigma = 0 \forall \sigma \in
    \extEdges \right\},
\end{gather*}
and for the fault
\begin{gather*}
    \Xhat \defeq \left\{ \hat{v} = \left( \left( \hat{v}_K \right)_{K
    \in \hat{\mesh}}, \left( \hat{v}_\sigma
    \right)_{\sigma \in \hat{\edges}}\right): \hat{v}_K \in \R, \hat{v}_\sigma \in \R \right\}
    \text{ and }
    \Xhatz \defeq \left\{ \hat{v} \in \Xhat: \hat{v}_\sigma = 0
    \forall \sigma \in
    \hat{\edges}_{\rm ext} \right\}.
\end{gather*}
Where the spaces $\Xz \subset \X$ and $\Xhatz \subset \Xhat$ include the boundary
conditions. We consider also the global discrete space as $\Y \defeq
\X \times \Xhat$ and $\Yz \defeq \Xz \times \Xhatz$, with $\Yz \subset \Y$.
Since the discretization of the fault is
constructed from the discretization of the porous medium, for exigence in
notation we will indicate, in presence of both, only the latter. The spaces
$\Xz$ and $\Xhatz$ are the discrete approximations of $\V$
and $\Vhat$, respectively. For each of the
previous space we introduce a discrete semi-norm: given $v \in \X$ and
$\hat{v} \in \Xhat$, we define
\begin{gather} \label{eq:discrete_semi-norms}
    \abs{v}_{\X} \defeq \sum_{K \in \mesh} \sum_{\sigma
    \in \edges_K} \dfrac{\abs{\sigma}}{d_{K,\sigma}} \left( v_\sigma - v_K
    \right)^2,
    \quad
    \abs{\hat{v}}_{\Xhat} \defeq \sum_{K \in
    \hat{\mesh}} \sum_{\sigma\in \edges_K} \dfrac{\abs{\sigma}}{d_{K,\sigma}}
    \left( \hat{v}_\sigma - \hat{v}_K \right)^2
\end{gather}
and concerning the global space for all $\left( v, \hat{v} \right) \in \Y$ we have
$\abs{\left(v, \hat{v}\right)}_{\Y}^2 \defeq \abs{v}_{\X}^2 + \abs{\hat{v}}_{\Xhat}^2$.
Given a function $v \in \X$, let us set $\Pi_\mesh v \in L^2(\Omega)$ the
piece-wise function defined by $\Pi_\mesh v ( \vect{x} ) = v_K $ for a.e.
$\vect{x} \in K$, for all $K \in \mesh$. We indicate with $\Pi_{\hat{\mesh}}:
\Xhat \rightarrow L^2(\gamhatiii)$ the same projector
operator defined on the two layers of the fault.
Introducing $D_\sigma v \defeq \abs{v_K - v_L}$ and $d_\sigma \defeq d_{K,\sigma} +
d_{L,\sigma}$ for $\mesh_\sigma = \left\{ K, L
\right\}$, or $D_\sigma v \defeq \abs{v_K}$ and $d_\sigma \defeq d_{K,\sigma}$
for $\mesh_\sigma = \left\{ K\right\}$, for each function $w = \Pi_\mesh v$,
with $v \in \X$, and for each function $\hat{w}=\Pi_{\hat{\mesh}} \hat{v}$,
with $\hat{v} \in \Xhat$, we define the following discrete norms
\begin{gather} \label{eq:discrete_norms}
    \norm{w}_{1,\mesh} \defeq \sum_{\sigma \in \edges} \abs{\sigma}
    \dfrac{\left( D_\sigma w \right)^2}{d_\sigma}
    \quad \text{and} \quad
    \norm{\hat{w}}_{1,\hat{\mesh}} \defeq \sum_{\sigma \in \hat{\edges}} \abs{\sigma}
    \dfrac{\left( D_\sigma \hat{w} \right)^2}{d_\sigma}.
\end{gather}
It is easy to show that
\begin{gather*}
    \norm{\Pi_\mesh v}_{1,\mesh} \leq \abs{v}_{\X} \quad \forall v \in \Xz
    \quad \text{and} \quad
    \norm{\Pi_{\hat{\mesh}} \hat{v}}_{1,\hat{\mesh}} \leq \abs{\hat{v}}_{\Xhat} \quad
    \forall \hat{v} \in \Xhatz.
\end{gather*}
Finally we introduce the projection operators $P_\disc: C\br{\Omega_{1,2}}
\rightarrow \X$ and $\hat{P}_\disc: C\br{\gamhatiii} \rightarrow
\Xhat$, such that given $\phi \in C\br{\Omega_{1,2}}$ and $\hat{\phi}
\in C\br{\gamhatiii}$ we have
\begin{gather*}
    P_\disc \phi = \left( \left( \phi \br{\vect{x}_K} \right)_{K \in \mesh},
    \left( \phi \br{\vect{x}_\sigma} \right)_{\sigma
    \in \edges} \right)
    \quad \text{and} \quad
    \hat{P}_\disc \hat{\phi} = ( ( \hat{\phi} \br{\vect{x}_K}
    )_{K \in \hat{\mesh}},
    ( \hat{\phi} \br{\vect{x}_\sigma} )_{\sigma
    \in \hat{\edges}} ).
\end{gather*}

To solve numerically problem \eqref{eq:red_problem} we consider the hybrid
finite volume scheme introduced in \cite{Eymard2010,Droniou2010}. We have chosen
to approximate the pressure field with a scalar value for each cell $K$,
indicated with a sub-script $K$, and a scalar value for each edges, indicated
with a sub-script $\sigma$. The core of the scheme is the construction of
approximate gradient $\dNabla$ in each cell. First of all, considering the
porous media, we introduce the classical cell gradient, indicated with
$\nabla_K$, which is constant for each cell. Considering the function $v \in
\X$ we define
\begin{gather*}
    \nabla_K v \defeq \dfrac{1}{\abs{K}} \sum_{\sigma \in \edges_K} \abs{\sigma}
    \left( v_\sigma - v_K \right) \n_{K,\sigma}.
\end{gather*}
Furthermore we consider, for each cone $D_{K,\sigma} \subset K$, a stabilization term
\begin{gather*}
    R_{K,\sigma} v \defeq \dfrac{\alpha \sqrt{N}}{d_{K,\sigma}} \left[ v_\sigma
    - v_K - \nabla_K v \cdot \left( \vect{x}_\sigma - \vect{x}_K \right) \right],
\end{gather*}
where $\alpha \in \R^+$ is a stabilization parameter, in \cite{Eymard2010}
$\alpha=1$ while in \cite{Droniou2010} the stabilization parameter is a symmetric and
positive defined matrix. In our presentation we consider only a scalar
stabilization coefficient.
Finally the discrete gradient $\dNabla v$ for the cell $K$ is defined for each cone
$D_{K,\sigma}$ of $K$ as
\begin{gather*}
    \left. \dNabla v \right|_{D_{K,\sigma}} \defeq \nabla_K v + R_{K,\sigma} v
    \n_{K,\sigma}.
\end{gather*}
The approximation of the averages and jumps operators involves only the unknowns
defined on the faces of the cells, so their computation is straightforward.  We
still consider the same scheme for the approximation of the fault differential
operators. In this case $\nabla_{\ttau}$ is approximated by $\TT \dNabla$, \ie
given $D_{K,\sigma} \subset K \in \hat{\mesh}$ and $\hat{v} \in \Xhat$ then
$\left. \TT \dNabla \hat{v} \right|_{D_{K,\sigma}}
\defeq \TT \nabla_K \hat{v} + R_{K,\sigma} \hat{v} \n_{K,\sigma}$ with
\begin{gather*}
    R_{K,\sigma} \hat{v} = \dfrac{\hat{\alpha} \sqrt{N-1}}{d_{K,\sigma}} \left[
    \hat{v}_\sigma - \hat{v}_K
    - \TT \nabla_K \hat{v} \cdot \left( \vect{x}_\sigma - \vect{x}_K \right) \right],
\end{gather*}
with $\hat{\alpha} \in \R^+$ the stabilization parameter for the fault
discretization.  The discrete problem require to introduce a new bilinear form
for the differential discrete operators, namely
\begin{gather*}
    a_{\disc,\Omega} \left(p,v\right) \defeq \sum_{j} \left( \Lambda_j \dNabla p_j,
    \dNabla v_j \right)_{\Omega_j }
    \quad \text{and} \quad
    a_{\disc,\gamma} \left(\hat{p},\hat{v}\right) \defeq  \sum_{j} \left( \hat{\lambda}
    \TT \dNabla \phatj, \TT \dNabla \hat{v}_j \right)_{\gamhatj},\\
    a_{\disc} \left( \left( p, \hat{p} \right), \left( v, \hat{v} \right)
    \right) \defeq
        a_{\disc,\Omega} \left(p,v\right)+ a_{\disc,\gamma} \left(\hat{p},\hat{v}\right).
\end{gather*}
for $\left(p, \hat{p}\right) \in \Yz$ and $\left( v, \hat{v} \right) \in
\Yz$. The global discrete bilinear form is defined as
\begin{gather*}
    \bFormDisc{\left( p, \hat{p} \right)}{\left( v, \hat{v}\right)}
    \defeq a_{\disc} \left( \left( p, \hat{p} \right), \left( v, \hat{v}
    \right) \right) + cc \left( \left( p, \hat{p} \right), \left( v, \hat{v}
    \right) \right),
\end{gather*}
The weak formulation for the discrete problem \eqref{eq:red_problem}:
find $\left(p, \hat{p} \right) \in \Yz$ such that
\begin{gather} \label{eq:weak_disc_form}
    \bFormDisc{\left( p, \hat{p} \right)}{\left( v, \hat{v}\right)}
    = F\left(v, \hat{v} \right) \qquad \forall \left( v,
    \hat{v} \right) \in \Y.
\end{gather}
Following \cite{Eymard2010}, we introduce some useful results to prove the
convergence of the numerical scheme to the exact solution. Lemma
\ref{lemma:norm_equivalence} shows the equivalence of the semi-norm
\eqref{eq:discrete_semi-norms} to the $L^2$-norm of the discrete tangential
gradient, while Lemma \ref{lemma:weak_discrete_compactness} guarantees the weak
compactness of $\Vhat$ in the discrete topology. Let us start with the norm
equivalence.


\begin{lemma}[Norm equivalence] \label{lemma:norm_equivalence}
    Given $\hat{v} \in \Xhat$ then $\abs{\hat{v}}_{\Xhat} \lesssim
    \norm{\TT \dNabla \hat{v}}_{L^2(\gamhatiii)} \lesssim
    \abs{\hat{v}}_{\Xhat}$.
\end{lemma}
\begin{proof}
    Considering that $(a-b)^2\geq \lambda/(1+\lambda) a^2 - \lambda b^2$, for
    $a,b \in \R$ and $\lambda > -1$, we have
    \begin{gather*}
        \norm{\TT \dNabla \hat{v}}^2_{L^2(\gamhatiii)} = \sum_{K \in
        \hat{\mesh}} \abs{K} \abs{ \TT \nabla_K \hat{v} }^2 + \sum_{\sigma \in
        \edges_K} \dfrac{\abs{\sigma} d_{K,\sigma}}{N-1} \br{R_{K,\sigma}
        \hat{v}}^2 \geq \\
        \geq \sum_{K \in \hat{\mesh}} \abs{K} \left[ 1 - \hat{\alpha}^2 \lambda \br{N-1}
        \theta_{\hat{\disc}}^2 \right] \abs{\TT \nabla_K \hat{v}}^2 +
        \dfrac{\lambda \hat{\alpha}^2}{1+\lambda}
        \sum_{\sigma \in \edges_K} \dfrac{\abs{\sigma}}{d_{K,\sigma}} \br{
        \hat{v}_\sigma - \hat{v}_K}^2
    \end{gather*}
    where we have considered \eqref{eq:quality_mesh} for $\theta_{\hat{\disc}}$.
    Choosing the parameter $\lambda^{-1} = \hat{\alpha}^2\br{N-1} \theta_{\hat{\disc}}^2$ we
    obtain $\norm{\TT \dNabla \hat{v}}_{L^2(\gamhatiii)} \gtrsim
    \abs{\hat{v}}_{\Xhat}$. Moreover, given $K \in \hat{\mesh}$, we have
    \begin{gather*}
        \abs{\TT \nabla_K \hat{v}}^2 \leq \dfrac{1}{\abs{K}^2} \sum_{\sigma \in
        \edges_K} \dfrac{\abs{\sigma}}{d_{K,\sigma}} \abs{\hat{v}_\sigma -
        \hat{v}_K}^2 \sum_{\sigma \in \edges_K} \abs{\sigma} d_{K,\sigma}
        \abs{\TT \n_{K,\sigma}}^2 = \dfrac{N-1}{\abs{K}} \sum_{\sigma \in
        \edges_K} \dfrac{\abs{\sigma}}{d_{K,\sigma}} \abs{\hat{v}_\sigma -
        \hat{v}_K}^2,
    \end{gather*}
    while the stabilization term is
    \begin{gather*}
        \abs{R_{K,\sigma} \hat{v}}^2 \leq \hat{\alpha}^2 \br{N-1} \br{
        \dfrac{\abs{\hat{v}_\sigma -\hat{v}_K}^2}{d_{K,\sigma}^2} +
        \dfrac{\abs{\TT \nabla_K \hat{v}}^2}{d_{K,\sigma}^2}
        \abs{\vect{x}_\sigma - \vect{x}_K}^2} \leq \\
        \leq \hat{\alpha}^2 \br{N-1} \br{
        \dfrac{\abs{\hat{v}_\sigma -\hat{v}_K}^2}{d_{K,\sigma}^2} +
        \abs{\TT \nabla_K \hat{v}}^2
        \theta_{\hat{\disc}}^2 },
    \end{gather*}
    obtaining the other inequality $\norm{\TT \dNabla
    \hat{v}}^2_{L^2(\gamhatiii)} \lesssim \abs{\hat{v}}_{\Xhat}$.
\end{proof}


We show now the goodness of the proposed discrete tangential gradient, which
weakly converge to the continuous tangential gradient in the discrete topology.


\begin{lemma}[Weak discrete $\Vhat$ compactness]\label{lemma:weak_discrete_compactness}
    We consider the family of functions $\br{ \hat{v}_{\hat{\disc}}
    }_{\hat{\disc} \in \hat{\family}}$ and we suppose that:
    $\hat{v}_{\hat{\disc}} \in \Xhatz$,
    $\abs{\hat{v}_{\hat{\disc}}}_{\Xhat} \lesssim 1$ and exists a function
    $\hat{v} \in
    L^2(\gamhatiii)$ such that $\Pi_{\hat{\mesh}} \hat{v}_{\hat{\disc}}
    \rightarrow
    \hat{v} $ in $L^2(\gamhatiii)$ as $h_\disc \rightarrow 0$.
    Then $\hat{v} \in \Vhat$ and $\TT \dNabla \hat{v}_{\hat{\disc}}
    \rightharpoonup \nabla_{\ttau} \hat{v}$ in $L^2(\gamhatiii)$ as
    $h_\disc \rightarrow 0$.
\end{lemma}
\begin{proof}
    Since we are dealing with surface problems, we prolong $\Pi_{\hat{\mesh}}
    \hat{v}_{\hat{\disc}}$ and $\TT \dNabla \hat{v}_{\hat{\disc}}$ by 0 in
    $\R^N$ outside of $\gamhatiii$. Thanks to the boundedness of $\br{ \TT \dNabla
    \hat{p}_{\hat{\disc}} }_{\hat{\disc} \in \hat{\family}}$ then, since
    $L^2(\R^N)$ is a reflexive Banach space, applying the
    Banach-Alaoglu theorem there exists a sub-sequence, still denoted by $\br{
    \TT \dNabla \hat{p}_{\hat{\disc}} }_{\hat{\disc} \in\hat{\family}}$, which
    weakly converge to a $\vect{G} \in \left[L^2(\R^N)\right]^N$. We have to
    show that $\left. \vect{G} \right|_{\gamhatiii} = \nabla_{\ttau} \hat{v}$.
    Let us set, with $\vect{\psi} \in \R^N$, the following
    $\br{\TT \dNabla \hat{v}_{\hat{\disc}}, \vect{\psi} }_{\R^N} = T_2 + T_3$
    with
    \begin{gather*}
        T_2 = \sum_{K \in \hat{\mesh}} \br{\TT \nabla_K \hat{v}_{\hat{\disc}},
        \vect{\psi}}_K
        \quad \text{and} \quad
        T_3 = \sum_{K \in \hat{\mesh}} \sum_{\sigma \in \edges_K}
        \br{R_{K,\sigma} \hat{v}_{\hat{\sigma}},
        \vect{\psi} \cdot \n_{K,\sigma}}_{D_{K,\sigma}}.
    \end{gather*}
    We define $\vect{\psi}_K = \int_K \vect{\psi}/\abs{K}$ and
    $\vect{\psi}_\sigma = \int_\sigma \vect{\psi}/\abs{\sigma}$, then we have
    \begin{gather*}
        T_2 = \sum_{K \in \hat{\mesh}} \dfrac{1}{\abs{K}} \sum_{\sigma \in
        \edges_K} \left( \hat{v}_\sigma - \hat{v}_K \right) \TT \n_{K,\sigma}
        \cdot \int_{D_{K,\sigma}} \vect{\psi} = \sum_{K \in \hat{\mesh}}
        \vect{\psi}_K \cdot \br{\sum_{\sigma \in \edges_K}
        \br{ \hat{v}_\sigma - \hat{v}_K } \n_{K,\sigma}},
    \end{gather*}
    since $\TT \n_{K,\sigma} = \n_{K,\sigma}$. We consider also the following
    term and using Hypotheses \ref{hyp:normal}
    \begin{gather*}
        T_4 = - \br{\hat{v}_{\hat{\disc}}, \nabla_{\ttau} \cdot
        \vect{\psi}}_{\R^N} = - \sum_{K \in \hat{\mesh}} \sum_{\sigma \in
        \edges_K} \br{ \hat{v}_K, \vect{\psi} \cdot \n_{K,\sigma}
        }_{\sigma} = - \sum_{K \in \hat{\mesh}} \sum_{\sigma \in\edges_K}
        \abs{\sigma} \hat{v}_K  \vect{\psi}_\sigma \cdot \n_{K,\sigma}
        \\
        = \sum_{K \in \hat{\mesh}} \sum_{\sigma \in
        \edges_K} \abs{\sigma} \br{\hat{v}_\sigma - \hat{v}_K}
        \vect{\psi}_\sigma \cdot \n_{K,\sigma}+ \Order{h_\disc} = T_5 +
        \Order{h_\disc}.
    \end{gather*}
    We show now that $T_2 = T_4$ for $h_\disc \rightarrow 0$, in fact we have
    \begin{gather*}
        \br{T_2 - T_5}^2 = \br{ \sum_{K \in \hat{\mesh}} \sum_{\sigma \in
        \edges_K} \abs{\sigma}
        \br{\hat{v}_\sigma - \hat{v}_K} \br{\vect{\psi}_K - \vect{\psi}_\sigma}
        \cdot \n_{K,\sigma}}^2 \leq \\ \leq
        \sum_{K \in \hat{\mesh}} \sum_{\sigma \in
        \edges_K} \dfrac{\abs{\sigma}}{d_{K,\sigma}} \br{\hat{v}_\sigma -
        \hat{v}_K}^2 \sum_{K \in \hat{\mesh}} \sum_{\sigma \in\edges_K}
        \abs{\sigma} d_{K,\sigma} \norm{\vect{\psi}_K
        -\vect{\psi}_\sigma}^2_{\R^N} \lesssim \abs{\hat{v}_{\hat{\disc}}}_{\Xhat}^2
        \Order{h_\disc^2},
    \end{gather*}
    thanks to the uniform boundedness of the semi-norm of
    $\hat{v}_{\hat{\disc}}$ we have the convergence as $h_\disc \rightarrow 0$.
    The last step is to show that the stabilization term vanishes as $h_\disc
    \rightarrow 0$, in fact we have
    \begin{gather*}
        T_3 = \sum_{K \in \hat{\mesh}} \sum_{\sigma \in\edges_K} R_{K,\sigma}
        \hat{v}_{\hat{\disc}} \int_{D_{K,\sigma}} \vect{\psi} \cdot
        \n_{K,\sigma} = \\ =\sum_{K \in \hat{\mesh}} \sum_{\sigma \in\edges_K}
        R_{K,\sigma} \hat{v}_{\hat{\disc}} \br{ \int_{D_{K,\sigma}} \vect{\psi}
        -  \dfrac{d_{K,\sigma} \abs{\sigma}}{\br{N-1}\abs{D_{K,\sigma}}}
        \int_{D_{K,\sigma}} \vect{\psi}_K} \cdot \n_{K,\sigma} = \\ =\sum_{K \in
        \hat{\mesh}} \sum_{\sigma \in\edges_K} R_{K,\sigma}
        \hat{v}_{\hat{\disc}} \int_{D_{K,\sigma}} \br{\vect{\psi} -
        \vect{\psi}_K} \cdot \n_{K,\sigma},
    \end{gather*}
    finally considering the square of $T_3$ and the mean value theorem
    we can end up with the proof, \ie
    \begin{gather*}
        \br{T_3}^2 \leq \sum_{K \in \hat{\mesh}} \sum_{\sigma \in\edges_K}
        d_{K,\sigma} \dfrac{\abs{\sigma}}{N-1}
        \br{R_{K,\sigma} \hat{v}_{\hat{\disc}}}^2
        \sum_{K \in \hat{\mesh}} \sum_{\sigma \in\edges_K}
        \dfrac{N-1}{d_{K,\sigma} \abs{\sigma}}
        \br{ \int_{D_{K,\sigma}}
        \br{\vect{\psi} -\vect{\psi}_K} \cdot \n_{K,\sigma}}^2 \leq \\ \lesssim
        \abs{\hat{v}_{\hat{\disc}}}_{\Xhat}^2 \sum_{K \in \hat{\mesh}}
        \sum_{\sigma \in\edges_K} \dfrac{N-1}{d_{K,\sigma} \abs{\sigma}}
        \int_{D_{K,\sigma}} \norm{\vect{\psi} -\vect{\psi}_K}^2_{\R^N} \lesssim
        \Order{h_\disc^2}.
    \end{gather*}
\end{proof}


We prove the consistency of the discrete tangential gradient, \ie the maximum
error between the latter and the tangential gradient vanishes as $h_\disc
\rightarrow 0$.


\begin{lemma}[Discrete tangential gradient consistency]\label{lemma:tangential_gradient_consistency}
    Given a function $\phi \in C^2\br{\gamhatiii}$ then
    \begin{gather*}
        \norm{ \TT \dNabla P_{\hat{\disc}} \phi -
        \nabla_{\ttau} \phi }_{L^\infty(\gamhatiii)} \leq \Order{h_\disc}.
    \end{gather*}
\end{lemma}
\begin{proof}
    For each cell $K \in \hat{\mesh}$ we have
    \begin{gather*}
        \TT \nabla_K P_{\hat{\disc}} \hat{\phi} = \dfrac{1}{\abs{K}} \sum_{\sigma \in
        \edges_K} \abs{\sigma} \br{ \nabla_{\ttau} \hat{\phi} \br{\vect{x}_K} \cdot
        \br{\vect{x}_\sigma - \vect{x}_K} + \Order{h_K^2} } \n_{K,\sigma},
    \end{gather*}
    then $\abs{ \TT \nabla_K \hat{P}_{\disc} \phi - \nabla_{\ttau} \phi
    \br{\vect{x}_K} } \leq \Order{h_K}$. Moreover the stabilization term goes
    similarly
    \begin{gather*}
        \abs{R_{K,\sigma} P_{\hat{\disc}} \hat{\phi}} =
        \dfrac{\sqrt{N-1}}{d_{K,\sigma}} \abs{\phi\br{\vect{x}_\sigma} -
        \phi\br{\vect{x}_K} - \TT \nabla_K \hat{P}_{\disc} \hat{\phi} \cdot
        \br{\vect{x}_\sigma - \vect{x}_K}} \leq
        \dfrac{\Order{h_K^2}}{d_{K,\sigma}} \leq \Order{h_K}.
    \end{gather*}
\end{proof}


We are ready to introduce the main result of this section, which shows the
correctness of the chosen approximation: convergence of both discrete porous medium
and fault pressures to the exact solution of continuous problem
\eqref{eq:weak_form}. A priori bound is given obtaining the well posedness of the
discrete problem \eqref{eq:weak_disc_form}.


\begin{lemma}[Convergence for matching grids] \label{lemma:convergence_matching}
    We suppose that the discretization of the two layers of the fault is
    matching.
    Let us consider the family of functions $\left( p_\disc, \hat{p}_{\hat{\disc}}
    \right) \in \Yz$, with $\disc \in \family$ and $\hat{\disc} \in \hat{\family}$,
    satisfies \eqref{eq:weak_disc_form} for each choice of
    discretization, then
    \begin{gather*}
        \lim_{h_\disc \rightarrow 0} \norm{\Pi_\mesh p_\disc - p
        }_{L^2(\Omega_{1,2})} = 0
        \quad \text{ and } \quad
        \lim_{h_\disc \rightarrow 0} \norm{\Pi_{\hat{\mesh}}
        \hat{p}_{\hat{\disc}} - \hat{p} }_{L^2(\gamhatiii)} =0,
    \end{gather*}
    where $(p,\hat{p}) \in \V_0 \times \Vhat_0$ is the unique solution of \eqref{eq:weak_form}.
    Moreover we have $\abs{\left(p_\disc, \hat{p}_{\hat{\disc}}
    \right)}_{\Y} \lesssim 1$.
\end{lemma}
\begin{proof}
    Given a discretization of both the domain and the fault $\disc \in \family$
    and $\hat{\disc} \in \hat{\family}$,
    let us use the following functions $(v, \hat{v}) \in \Yz$.
    Considering the continuity of the functional in \eqref{eq:weak_disc_form}
    with constant $c_F \in \R^+$, we have
    \begin{gather*}
        \abs{\bFormDisc{\left( v, \hat{v} \right)}{\left(
        v, \hat{v} \right)}}
        = \abs{F\left( \Pi_\mesh v, \Pi_{\hat{\mesh}}
        \hat{v} \right)} \leq c_F
        \sum_{j} \norm{ \Pi_\mesh v_{j} }_{L^2(\Omega_j)} + c_F
        \norm{\Pi_{\hat{\mesh}} \hat{v}_{j} }_{L^2(\gamhatj)},
    \end{gather*}
    thanks to Lemma 5.3 of \cite{Eymard2010} we can bound the $L^2$-norms by the
    norms defined in \eqref{eq:discrete_norms}, obtaining
    \begin{gather} \label{eq:convergence_1}
        \abs{\bFormDisc{\left( v, \hat{v} \right)}{\left( v, \hat{v}\right)}}
        \lesssim \norm{\Pi_\mesh v}_{1,\mesh} + \norm{\Pi_{\hat{\mesh}}
        \hat{v}}_{1,\hat{\mesh}} \lesssim \abs{ \left( v, \hat{v} \right)
        }_{\Y}.
    \end{gather}
    We derive now a lower bound for the bilinear form $\bFormDisc{\cdot}{ \cdot
    }$, using the semi-norm \eqref{eq:discrete_semi-norms}. We start from
    \begin{gather*}
        \bFormDisc{\left( v, \hat{v} \right)}{\left( v, \hat{v}\right)}
        \geq a_{\disc} \left( \left( v, \hat{v} \right), \left( v, \hat{v}
        \right) \right) \geq c_{\Lambda} \sum_{j} \norm{\dNabla
        v_j}^2_{L^2(\Omega_j)} +
        c_{\hat{\lambda}} \norm{\TT \dNabla
        \hat{v}_j}^2_{L^2(\gamhatj)},
    \end{gather*}
    where $c_{\Lambda}, c_{\hat{\lambda}} \in \R^+$ is the minimum eigenvalue of
    $\Lambda$ and the minimum of $\hat{\lambda}$, respectively.  Thanks to Lemma
    \ref{lemma:norm_equivalence} and Lemma 4.1 of \cite{Eymard2010}, we end up
    with $\bFormDisc{\br{ v, \hat{v} }}{\br{v, \hat{v}}} \gtrsim \abs{\br{v,
    \hat{v}}}_{\Y}^2$.  Considering the latter inequality and
    \eqref{eq:convergence_1} we have an a-priori bound on the discrete solution
    of \eqref{eq:weak_disc_form}: $\abs{\left(p_\disc, \hat{p}_{\hat{\disc}}
    \right)}_{\Y} \lesssim 1$, with a constant independent on the chosen
    discretization. Using Lemma 5.7 of \cite{Eymard2010} we can extract a
    sub-sequence from $\left( p_\disc, \hat{p}_{\hat{\disc}}\right)$, still
    denoted by $\left( p_\disc,\hat{p}_{\hat{\disc}}\right)$, and $\left( p^*,
    \hat{p}^*\right) \in \V \times \Vhat$ such that
    \begin{gather} \label{eq:convergence_2}
        \lim_{h_\disc \rightarrow 0} \norm{ p_\disc - p^* }_{L^2(\Omega_{1,2})}
        = 0
        \quad \text{ and } \quad
        \lim_{h_\disc \rightarrow 0} \norm{
        \hat{p}_{\hat{\disc}} - \hat{p}^* }_{L^2(\gamhatiii)} =0.
    \end{gather}
    The result is proved if we show that $\left( p^*, \hat{p}^*\right)$ is the unique solution of
    \eqref{eq:weak_form}. We start considering as test functions $\left(v,
    \hat{v} \right) = \left( P_\disc \phi, P_{\hat{\disc}} \hat{\phi} \right)$,
    with $\phi \in C^\infty_c(\Omega_{1,2})$ and $\hat{\phi} \in
    C^\infty_c(\gamhatiii)$, then
    \begin{gather*}
        \lim_{h_\disc \rightarrow 0} \bFormDisc{\left( p_\disc,
        \hat{p}_{\hat{\disc}} \right)}{\left( v, \hat{v}
        \right)} = \lim_{h_\disc \rightarrow 0} a_{\disc}
        \left( \left( p_\disc, \hat{p}_{\hat{\disc}} \right), \left(
        v, \hat{v} \right) \right) + \lim_{h_\disc
        \rightarrow 0} cc \left( \left( p_\disc, \hat{p}_{\hat{\disc}} \right),
        \left( v,  \hat{v} \right) \right),
    \end{gather*}
    the second term, since involve only algebraic conditions on the fault,
    converge to the bilinear form $cc \left( \left( p^*, \hat{p}^* \right), \left(
    v, \hat{v} \right) \right)$ as $h_\disc \rightarrow 0$ thanks to
    \eqref{eq:convergence_2} for the fault unknowns and thanks to
    \cite{Eymard2010}
    for the porous medium unknowns. For the $a_\disc$ bilinear form, thanks to Lemma
    4.4 of \cite{Eymard2010}, we have that $a_{\disc,\Omega} \br{ p_\disc, v}$
    converge to $a_\Omega \br{ p^*, v }$ as $h_\disc \rightarrow 0$.
    Considering Lemma \ref{lemma:weak_discrete_compactness} and Lemma
    \ref{lemma:tangential_gradient_consistency} we have the convergence of
    the discrete
    bilinear form $a_{\gamma,\disc} \br{\hat{p}_{\hat{\disc}}, \hat{v}}$ to
    $a_\gamma \br{\hat{p}^*, \hat{v}}$ as
    $h_\disc \rightarrow 0$. Since \eqref{eq:weak_form}
    is well posed then $p^* = p$ and $\hat{p}^* = \hat{p}$.
\end{proof}

For the implementation prospective we consider the virtual cell approach
presented and discussed in \cite{Faille2014a}. The fault cells, in the
co-dimensional domain, are (virtually) extruded in the normal direction by their
thickness and the normal hybrid finite volume scheme is employed. As proved in the
aforementioned work, this approach gives an equivalent scheme, with respect to
the discretization of the co-dimensional object, for matching
grids.
For non-matching grids, the virtual cell approach is preferred as it does not
boil down to a two point flux approximation across the non-matching faces.
In the subsequent examples we therefore employ such an approximation for
both matching and non-matching grids.




\section{Examples} \label{sec:examples}


\def\pathImages{./Parts/Images}


In this section we present some numerical results to illustrate and assess the
properties of the DLRM. Even if Lemma \ref{lemma:convergence_matching} ensure
the convergence of the scheme for matching grids we consider different tests for
both matching and non-matching grids to obtain a numerical evidence of the
convergence.  Since it is quite difficult to exhibit an exact solution for some
realistic test case in Example \ref{subsec:partially_impermeable_fault} and
\ref{subsec:conductive_fault}, which are inspired by \cite{Martin2005}, we
propose two different problems. They represent a fault immersed in a domain with
normal permeability smaller than the one of the surrounding rock matrix in order
to obtain a
pressure jump, and a purely conductive fault. In both cases we consider a
reference solution given by a very fine grid. Finally in Example
\ref{subsec:anisotropic_fault} we analyse the effect of the mesh size difference
between the two sub-domains $\Omega_j$ and consequently on the
fault $\gamhatj$. For each test we use a direct method to solve the linear
system. The last test in Example \ref{subsec:slipping_domain} consider a more
realistic simulation with a slipping domain. The code is developed in the Arcane
framework \cite{Grospellier2009}.

To evaluate numerically the order of the error we consider as a reference
solution an approximate solution computed by an extremely refined
Cartesian mesh.  The error for the
porous medium is defined in the following way
\begin{gather*}
    \norm{p_\disc - p_{\rm ref}}_{L^2}^2 \defeq \sum_{K \in \mesh} \abs{K}
    \left( p_\disc|_K - \Pi p_{\rm ref} |_K \right)^2,
\end{gather*}
where $\mesh$ is the coarse mesh and $\Pi$ is an interpolation operator between
the fine mesh and the coarse mesh. Moreover the error for the two layers of the
fault is defined as
\begin{gather*}
    \norm{\hat{p}_{\hat{\disc}} - \hat{p}_{\rm ref}}_{L^2}^2 \defeq \sum_{j} \sum_{K \in
    \hat{\mesh}_j} \abs{K} \left( \hat{p}_{j,\hat{\disc}}|_{K} - \hat{\Pi}_j
    \hat{p}_{j, \rm ref} |_{K}
    \right)^2,
\end{gather*}
where $\hat{\mesh}_j$ is the coarse mesh for the layer $j$ of the fault and
$\hat{\Pi}_j$ is an interpolation operator between the fine mesh for the layer
$j$ of the fault and the coarse mesh $\hat{\mesh}_j$.


\subsection{Partially impermeable fault} \label{subsec:partially_impermeable_fault}


We consider the domain $\Omega = (0,1)^2$ with a vertical fault, with thickness
$d=10^{-2}$, in the centre of the domain, see Figure
\ref{fig:reducedModel_test2_domain} for a sketch of the computational domain.
\begin{figure}[!htp]
    \centering
    \subfloat[Domain.]
    {
        \includegraphics{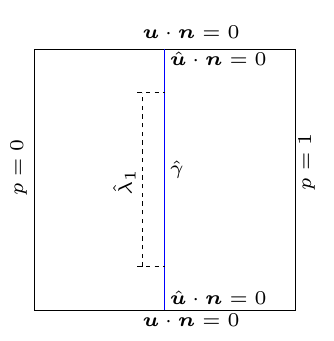}
        \label{fig:reducedModel_test2_domain}
    }
    \hspace{0.1\textwidth}%
    \subfloat[Mesh.]
    {
        \raisebox{0.05\height}{\includegraphics[width=0.2\textwidth]{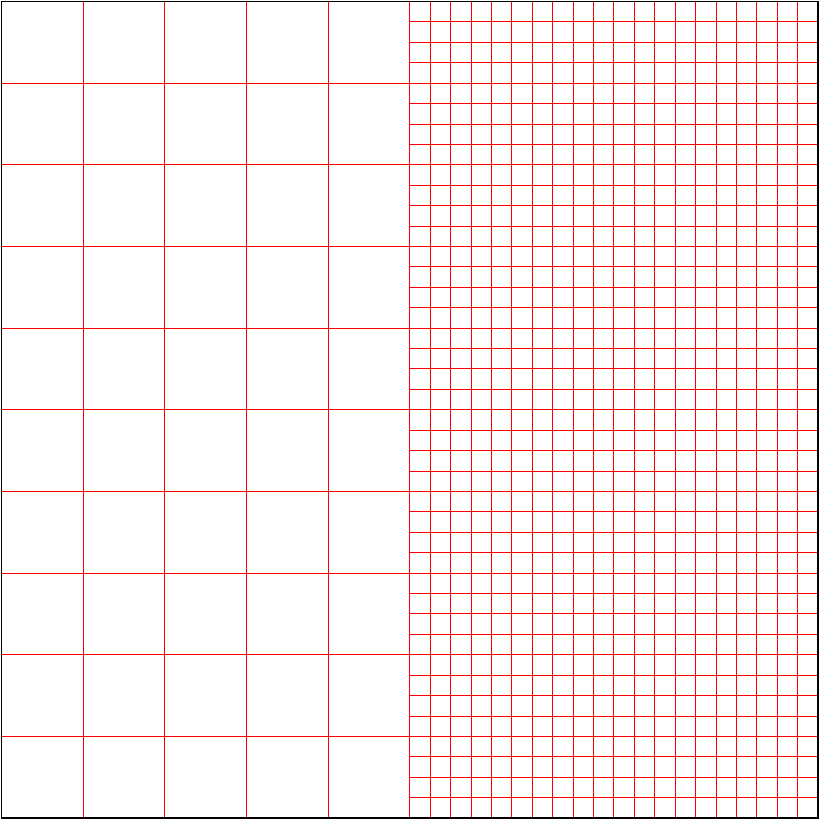}}%
        \label{fig:reducedModel_test2_mesh}
    }
    \caption{Computational domain for tests in the subsection \ref{subsec:partially_impermeable_fault}
        with the boundary conditions and an example of a mesh used.}%
\end{figure}
We assume homogeneous Neumann boundary conditions on the top and bottom of the
domain and the fault. Homogeneous Dirichlet boundary condition at left and
Dirichlet boundary condition $p=1$ in the right part of the domain. We
consider identity matrix as permeability in the domain. In the fault we consider
a discontinuous tangential permeability, with value $\hat{\lambda}_1(s) =
10^{-2}$ for $s \in (0.25,0.75)$ and $\hat{\lambda}_2 = 1$ in the rest of the
fault. The computational mesh is composed by quadrangular elements, non-matching
at the fault.
\begin{figure}[!htp]
    \centering
    \subfloat[Solution.]
    {
        \includegraphics[width=0.33\textwidth]{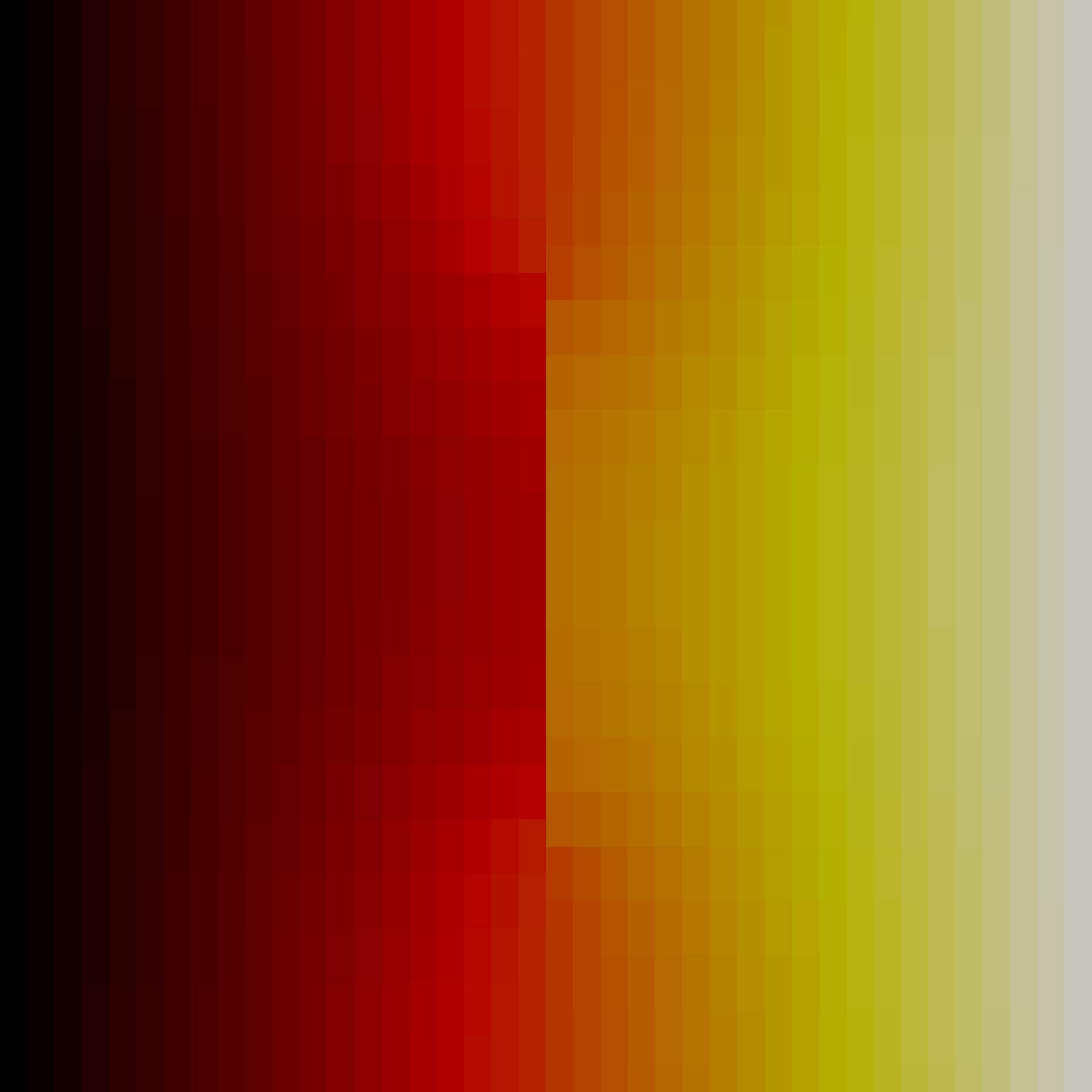}
    }%
    \hspace{0.05\textwidth}%
    \subfloat[Warped solution.]
    {
        \includegraphics[width=0.33\textwidth]{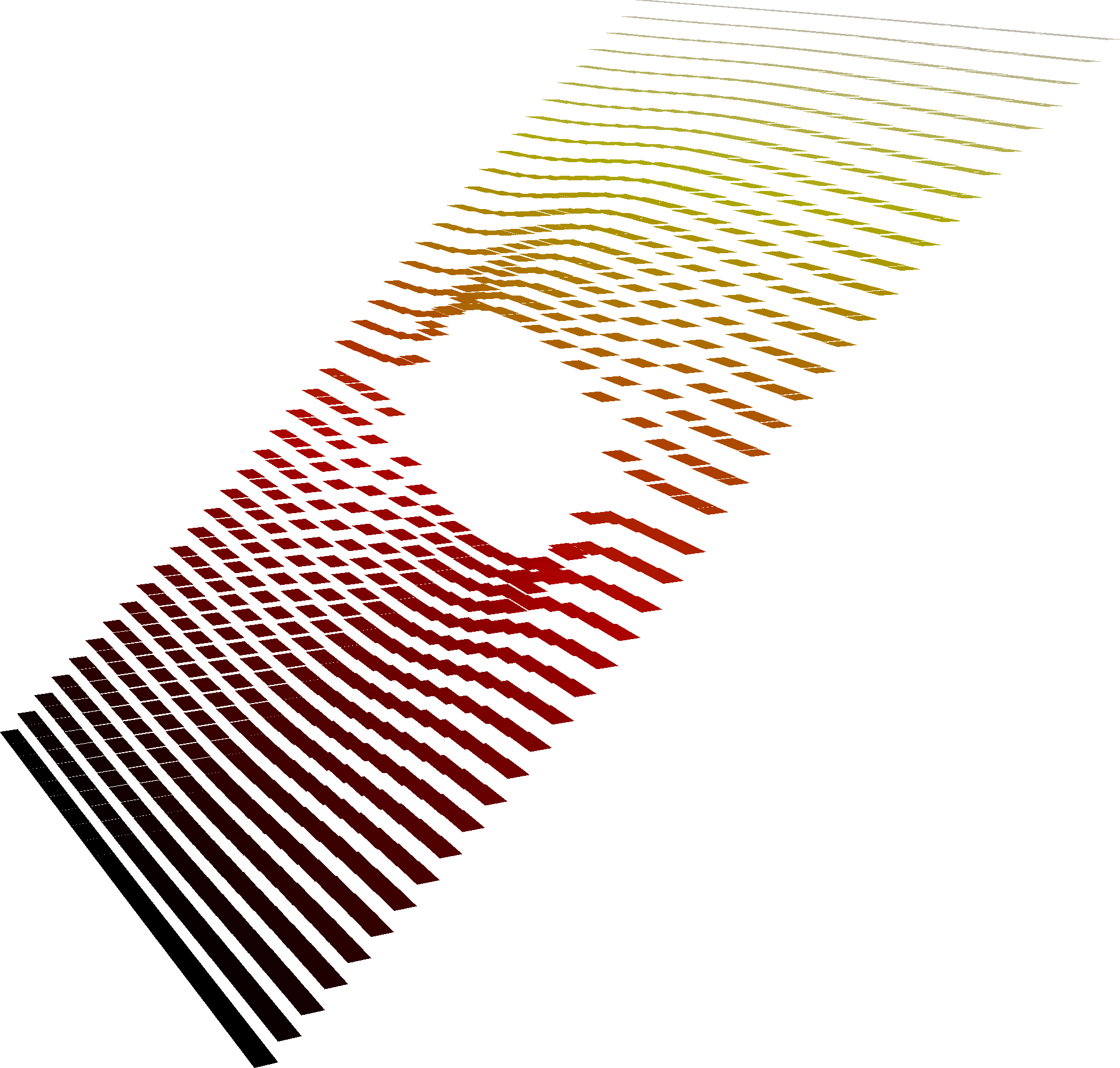}
    }%
    \hspace{0.05\textwidth}%
    \includegraphics{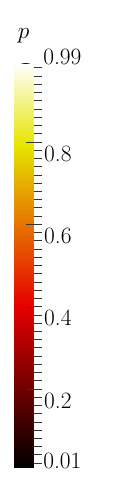}
    \caption{Pressure field for the partially permeable case.}%
    \label{fig:reducedModel_test2}%
\end{figure}
The solution is reported in Figure \ref{fig:reducedModel_test2}. We
can notice that the solution across the fault exhibit a jump where the fault
has a low permeability. The maximum and minimum discrete principle, in this
particular case, are fulfilled.

We evaluate the error decay considering a reference grid of approximatively two
millions of elements. Following
\cite{Frih2011} the analytical solution exhibit a singularity at $(0.5,0.25)$
and $(0.5,0.75)$, to focus our attention only on the dependence the regularity
of the solution on the error order, we consider a family of Cartesian meshes.
The error history is presented in Figure
\ref{fig:error_dacay_partially_impermeable_fault_conforming}, which shows a
pressure error for the both the sides of the fault is close to
$\Order{h_\disc^2}$.  Moreover the pressure error for the porous medium is close
to $\Order{h_\disc^{\frac{3}{2}}}$, confirming the dependence of the error order
to the regularity of the exact solution. Figure
\ref{fig:reducedModel_test2_error_1} shows the error of a particular mesh,
highlight the two peaks of error close to the singularities.
\begin{figure}[!htp]
    \centering
    \subfloat[Error history.]
    {
        \includegraphics{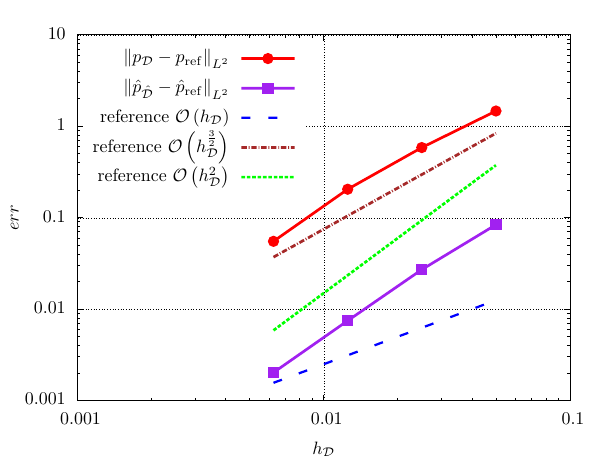}
        \label{fig:error_dacay_partially_impermeable_fault_conforming}
    }%
    \hspace{0.025\textwidth}%
    \subfloat[Error.]
    {
        \includegraphics[width=0.33\textwidth]{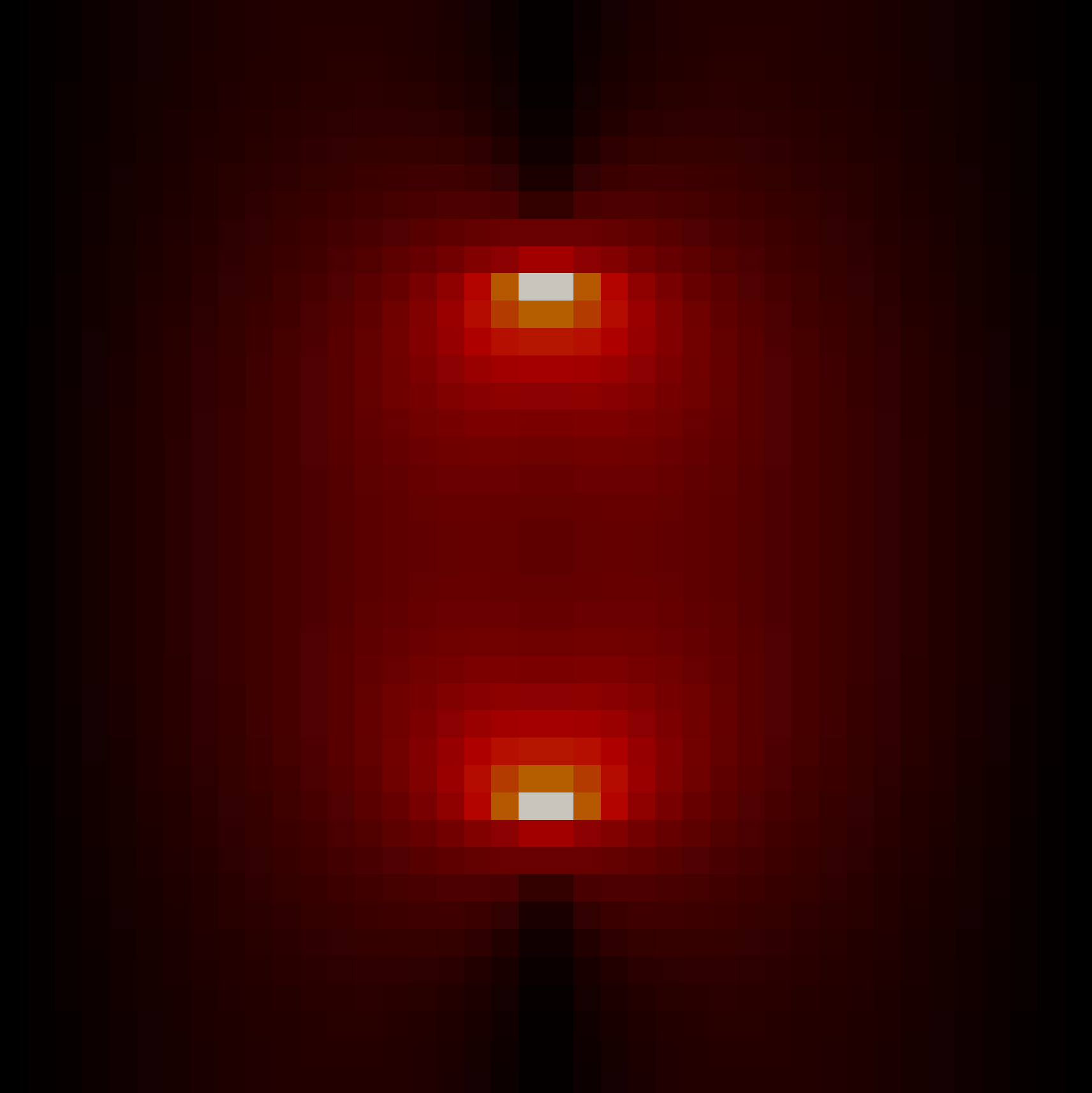}%
        \label{fig:reducedModel_test2_error_1}
    }%
    \hspace{0.025\textwidth}%
    \includegraphics{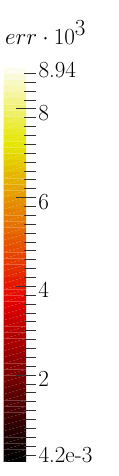}
    \caption{Error history for both the porous medium and the fault and a
        representation of the error for a particular choice of the mesh, for
        the matching case. In dashed lines are represented also some reference curves.}%
\end{figure}
We consider also a different family of meshes for the error analyses, a coarser
example is represented in Figure \ref{fig:reducedModel_test2_mesh}. Each
elements in the left part of the domain is constructed with 16 of small elements
used for the right part. Even if the error is bigger then the previous case,
both the pressure errors are close to $\Order{h_\disc^{\frac{3}{2}}}$.
\begin{figure}[!htp]
    \centering
    \subfloat[Error history.]
    {
    \includegraphics{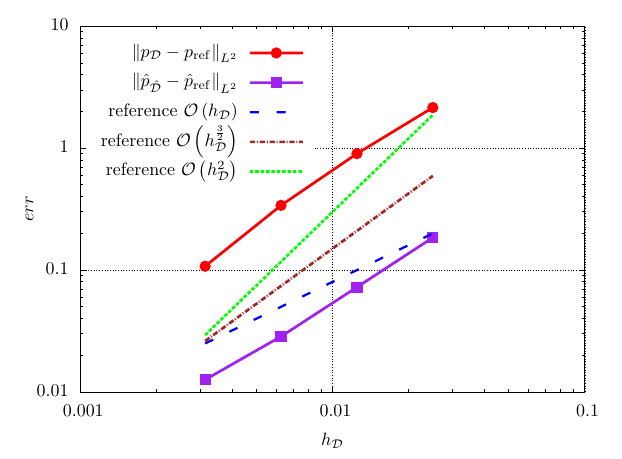}
        \label{fig:error_dacay_partially_impermeable_fault_non-matching}
    }%
    \hspace{0.025\textwidth}%
    \subfloat[Error.]
    {
        \includegraphics[width=0.33\textwidth]{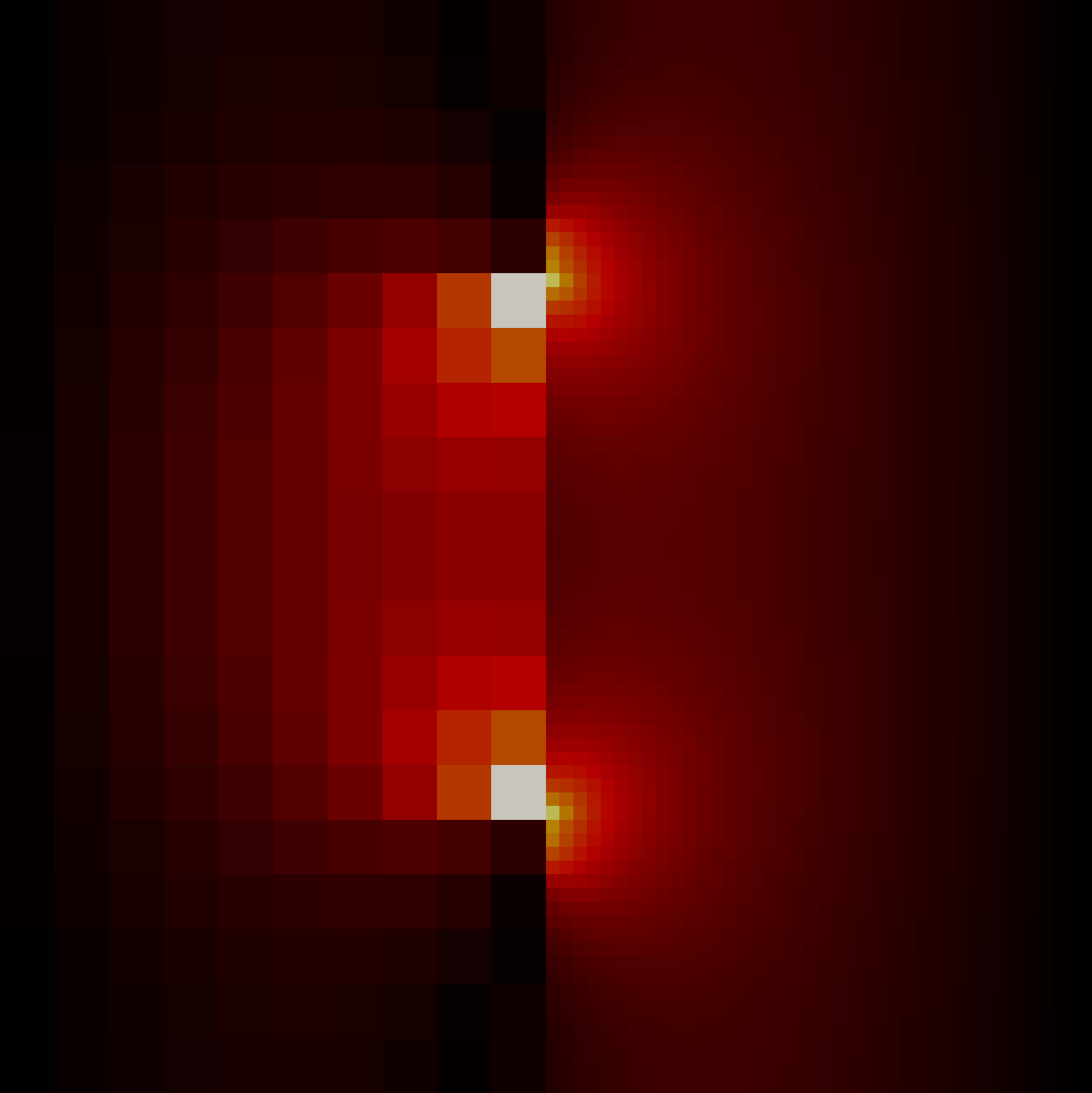}%
        \label{fig:reducedModel_test2_error_1-3}
    }
    \hspace{0.025\textwidth}%
    \includegraphics{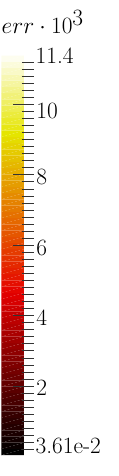}
    \caption{Error history for both the porous medium and the fault and a
        representation of the error for a particular choice of the mesh.
        The family of meshes used are represented in Figure
        \ref{fig:reducedModel_test2_mesh}.
        In dashed lines are represented also some reference curves.}%
\end{figure}
In Figure \ref{fig:reducedModel_test2_error_1-3} we can see the different distribution of
the error for the two sides of the domain, mainly present in its coarse part.
Anyway, in each side, the error is concentrated close to the singularities.


\subsection{Conductive fault} \label{subsec:conductive_fault}


We consider the domain $\Omega = (0,1)^2$ with a vertical fault, of thickness
$d=10^{-2}$, in the centre of the domain, see Figure
\ref{fig:reducedModel_test1_domain} for a sketch of the computational domain.
\begin{figure}[!htp]
    \centering
    \subfloat[Domain.]
    {
        \includegraphics{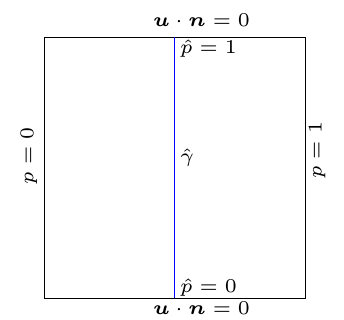}
        \label{fig:reducedModel_test1_domain}
    }%
    \hspace{0.1\textwidth}%
    \subfloat[Mesh.]
    {
        \raisebox{0.05\height}{\includegraphics[width=0.2\textwidth]{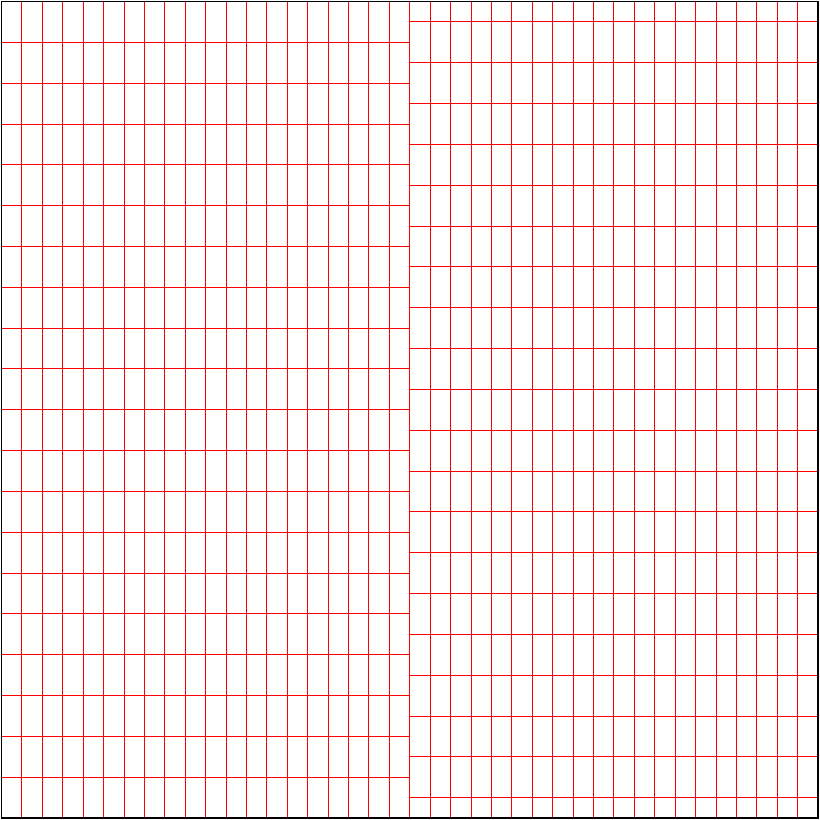}}%
        \label{fig:domain_reducedModel_test1_mesh}
    }
    \caption{Computational domain for tests in the subsection \ref{subsec:conductive_fault}
        with the boundary conditions and an example of a mesh used.}%
    \label{fig:domain_reducedModel_test1}%
\end{figure}
We assume homogeneous Neumann boundary conditions
on the top and bottom of the domain, homogeneous Dirichlet boundary condition
at left and Dirichlet boundary condition $p=1$ in the right part of the domain.
We impose Dirichlet boundary conditions for both the ending of the
fault, with value $p=1$ at the top and homogeneous at the bottom.
Finally we consider identity matrix as permeability in the domain and in the
fault we impose $\lambda_{f,\ttau}=10^{-2}$ and $\lambda_{f,\normal}=1$.
\begin{figure}[!htp]
    \centering
    \subfloat[Solution.]
    {
        \includegraphics[width=0.33\textwidth]{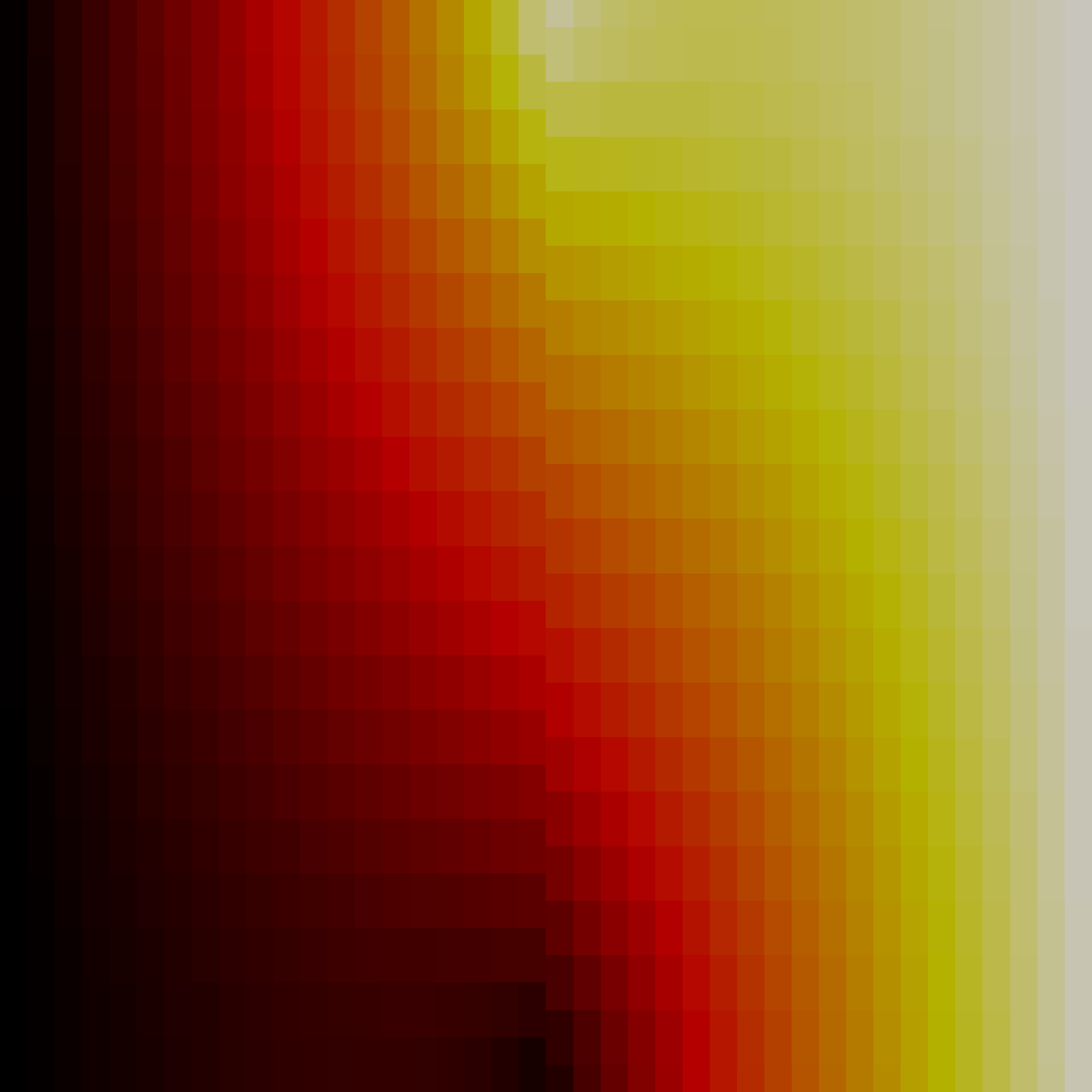}
    }%
    \hspace{0.05\textwidth}%
    \includegraphics{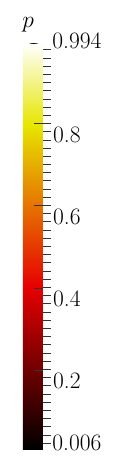}
    \hspace{0.05\textwidth}%
    \subfloat[Warped solution.]
    {
        \includegraphics[width=0.31\textwidth]{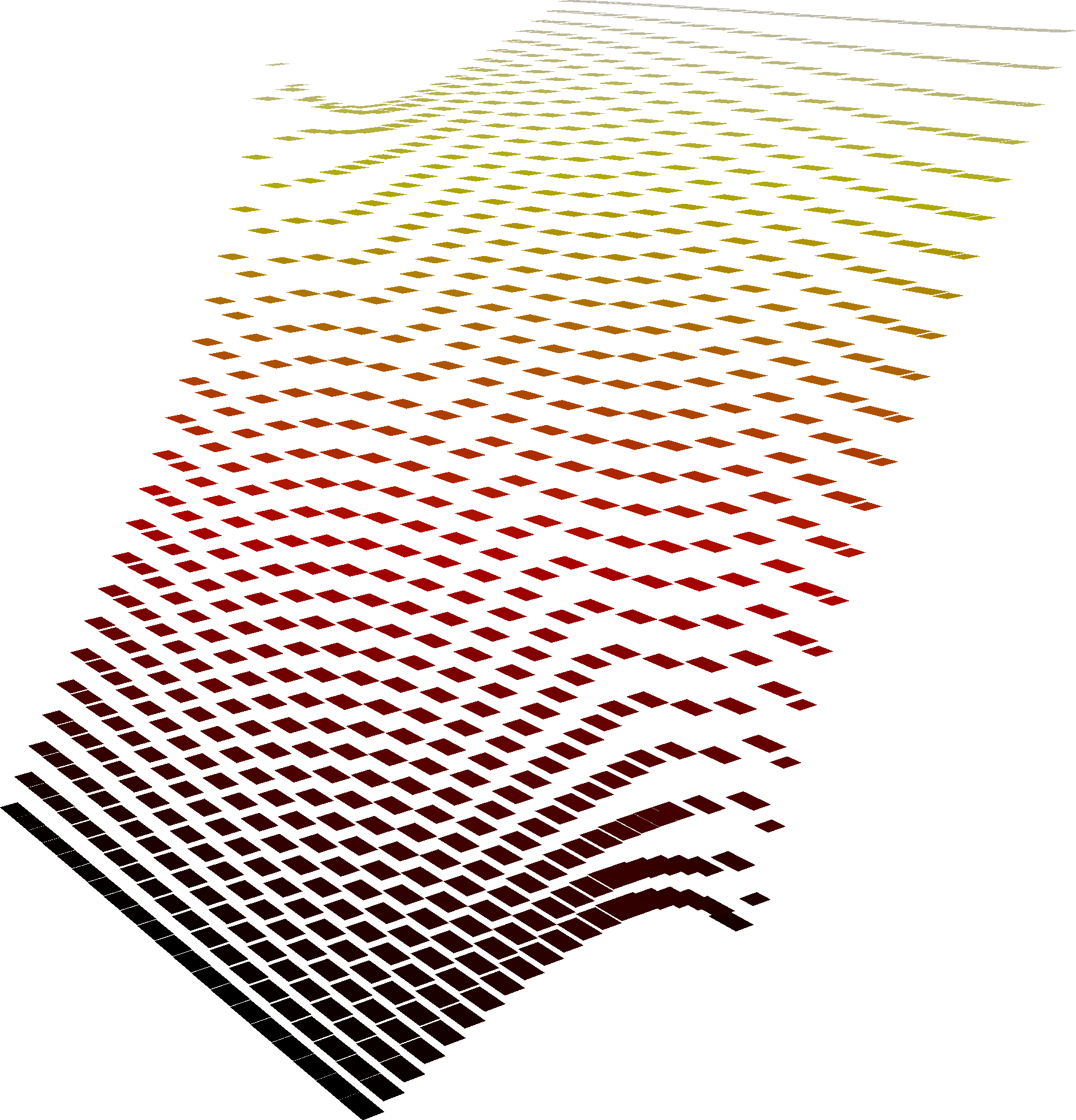}
    }%
    \caption{Pressure field for the conductive case with non-matching grid at the fault.}%
    \label{fig:reducedModel_test1}%
\end{figure}
The computational mesh is composed by quadrangular elements, non-matching at the
fault. The solution of \eqref{eq:red_problem} is depicted in Figure
\ref{fig:reducedModel_test1}, we notice that the solution across the fault is
continuous, as we expect, and the geometrical non-conformity is handled without
any problem. The maximum and minimum discrete principle, in this particular
experiment, are fulfilled.

To compute the error decay we consider a reference grid of approximatively two
millions of elements. In Figure
\ref{fig:reducedModel_test1_error_dacay_plot} we present the error history.  The
estimated order of the pressure error for the porous medium is a little lower
then $\Order{h_\disc^2}$. Moreover the error for the two layers of the fault is
in between $\Order{h_\disc^{\frac{3}{2}}}$ and $\Order{h_\disc^2}$, closer to
the latter.  If we suppose that the exact solution is continuous in $\Omega$,
then we have the numerical evidence of the second order of convergence of both
the pressure in the porous medium and in the fault. In Figure
\ref{fig:reducedModel_test1_error_1-3} is represented an example of the error,
we can notice that the highest error is close the both the ends of the fault,
which is a normal behaviour.
\begin{figure}[!htp]
    \centering
    \subfloat[Error history.]
    {
    \includegraphics{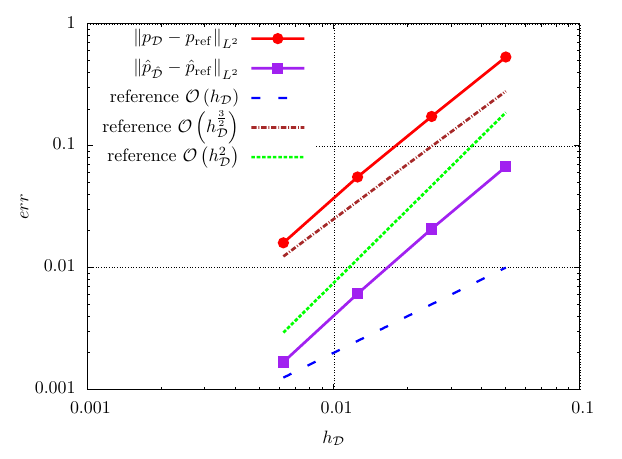}
        \label{fig:reducedModel_test1_error_dacay_plot}%
    }
    \hspace{0.025\textwidth}%
    \subfloat[Error.]
    {
        \includegraphics[width=0.33\textwidth]{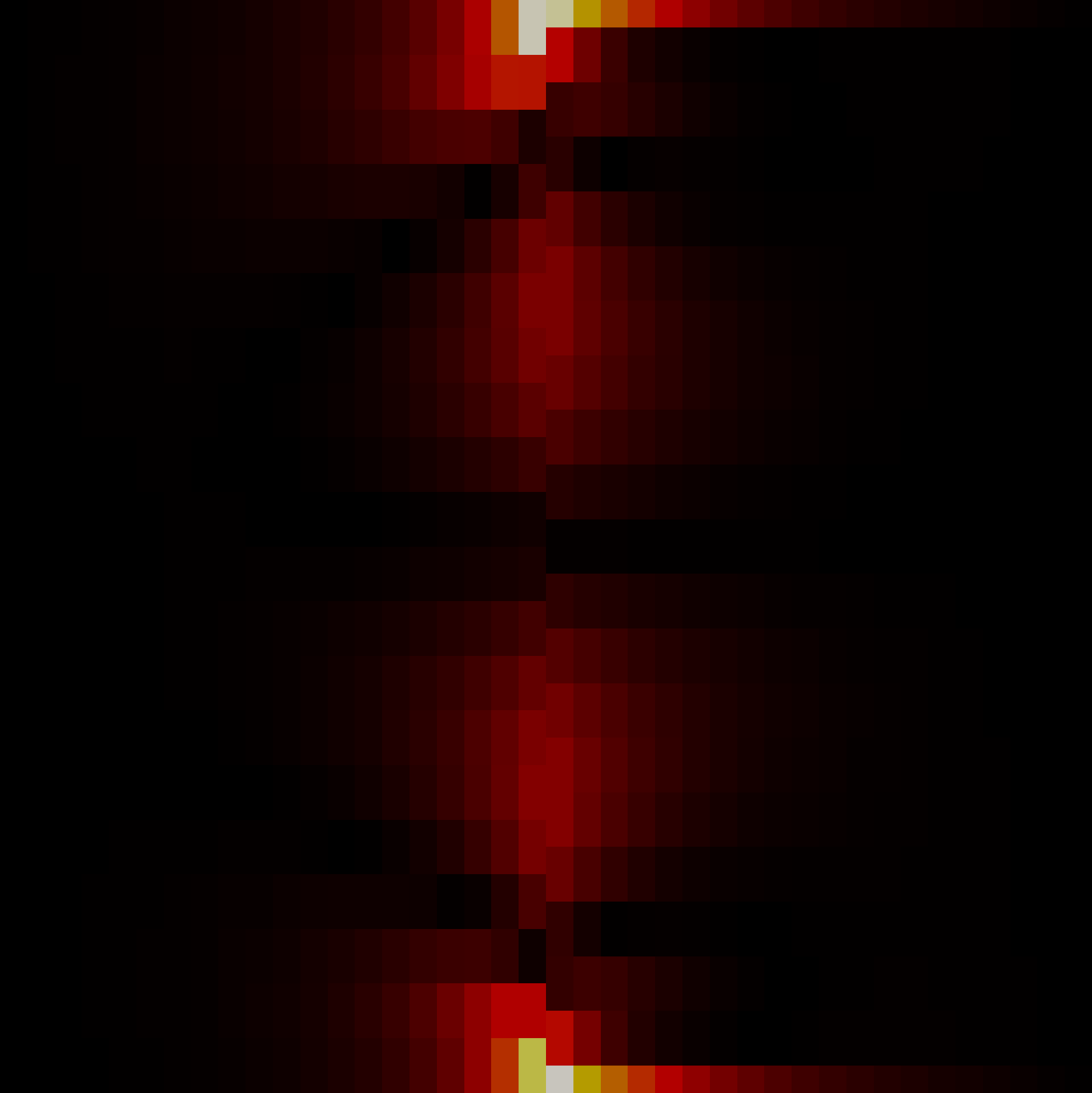}%
        \label{fig:reducedModel_test1_error_1-3}
    }
    \hspace{0.025\textwidth}%
    \includegraphics{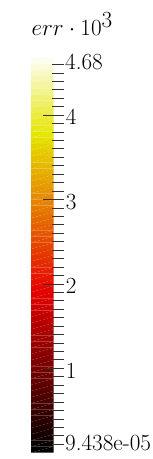}
    \caption{Error history for both the porous medium and the fault and a
        representation of the error for a particular choice of the mesh. The
        family of meshes used are represented in Figure
        \ref{fig:domain_reducedModel_test1_mesh}. In dashed
        lines are represented also some reference curves.}%
    \label{fig:reducedModel_test1_error_dacay}%
\end{figure}
We consider also a different family of meshes for the error analyses, a coarser
example is represented in Figure \ref{fig:reducedModel_test2_mesh}. Each
elements in the left part of the domain is constructed with 16 of small elements
used for the right part. Also in this case both the pressure errors are close to
$\Order{h_\disc^2}$. Figure \ref{fig:reducedModel_test1_error_1-3_non_matching} shows
the error for a particular mesh, also in this case it is concentrated close the
two ends of the fault. As we expect the error is higher in the coarse part of
the mesh.
\begin{figure}[!htp]
    \centering
    \subfloat[Error history.]
    {
    \includegraphics{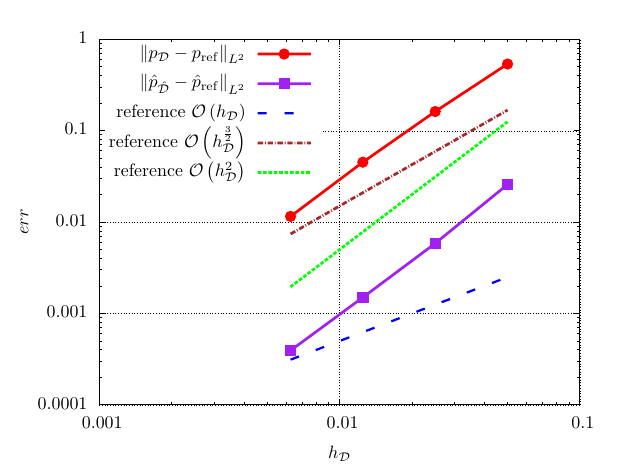}
    }
    \hspace{0.025\textwidth}%
    \subfloat[Error.]
    {
        \includegraphics[width=0.33\textwidth]{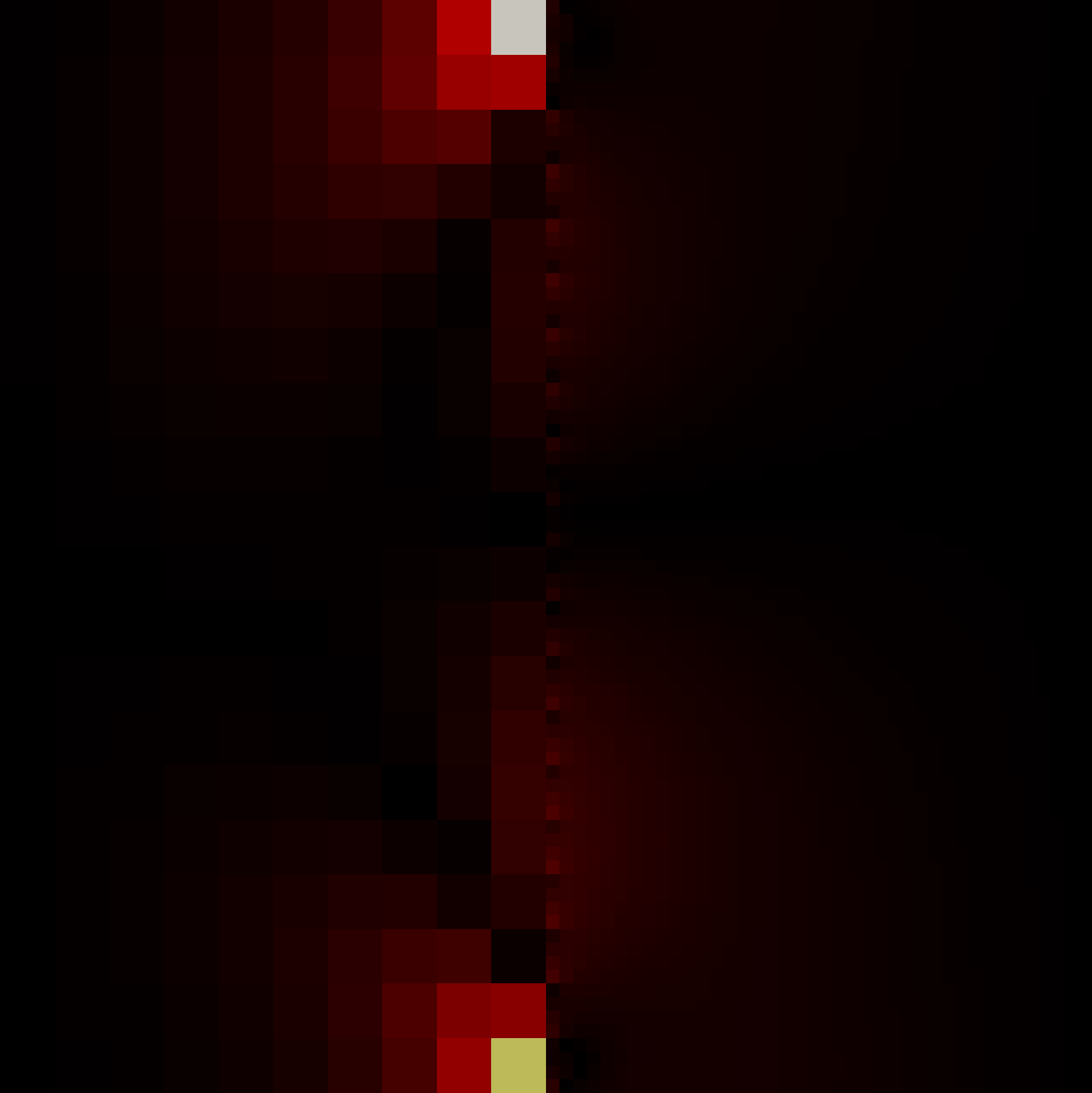}%
        \label{fig:reducedModel_test1_error_1-3_non_matching}
    }
    \hspace{0.025\textwidth}%
    \includegraphics{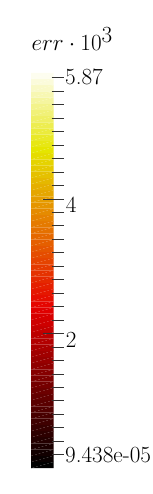}
    \caption{Error history for both the porous medium and the fault and a
        representation of the error for a particular choice of the mesh.
        The family of meshes used are represented in Figure
        \ref{fig:reducedModel_test2_mesh}. In dashed
        lines are represented also some reference curves.}%
    \label{fig:reducedModel_test1_error_dacay_non_matching}%
\end{figure}
We notice in both Figures \ref{fig:reducedModel_test1_error_1-3} and
\ref{fig:reducedModel_test1_error_1-3_non_matching}, especially in the right
part of the domain for the latter, some oscillations in the error. Contrary to
\cite{Frih2011}, in this case these spurious effects are due to a mesh effect. In
Figure \ref{fig:reducedModel_test1_error_1-3_cartesian} we compute the error for
a Cartesian mesh, the oscillations are not present.
\begin{figure}[!htp]
    \centering
    \includegraphics[width=0.33\textwidth]{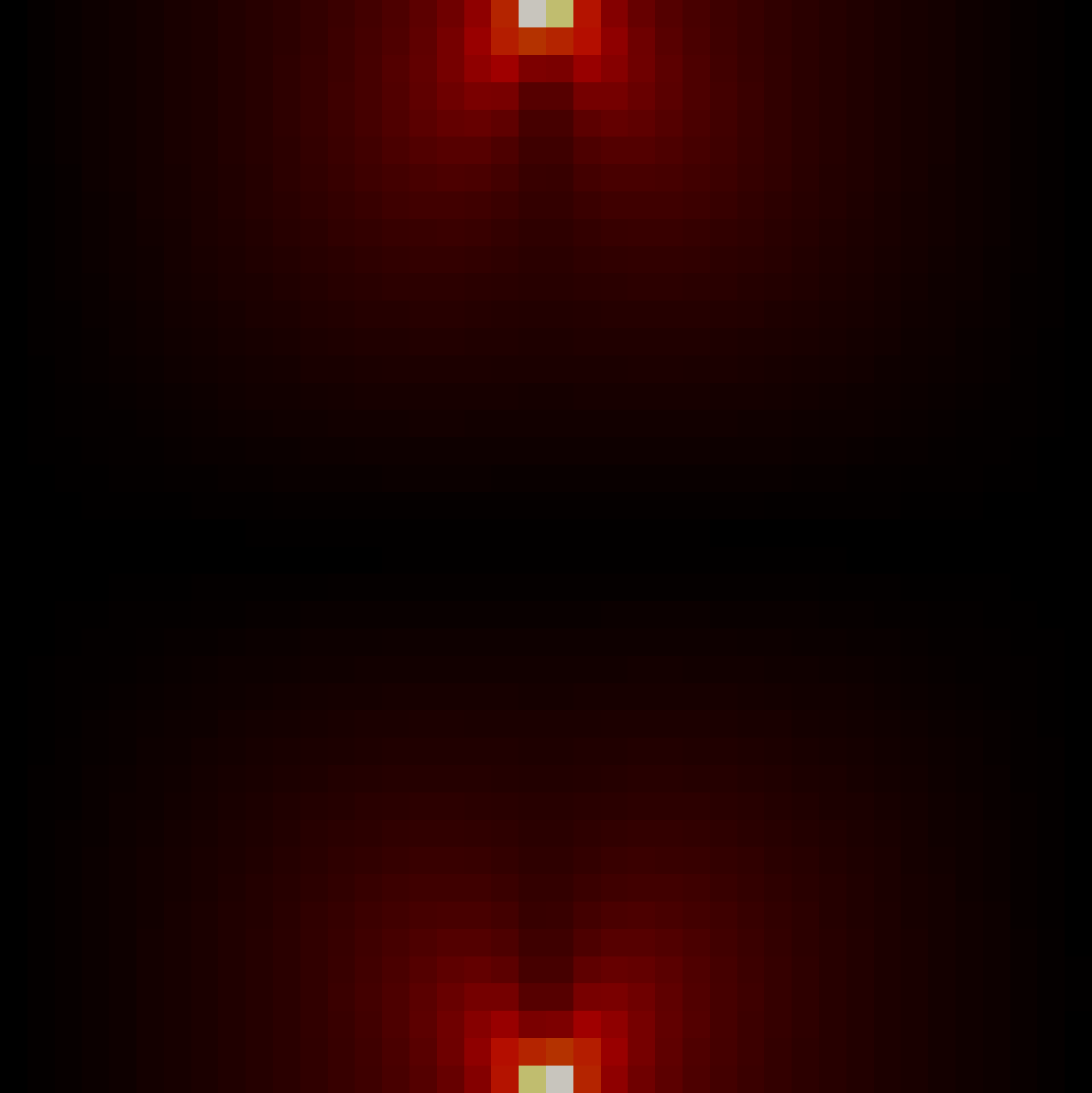}%
    \hspace{0.025\textwidth}%
    \includegraphics{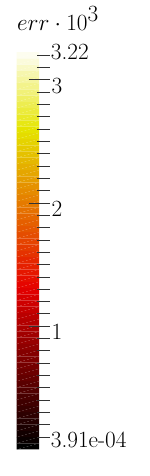}
    \caption{Representation of the error for a particular choice of a Cartesian mesh.}%
     \label{fig:reducedModel_test1_error_1-3_cartesian}
\end{figure}


\subsection{Anisotropic fault} \label{subsec:anisotropic_fault}


In this test case we present a much more involved example then the previous one,
to verify the goodness of the numerical solution in presence of strong contrast
in the mesh size. We consider the domain $\Omega=(0,1)^2$  with a vertical
fault of width $d=10^{-2}$. See Figure \ref{fig:domain_reducedModel_test3} for
a sketch of the computational domain.
\begin{figure}[!htp]
    \centering
\includegraphics{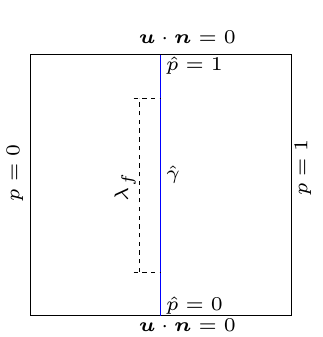}
    \caption{Computational domain for tests in the subsection
        \ref{subsec:anisotropic_fault} with the boundary conditions.}%
    \label{fig:domain_reducedModel_test3}%
\end{figure}
We impose homogeneous Neumann boundary conditions on the top and bottom of the
domain and Dirichlet boundary conditions for the right and left part of the
domain, as well as the fault. For the domain we assume $p=0$ on the left side
and $p=1$ on the right side, while for the fault $p=0$ on the bottom and $p=1$
in the top. We consider identity matrix for the porous medium and, given
$\lambda_f = 100$, for the fault
\begin{gather*}
    \Lambda_f (s) =
    \begin{bmatrix}
        \lambda_f & 0 \\
        0 & \lambda_f^{-1}
    \end{bmatrix}
    \text{ for } s \in (0.25,0.75), \quad
    \Lambda_f (s) =
    \begin{bmatrix}
        \lambda_f^{-1} & 0\\
        0 & \lambda_f
    \end{bmatrix}
    \text{ for } s \in (0,0.25) \cup (0.75,1).
\end{gather*}
In its two extreme parts, the fault behaves as a low permeable strata for the flow across
itself while as a channel for the flow inside. Vice versa for the other part of the
fault, giving a solution with two singularities in the points $(0.5,0.25)$ and
$(0.75,1)$. We consider a family of meshes composed by fixed coarse
discretization of the left part and a refined discretization of the right part
of the domain.
\begin{figure}[!htp]
    \centering
    \subfloat[4 cells.]
    {
        \includegraphics[width=0.27\textwidth]{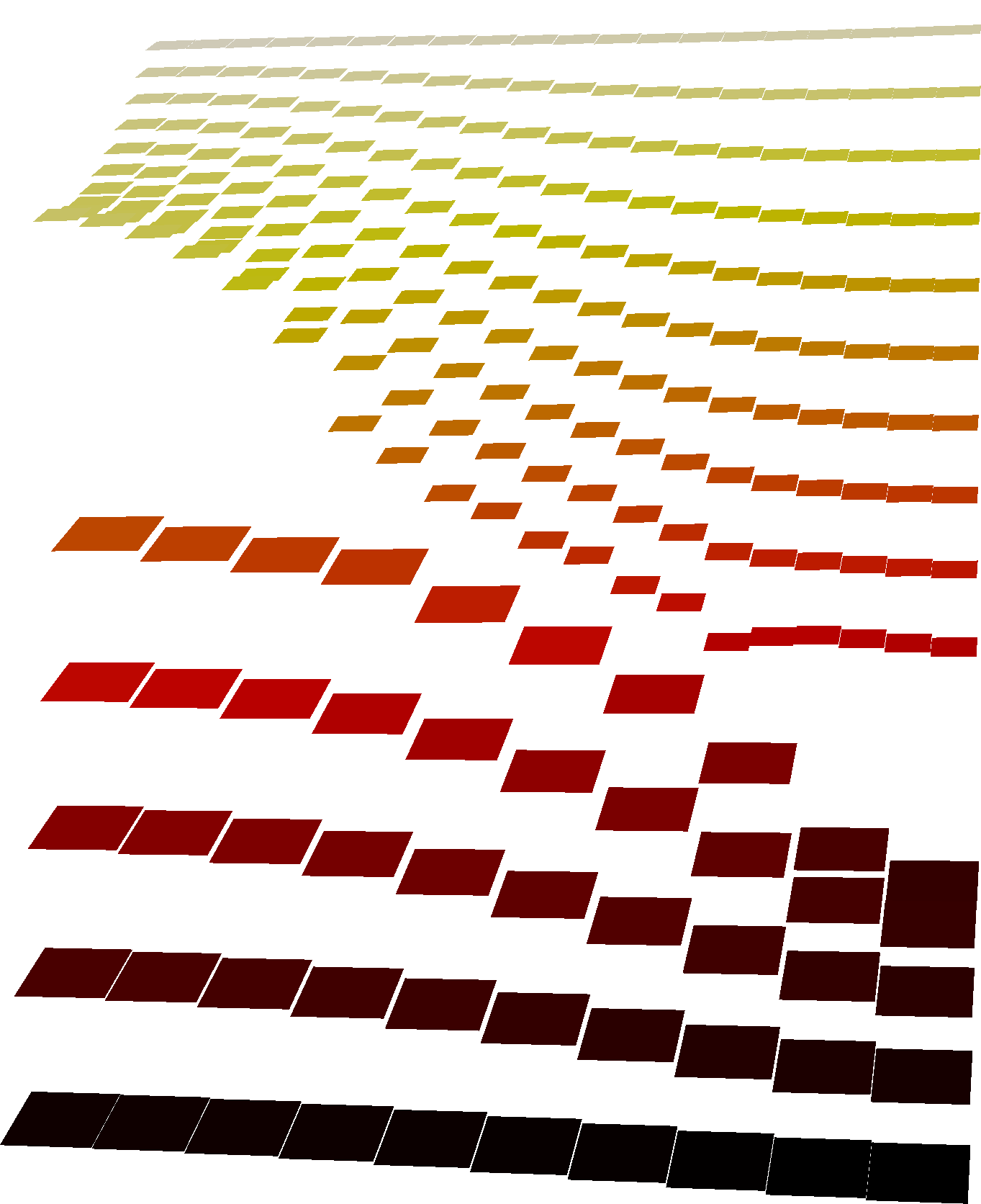}%
    }
    \subfloat[256 cells.]
    {
        \includegraphics[width=0.27\textwidth]{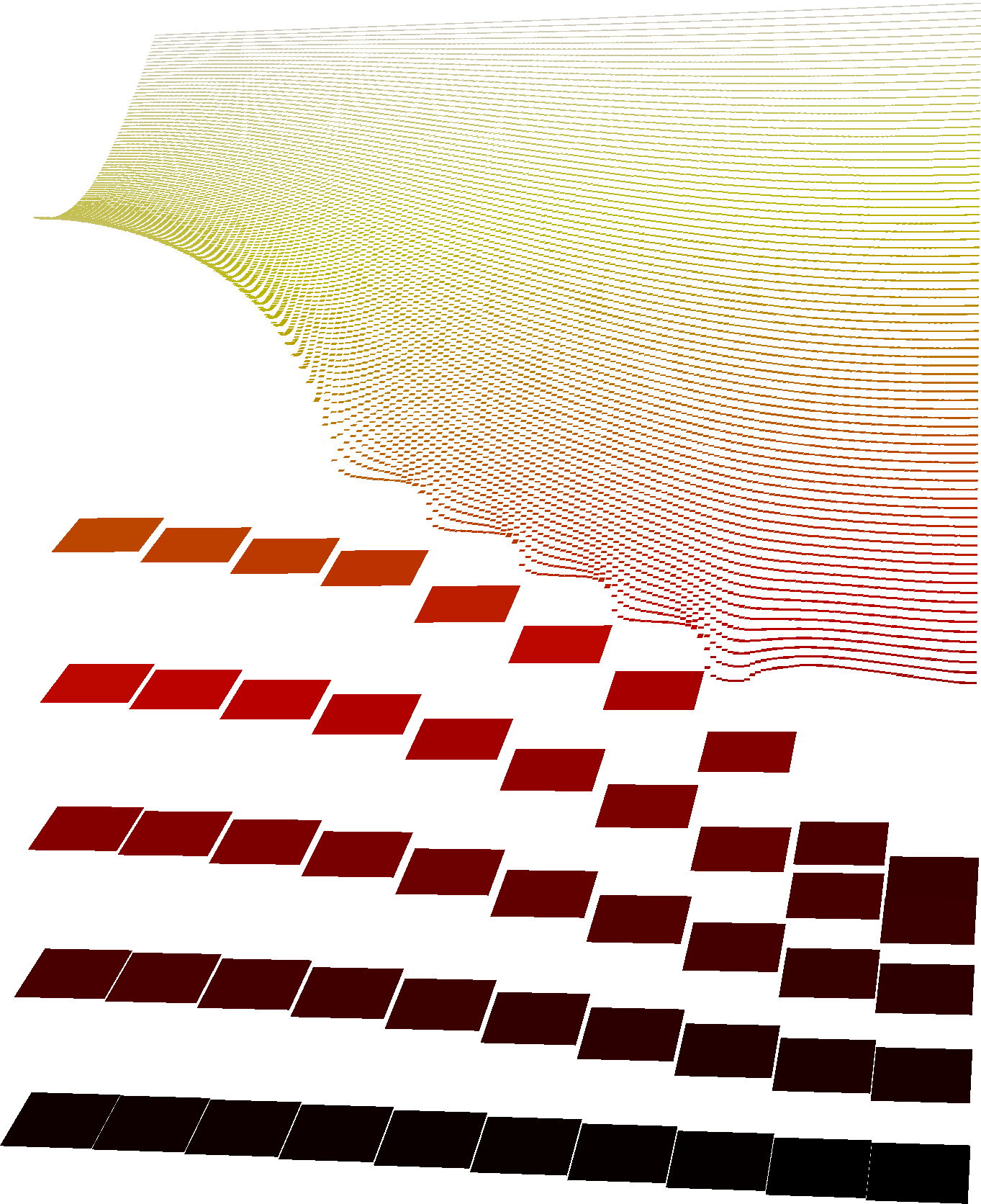}%
    }
    \subfloat[4096 cells.]
    {
        \includegraphics[width=0.27\textwidth]{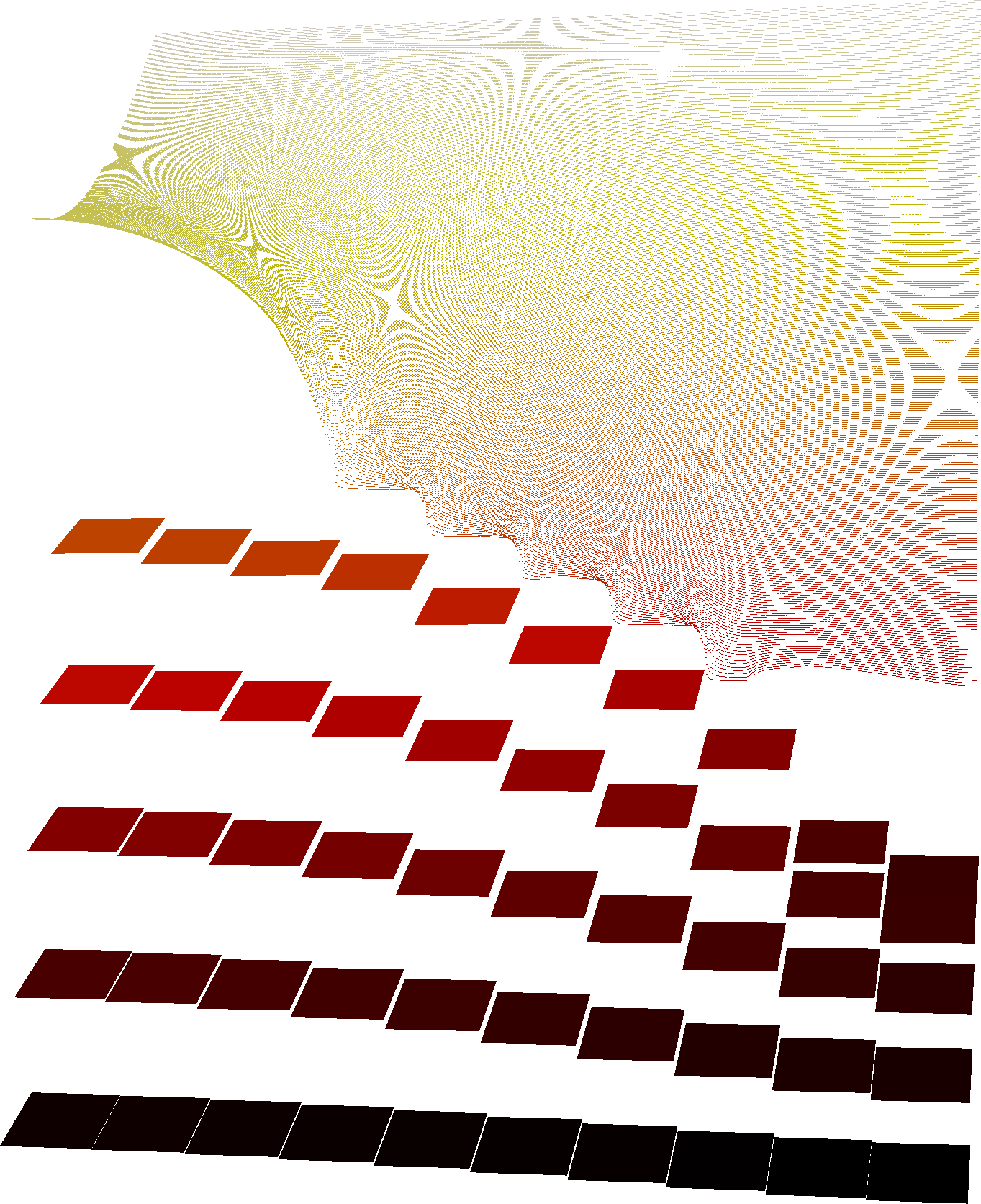}%
        \label{fig:reducedModel_test3_oscillations_part3}
    }
    \hspace{0.025\textwidth}%
    \includegraphics{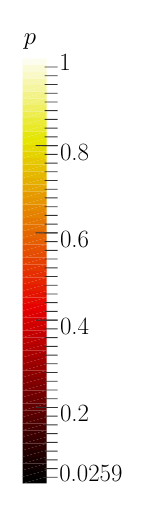}
    \caption{Pressure solutions for different meshes. Each legend depict the
        division of a coarse element to obtain the finer mesh.
        The family of meshes used are represented in Figure
        \ref{fig:reducedModel_test2_mesh}.}%
    \label{fig:reducedModel_test3_oscillations}%
\end{figure}
In Figure \ref{fig:reducedModel_test3_oscillations} are reported different
solutions for different meshes. The solutions keep in evidence the natural
dependence on the mesh, which becomes much significant when the discretization
of one side of the fault is much finer then the other side. In particular for
Figure \ref{fig:reducedModel_test3_oscillations_part3} the fine solution, in the
central part of the fault, is almost flat in correspondence of each element of
the coarse solution and exhibit a ``jump'' in correspondence of two different
coarse elements. Anyway the obtained solution is reasonable.
\begin{table}
    \centering
    \begin{tabular}{|lcccccccc|}
        \hline
            $h_\disc^{\rm max}/h_\disc^{\rm min}$& 1 & 2 & 4 & 8 & 16 & 32 & 64 & 128 \\
        \hline
            AMG & 8 & 10 & 10 & 11 & 11 & 12 & 12 & 12 \\
        \hline
            ILU4 & 12 & 16 & 23 & 48 & 110 & 253 & 678 & 1462 \\
        \hline
    \end{tabular}
    \caption{Number of iterations, for different mesh ratio $h_\disc^{\rm
        max}/h_\disc^{\rm min}$, to reach the convergence with to different
        preconditioners.}%
    \label{tab:iterations}
\end{table}
In Table \ref{tab:iterations} we present the number of iterations of a GMRES linear solver
to obtain the solution of the problem. We consider a stopping criteria on the
residual smaller then $10^{-12}$, running the code only with one processor. In
the table we consider two different preconditioner for the linear system: the
algebraic multi-grid (AMG), form the Hypre library \cite{Falgout2002}, and the
incomplete LU factorization with level of fill equal to 4, from the library
PETSc \cite{Balay2013}. The result are quite
promising for the AMG method since the number of iterations is almost constant,
while for ILU4 the number of iterations increases at each refinement.
Finally, even if an iteration of the ILU4 is cheaper in terms of CPU time than
an iteration of the AMG, the numbers of iterations are so different that,
from our experiments, we suggest to use the AMG method to solve also realistic
problems.


\subsection{Slipping domain} \label{subsec:slipping_domain}


We consider now an example where one part of the domain slides, thanks
to the fault, on the other part. The simulation is a sequence of problems in a
moving domain: in its left side we have a deposition of sedimentary material and
a movement from the top to the bottom of the sub-domain. The right part of the
domain remains in the same position.
\begin{figure}[!htp]
    \centering
    \subfloat[Initial domain $t=t^*$.]
    {
        \raisebox{0.1125\height}{\includegraphics[width=0.45\textwidth]{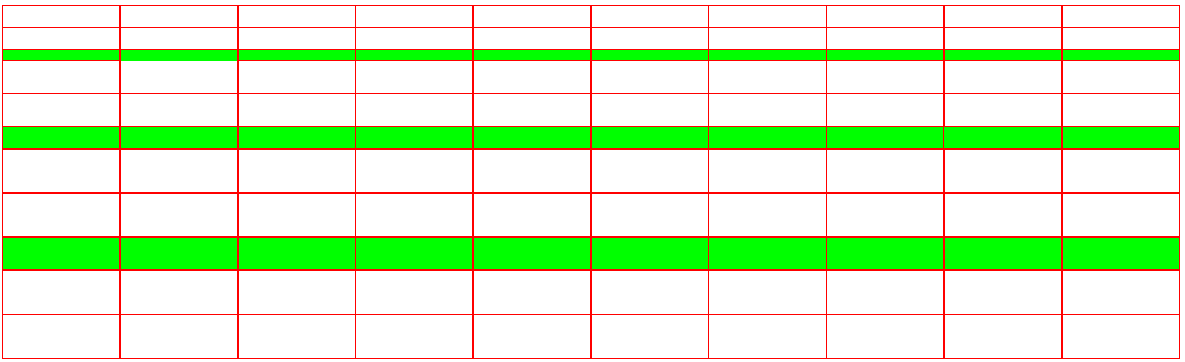}}%
        \label{fig:splip_domain_initial}
    }
    \subfloat[Final domain $t=T$.]
    {
        \includegraphics[width=0.45\textwidth]{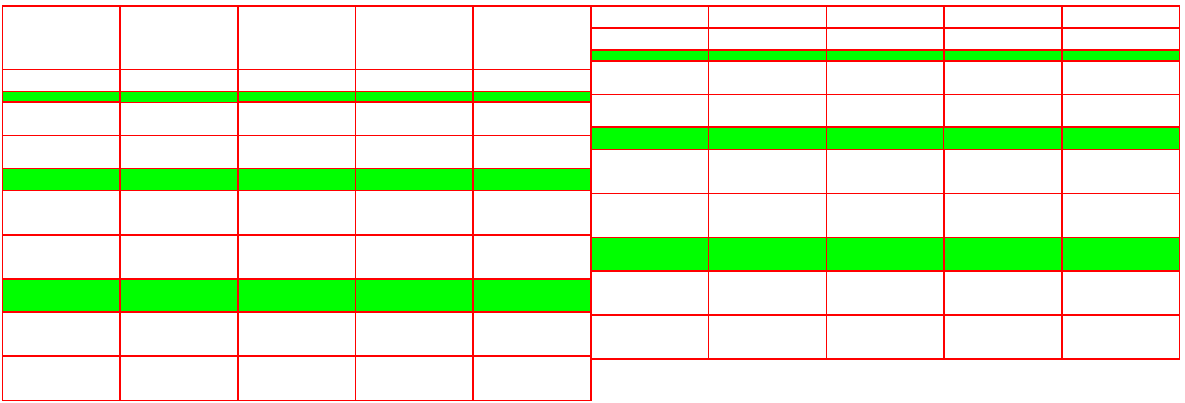}%
        \label{fig:splip_domain_final}
    }
    \caption{Representation of the meshes for two different configuration:
        the begin and the end of the simulation. The green cells are the
        discretization of $\Omega^{\rm barr}$.}%
    \label{fig:splip_domain}
\end{figure}
\begin{figure}[htbp]%
    \centering%
    \includegraphics[width=0.48\textwidth]{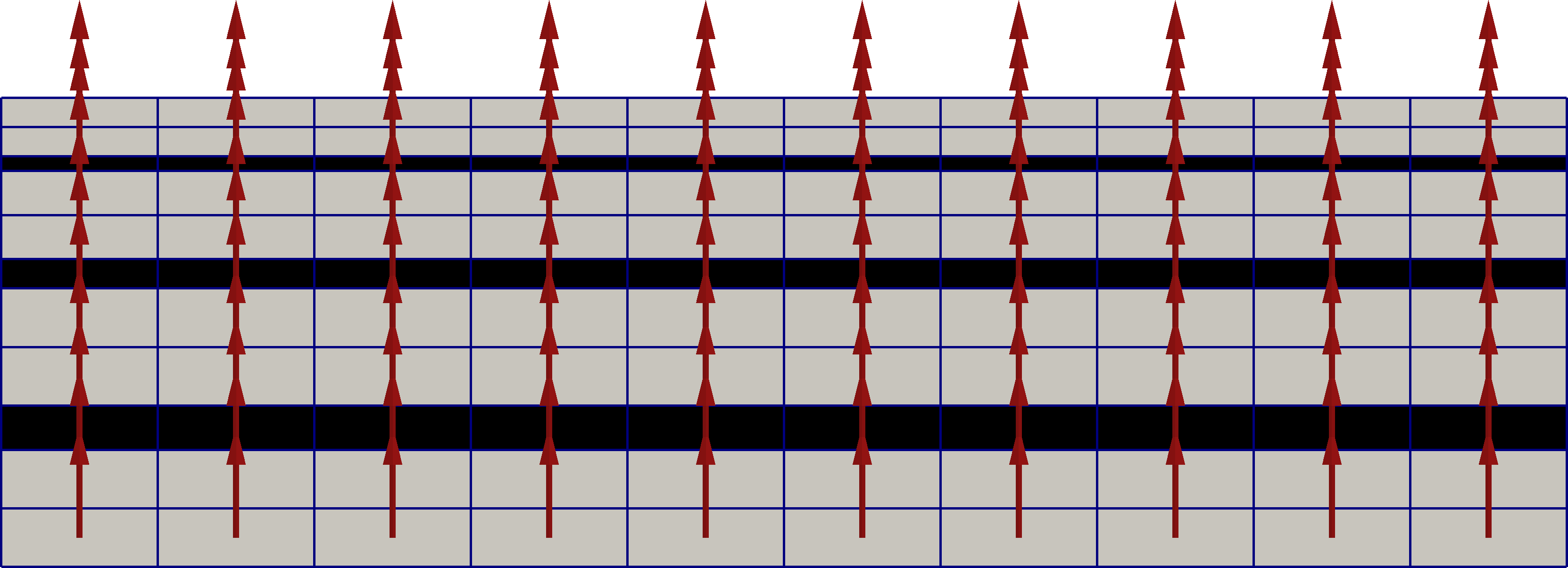}%
    \hspace{0.0025\textwidth}%
    \includegraphics[width=0.48\textwidth]{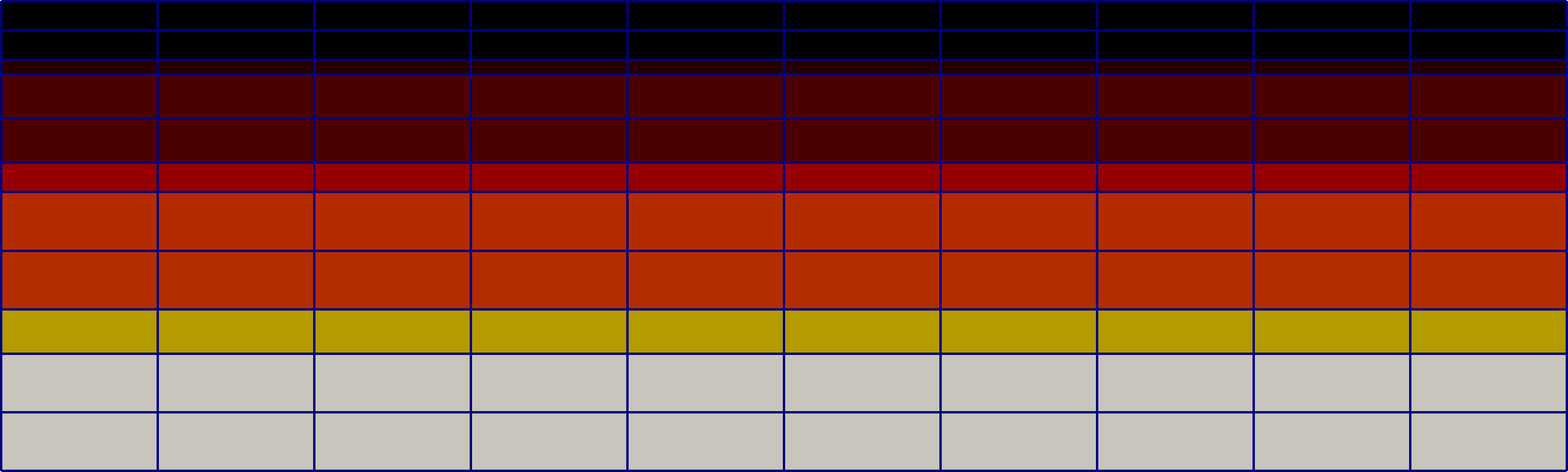}%
    \hspace{0.0025\textwidth}%
    \includegraphics{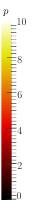}
    \caption{Representation of the initial solution for the pressure $p_0$ and
        the Darcy velocity.}%
    \label{fig:test4_0}%
\end{figure}
In Figure \ref{fig:splip_domain} we present the domain at two different times:
in the left at the beginning of the simulation and in the right at the end of
the simulation. In the former case we have $\Omega = (0,10) \times (-3, 0) Km^2$
and in the latter $\Omega \subset (0,10) \times (-3.35,0) Km^2$. The others
configurations of $\Omega$ moves from Figure \ref{fig:splip_domain_initial} to
Figure \ref{fig:splip_domain_final} linearly in time. The fault thickness is $d=50 m$.
The mathematical model is
the following: given $\partial \Omega^{\rm top}$ the top part of the boundary
condition and $t^*$ and $T$ the initial and final times, find $p$ such that
\begin{gather*}
    \begin{aligned}
        &c \Phi \dfrac{\partial p}{\partial t} - \nabla \cdot
        \dfrac{\Lambda}{\mu} \nabla p = 0 &\quad&
        \text{in } \Omega \times (t^*,T) \\
        &\Lambda \nabla p \cdot \n = 0 &&
        \text{on } \partial \Omega \setminus \partial
        \Omega^{\rm top} \times (t^*, T) \\
        &p = 0 && \text{on } \partial \Omega^{\rm top} \times (t^*, T) \\
        &p = p_0 && \text{in } \Omega \times \left\{ t^* \right\}
    \end{aligned},
\end{gather*}
where $\mu = 3.1 \cdot 10^4 Pa \cdot s$ is the dynamic viscosity.
Considering Figure \ref{fig:splip_domain}
we divide the domain $\Omega$ in the green part $\Omega^{\rm barr}$, which
behaves like a low permeable strata, and the remain part $\Omega \setminus\Omega^{\rm barr}$.
We impose as permeability and porosity and compressibility for the porous medium $\Lambda = \diag
\br{10^{-19}} m^2$ and $c\Phi = 0.1 \cdot 10^{-7} Pa^{-1}$ in $\Omega^{\rm barr}$ and $\Lambda = \diag
\br{10^{-15}} m^2$ and $c\Phi = 0.5 \cdot 10^{-7} Pa^{-1}$ in $\Omega \setminus \Omega^{\rm barr}$. The
initial and final times are: $t^* = -0.049 My$ and $T = 0.3 My$.  The initial
solution $p_0$ is computed, with domain in Figure
\ref{fig:splip_domain_initial}, thanks to the following problem
\begin{gather*}
    \begin{aligned}
        &- \nabla \cdot \dfrac{\Lambda}{\mu} \nabla p_0 = 0 &\quad& \text{in } \Omega \\
        &\Lambda \nabla p_0 \cdot \n = 0 && \text{on } \partial \Omega^{\rm left, right} \\
        &p_0 = 0 && \text{on } \partial \Omega^{\rm top} \\
        &p_0 = 10 && \text{on } \partial \Omega^{\rm bottom}
    \end{aligned},
\end{gather*}
with $\partial \Omega^{\rm left, right}$ is the left and right part of the $\partial
\Omega$ and $\Omega^{\rm bottom}$ the bottom part of the domain.  For the
computation of $p_0$ we consider the permeability in the fault cells equal to
the surrounding domain cell. The initial pressure is depicted in Figure
\ref{fig:test4_0}.
\begin{figure}[htbp] 
    \centering%
    \subfloat[Current time $t=0.02My$ and time step number $k=3$.]%
    {
        \includegraphics[width=0.48\textwidth]{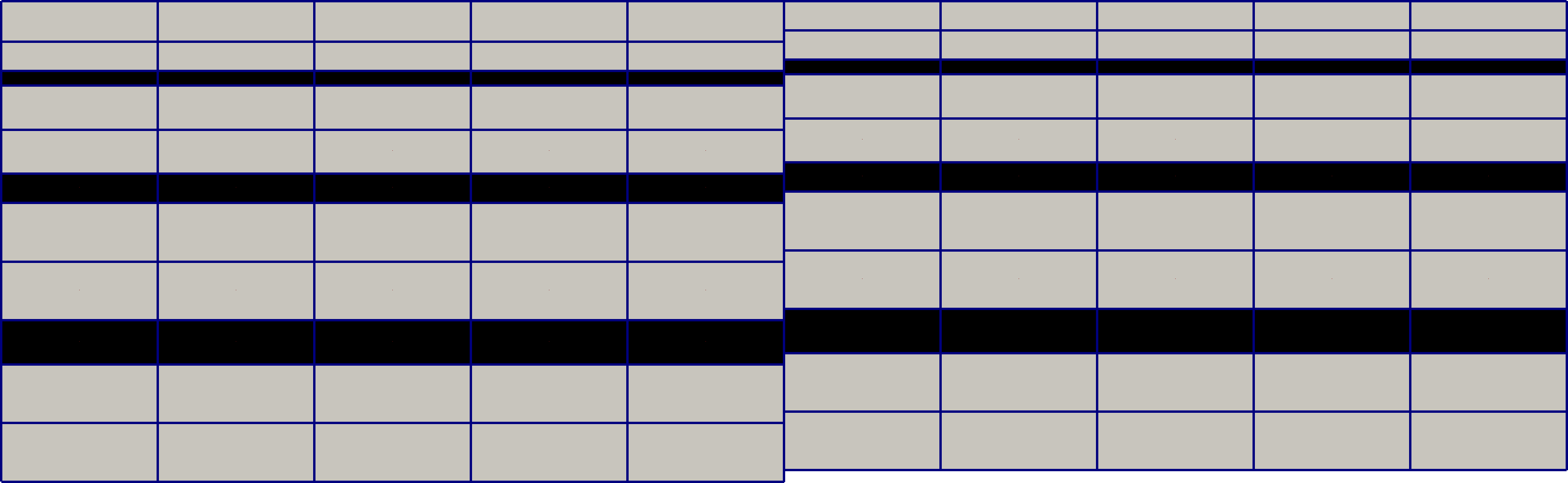}%
        \hspace{0.0025\textwidth}%
        \includegraphics[width=0.48\textwidth]{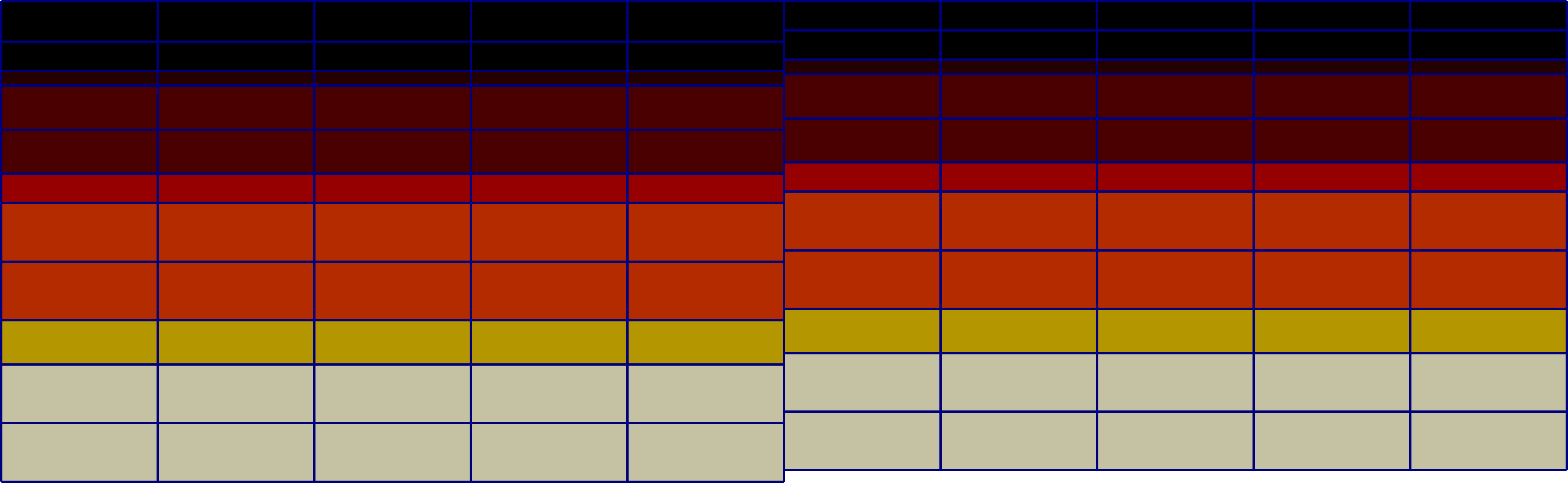}%
        \hspace{0.0015\textwidth}%
    \includegraphics[scale=0.8]{test4_legend_1}
    }
    \\
    \subfloat[Current time $t=0.083My$ and time step number $k=6$.]%
    {
        \includegraphics[width=0.48\textwidth]{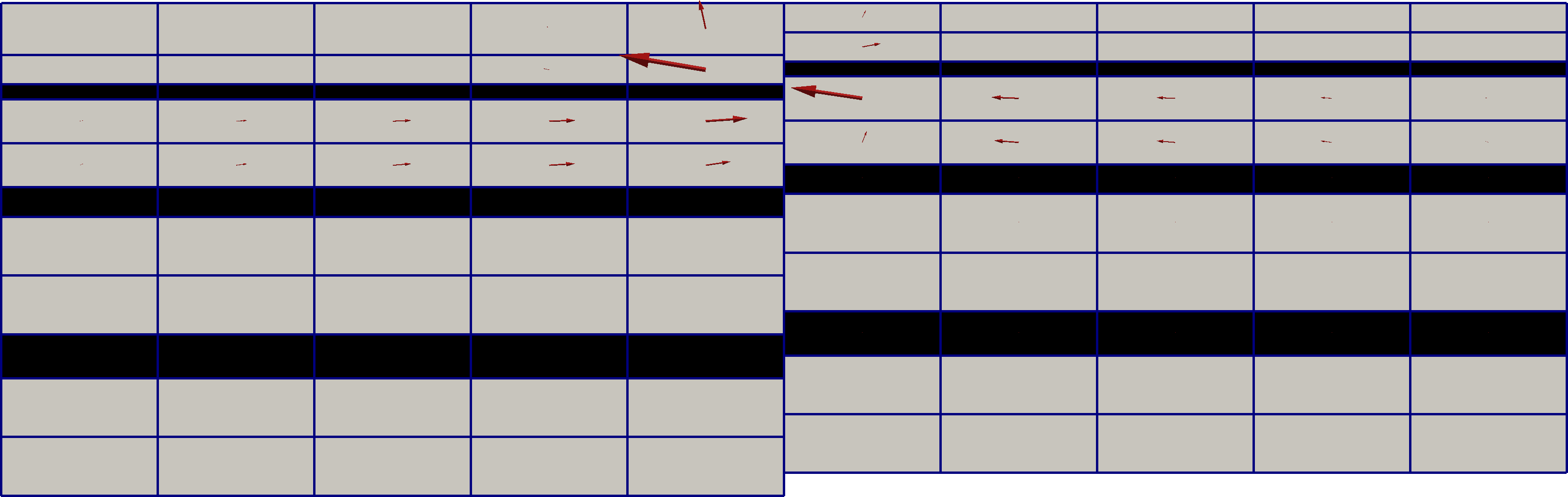}%
        \hspace{0.0025\textwidth}%
        \includegraphics[width=0.48\textwidth]{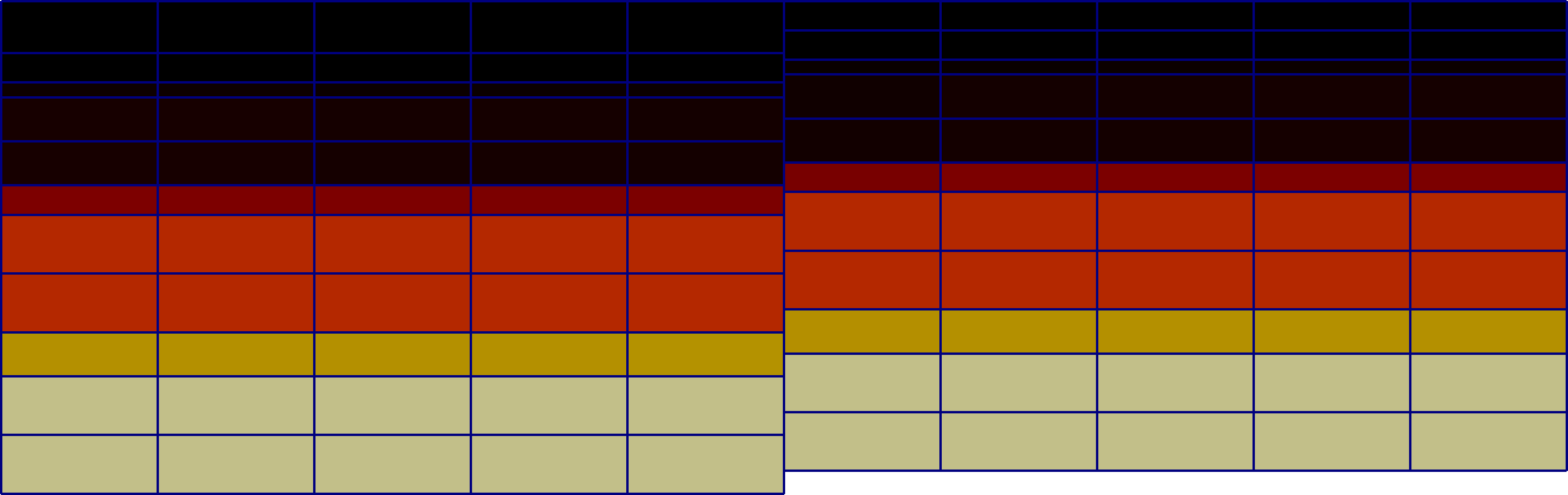}%
        \hspace{0.0015\textwidth}%
    \includegraphics[scale=0.8]{test4_legend_1}
    }
    \\
    \subfloat[Current time $t=0.125My$ and time step number $k=8$.]%
    {
        \includegraphics[width=0.48\textwidth]{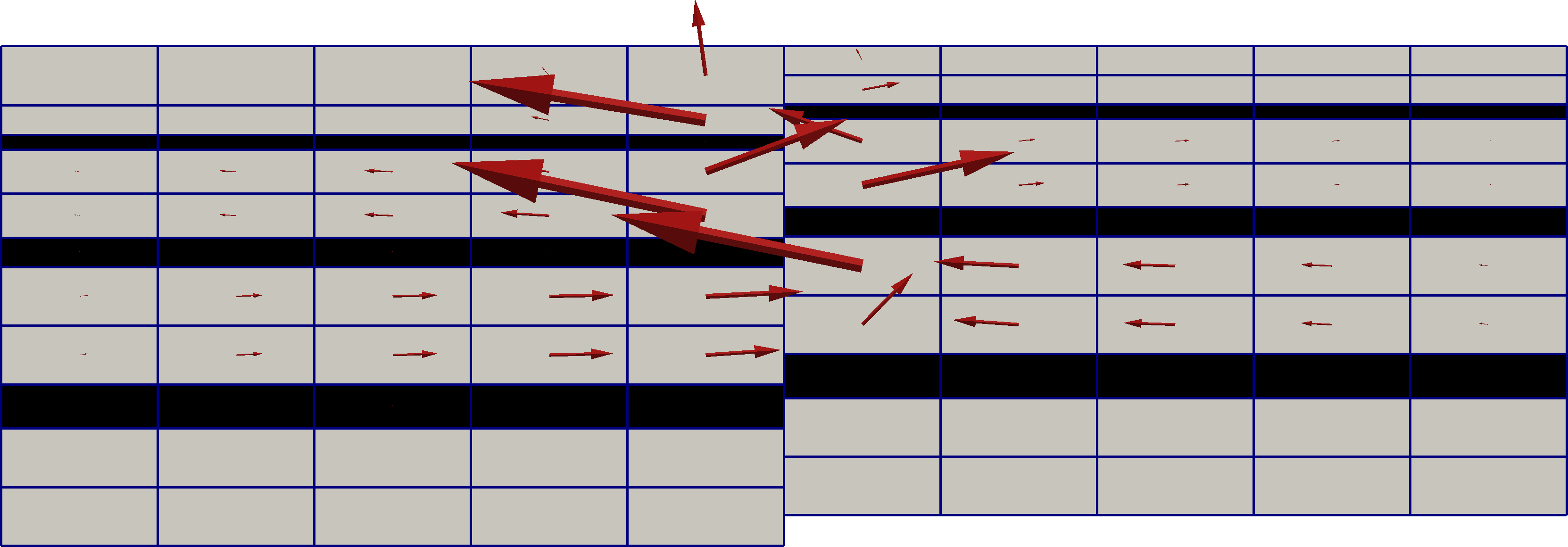}%
        \hspace{0.0025\textwidth}%
        \includegraphics[width=0.48\textwidth]{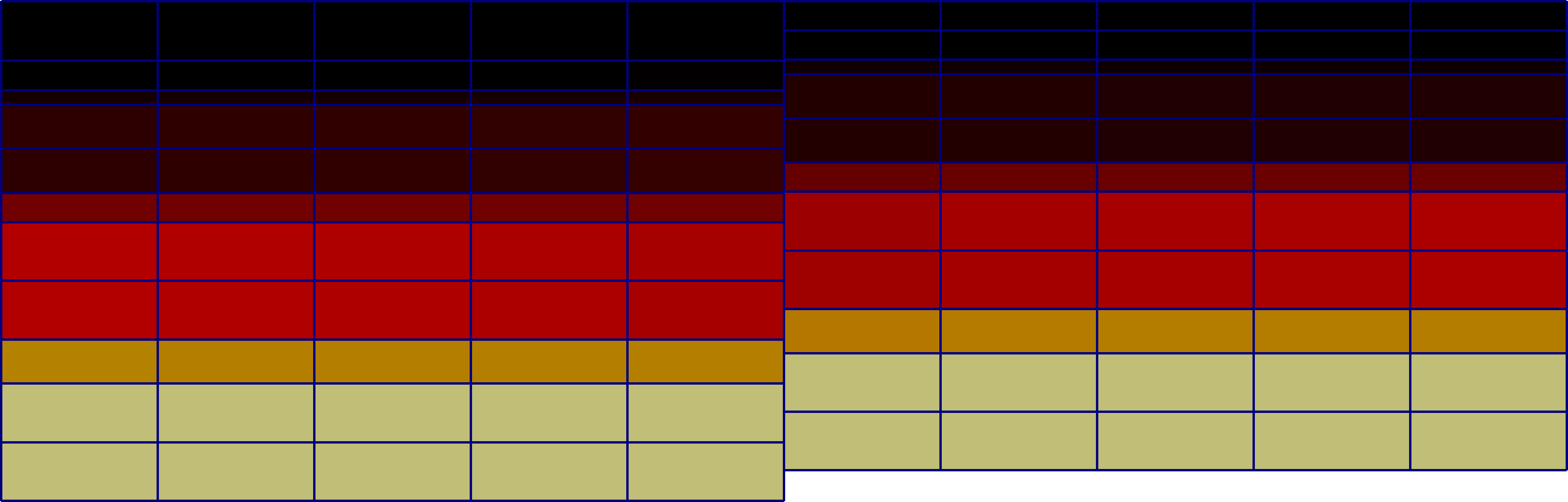}%
        \hspace{0.0015\textwidth}%
    \includegraphics[scale=0.8]{test4_legend_1}
    }
    \caption{Representation of different solution, pressure and Darcy velocity,
        for the neutral fault. The parameter $\chi=0.034$.}%
    \label{fig:test4_1}%
\end{figure}
We consider an implicit Euler scheme for the time discretization, no
interpolation operator is considered in the left part of the domain.  We
consider now three different test to validate the model. In each test we change
the value of the permeability inside the fault, while the porosity in the fault
is equal to the porosity of the surrounding porous medium. In all the images we
present both the pressure and the Darcy velocity, the latter using arrows with
size $\chi$-times its magnitude. We change the parameter $\chi$ to enhance the
readability.

As a first test, represented in Figure \ref{fig:test4_1} and \ref{fig:test4_1_b},
we present a sequence of solutions for
different time steps. For each cell in each layer of the fault we consider the
permeability equal to the permeability of the surrounding porous media. Then for
certain time steps the three layers open one after the other leading to a pressure drop.
\begin{figure}[htbp]%
    \centering%
    \subfloat[Current time $t=0.188My$ and time step number $k=11$.]%
    {
        \includegraphics[width=0.48\textwidth]{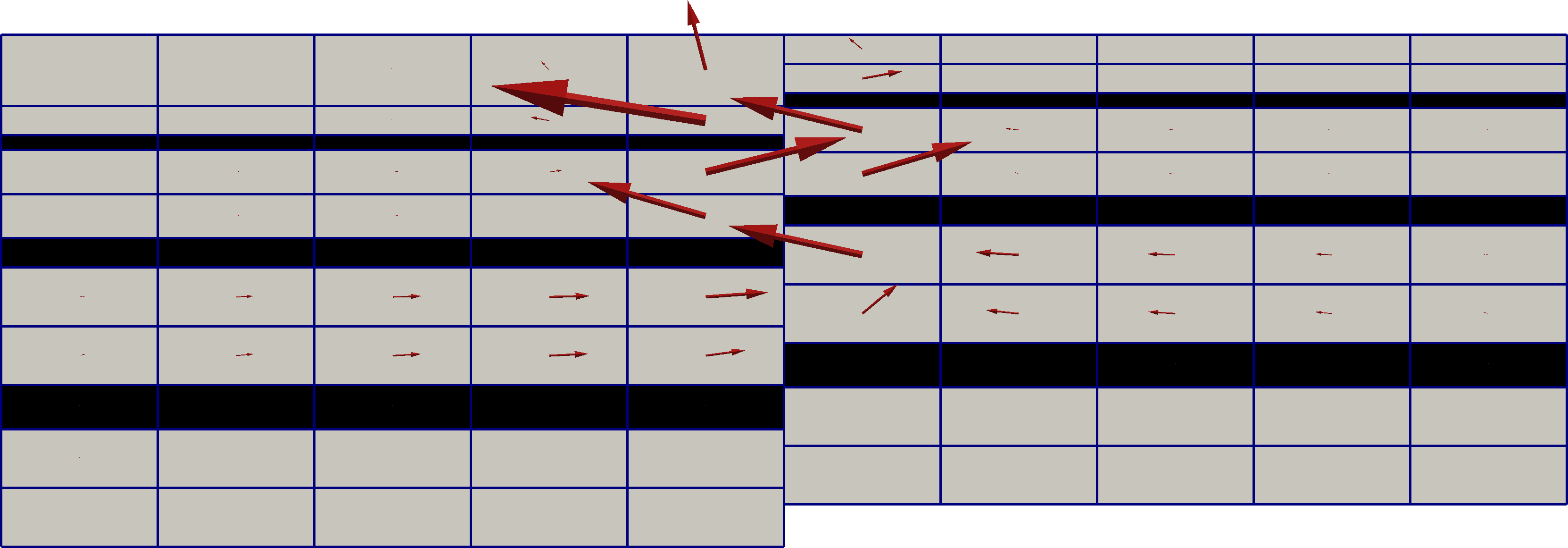}%
        \hspace{0.0025\textwidth}%
        \includegraphics[width=0.48\textwidth]{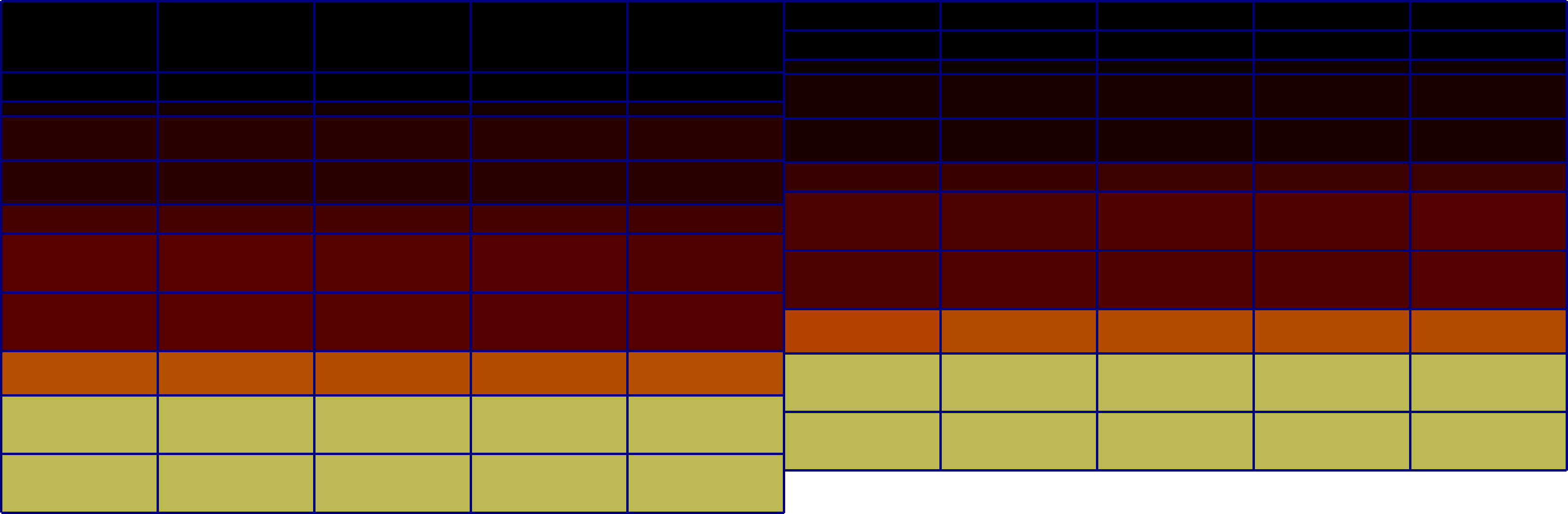}%
        \hspace{0.0015\textwidth}%
    \includegraphics[scale=0.8]{test4_legend_1}
    }
    \\
    \subfloat[Current time $t=0.209My$ and time step number $k=12$.]%
    {
        \includegraphics[width=0.48\textwidth]{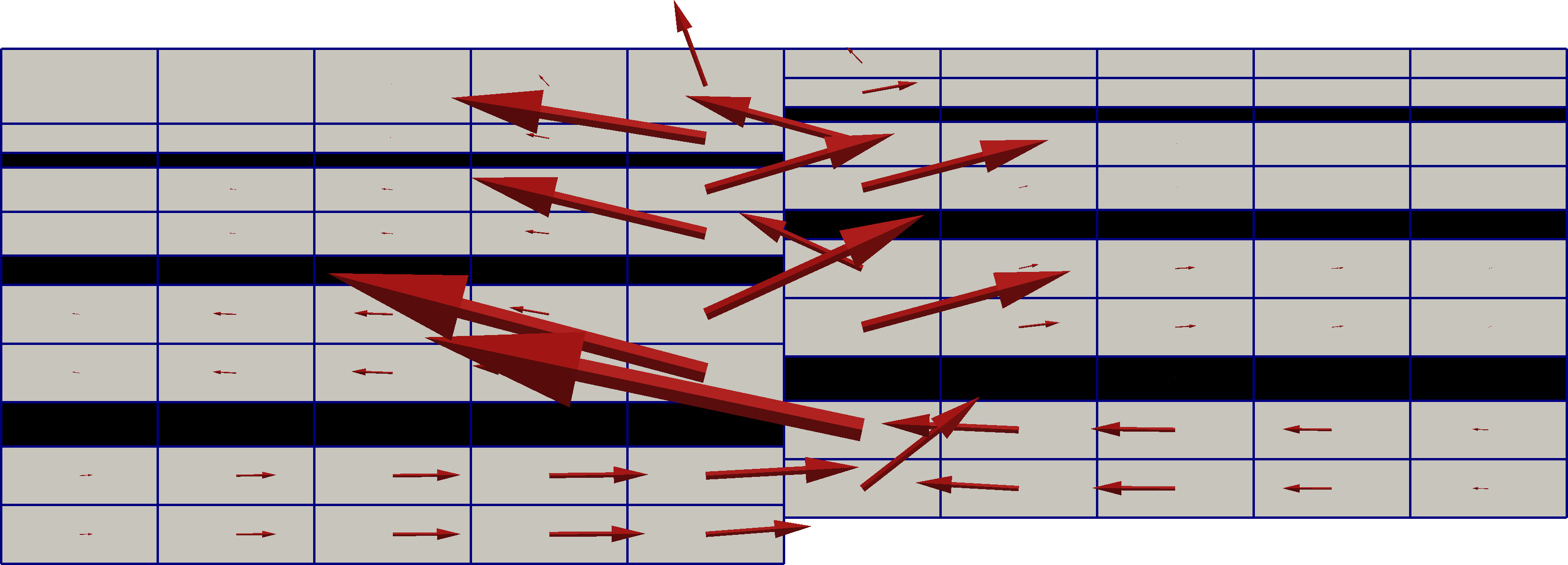}%
        \hspace{0.0025\textwidth}%
        \includegraphics[width=0.48\textwidth]{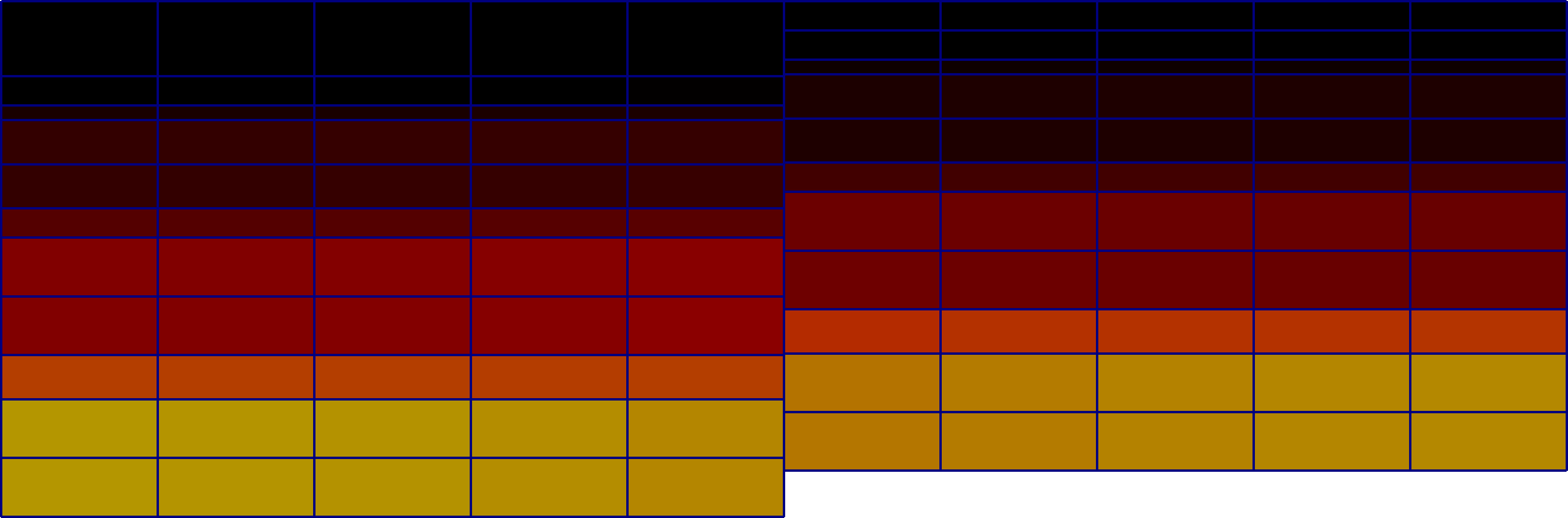}%
        \hspace{0.0015\textwidth}%
    \includegraphics[scale=0.8]{test4_legend_1}
        \label{fig:test4_1-33}
    }
    \\
    \subfloat[Current time $t=T$ and time step number $k=15$.]%
    {
        \includegraphics[width=0.48\textwidth]{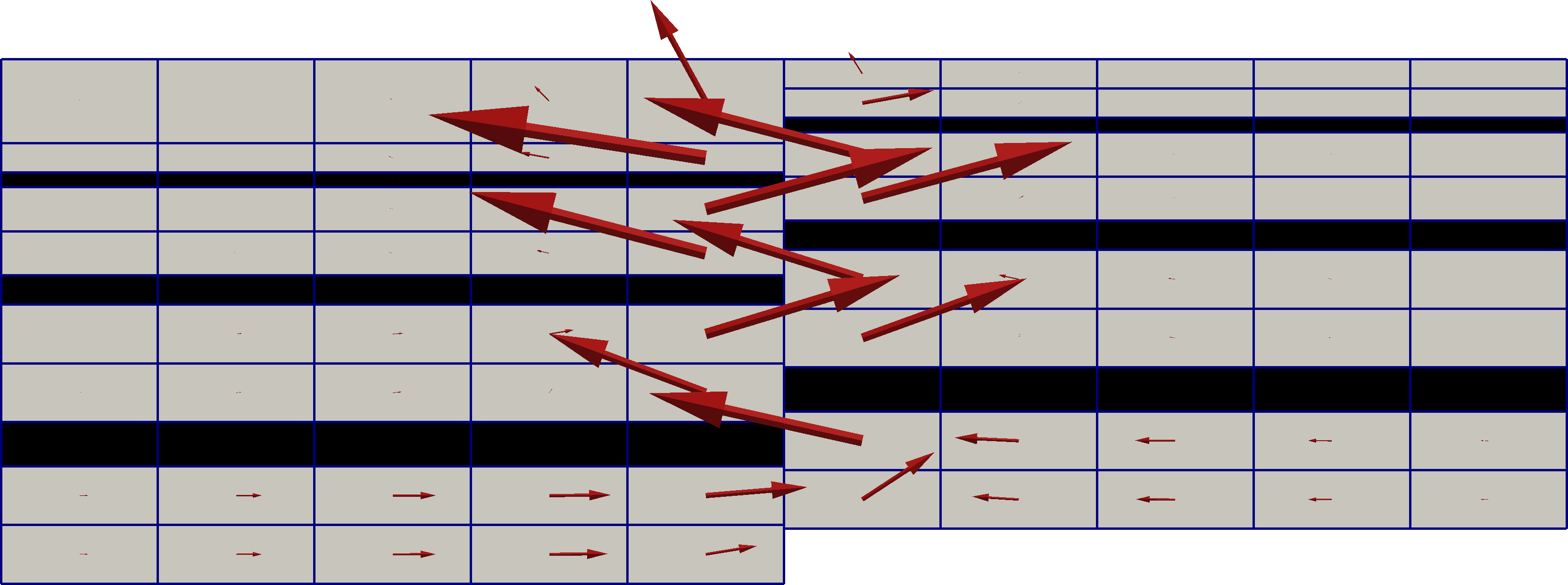}%
        \hspace{0.0025\textwidth}%
        \includegraphics[width=0.48\textwidth]{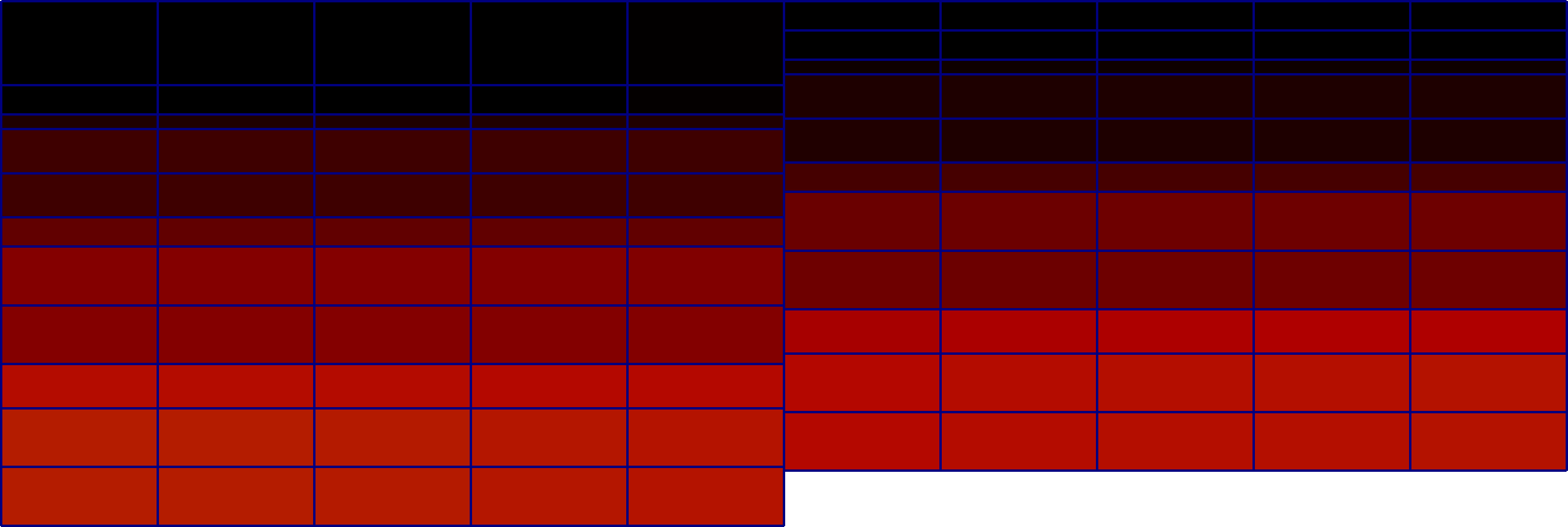}%
        \hspace{0.0015\textwidth}%
    \includegraphics[scale=0.8]{test4_legend_1}
    }
    \caption[]{Representation of different solution, pressure and Darcy velocity,
        for the neutral fault. The parameter $\chi=0.034$.}%
    \label{fig:test4_1_b}
\end{figure}
We notice that the Darcy velocity is very small before the opening of a low
permeable strata, while increases after the opening. Then once one of this
strata is opened, for example in Figure \ref{fig:test4_1-33}, the flow starts to
enter in the upper layers while the flow in the others layers tends to spread
far from the centre of the domain. Moreover we have a pressure decrease,
especially close to the fault, for each time step.

In Figure \ref{fig:test4_2} we consider a second test case where the
permeability in the fault is set to $\Lambda_f = \diag \br{10^{-13}} m^2$, so
the fault behaves like a channel for the flow. To limit the evolution of the
pressure we impose the porosity and compressibility as $c\Phi=10^{-6} Pa^{-1}$
in $\Omega^{\rm barr}$ and $c\Phi=0.2 \cdot 10^{-6} Pa^{-1}$ elsewhere.
Considering Figure \ref{fig:test4_2-22} we see a pressure drop of the cells
close to the fault, which is bigger in the bottom part of the domain where the
pressure is higher. All the arrows of the Darcy velocity are almost parallel to
the abscissa and pointing to the fault.
\begin{figure}[htbp]%
    \centering%
    \subfloat[Current time $t=0My$ and time step number $k=1$.
        The arrows are five times smaller then in the other
        representations.]%
    {
        \includegraphics[width=0.48\textwidth]{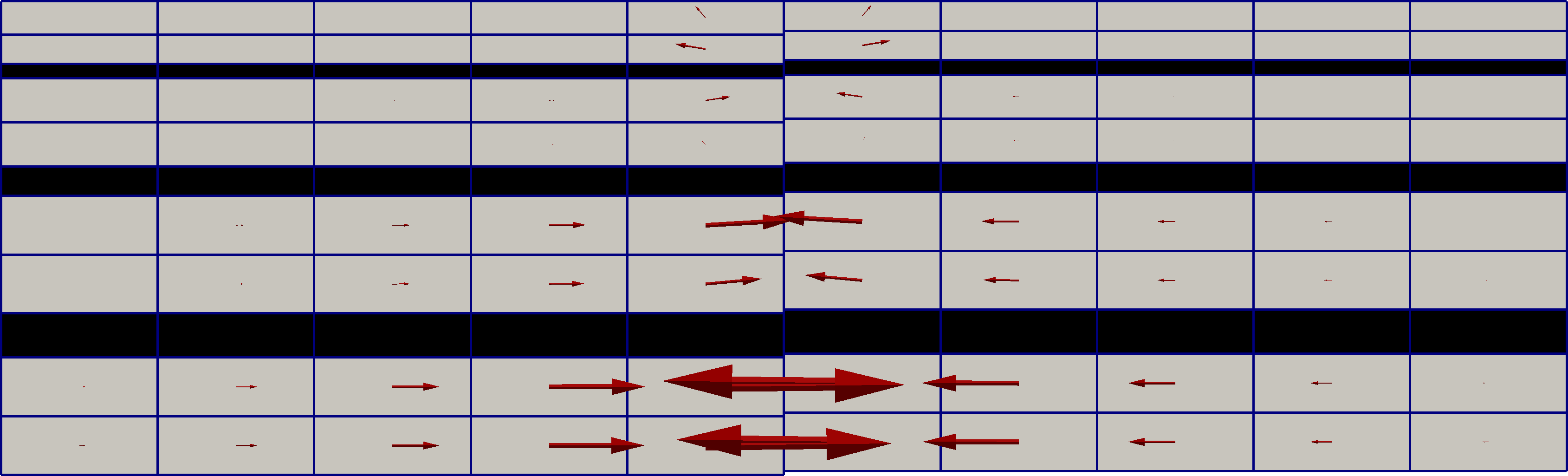}%
        \hspace{0.0025\textwidth}%
        \includegraphics[width=0.48\textwidth]{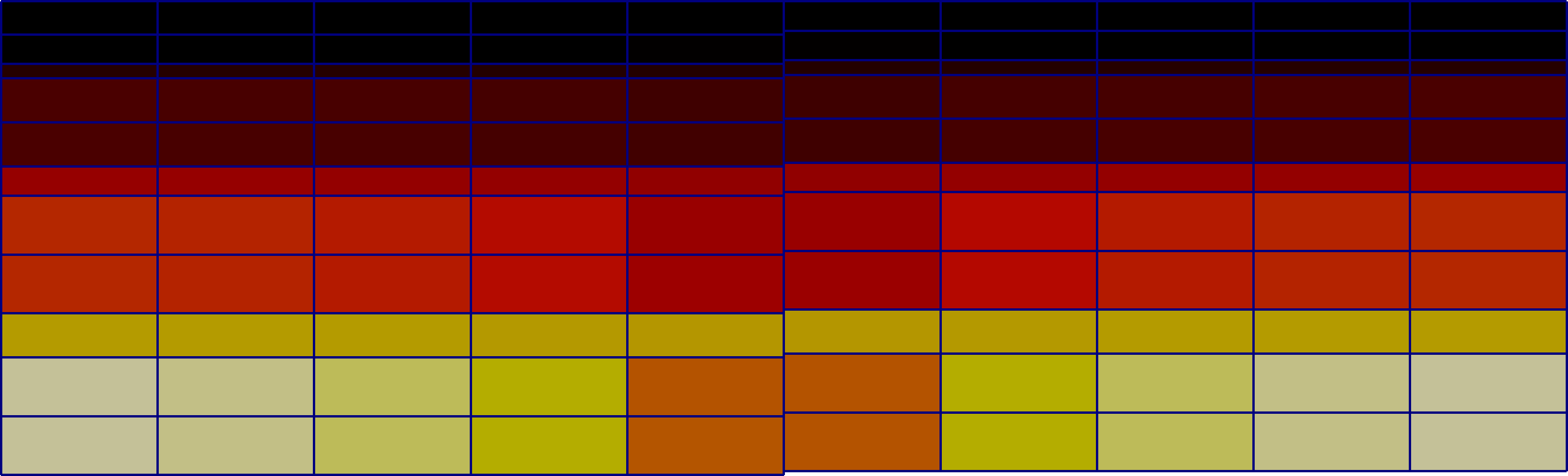}%
        \hspace{0.0015\textwidth}%
    \includegraphics[scale=0.8]{test4_legend_1}
        \label{fig:test4_2-22}
    }
    \\
    \subfloat[Current time $t=0.146My$ and time step number $k=9$.]%
    {
        \includegraphics[width=0.48\textwidth]{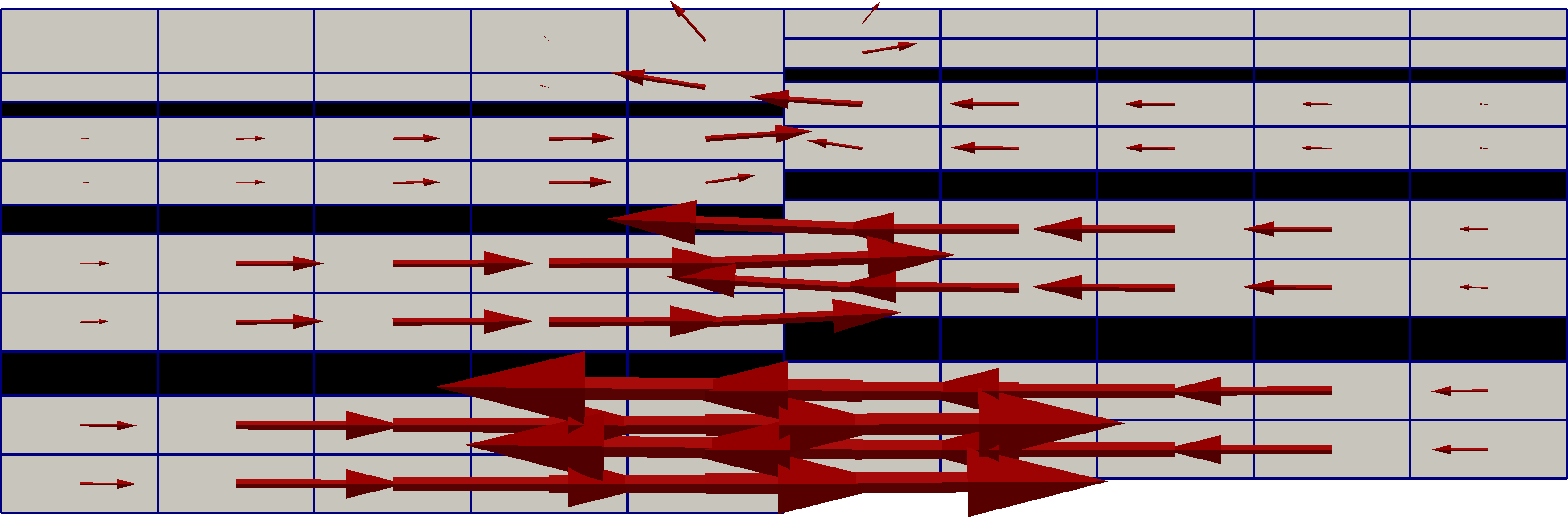}%
        \hspace{0.0025\textwidth}%
        \includegraphics[width=0.48\textwidth]{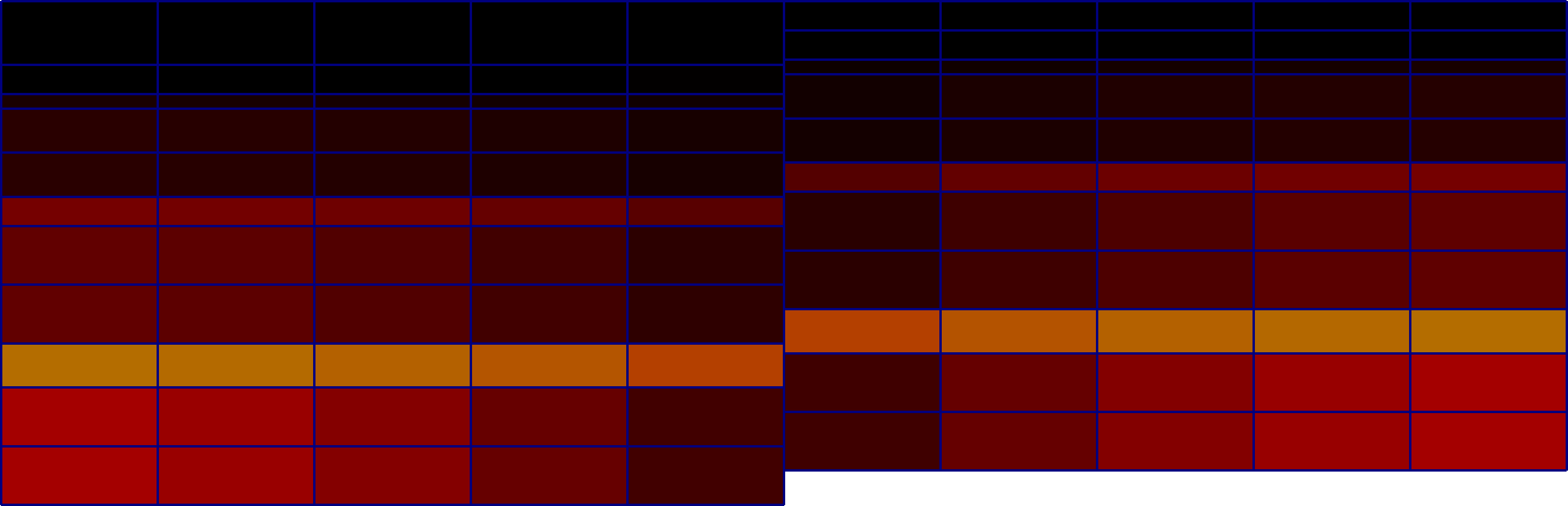}%
        \hspace{0.0015\textwidth}%
    \includegraphics[scale=0.8]{test4_legend_1}
    }
    \\
    \subfloat[Current time $t=T$ and time step number $k=16$.]%
    {
        \includegraphics[width=0.48\textwidth]{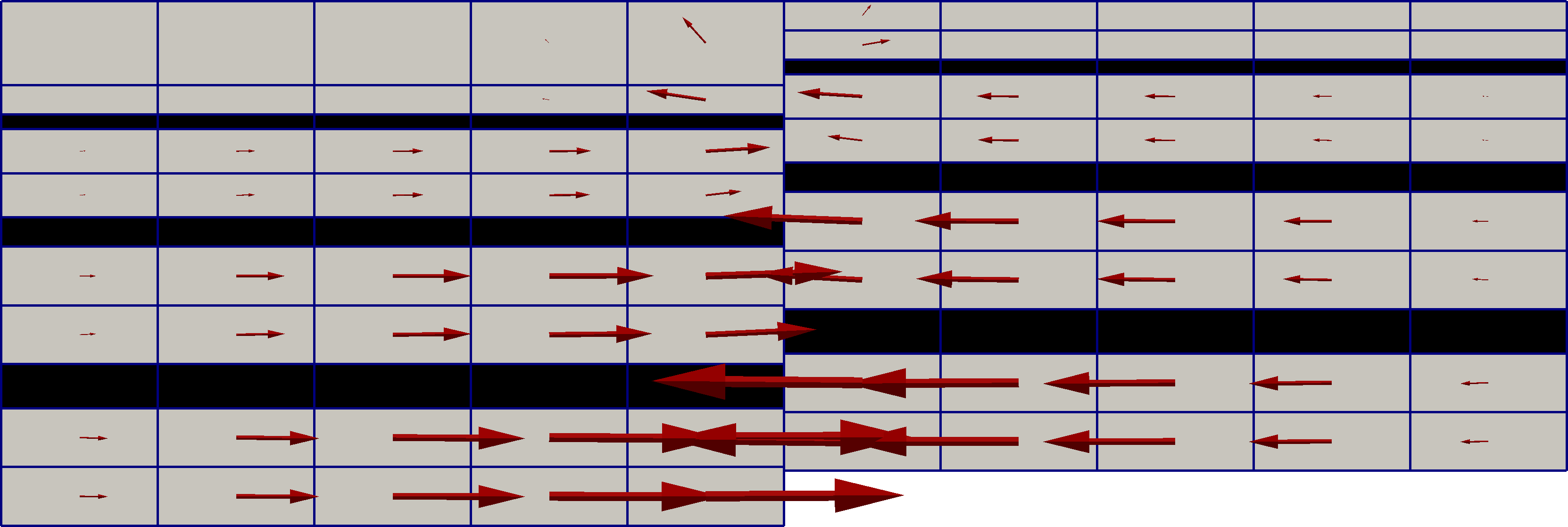}%
        \hspace{0.0025\textwidth}%
        \includegraphics[width=0.48\textwidth]{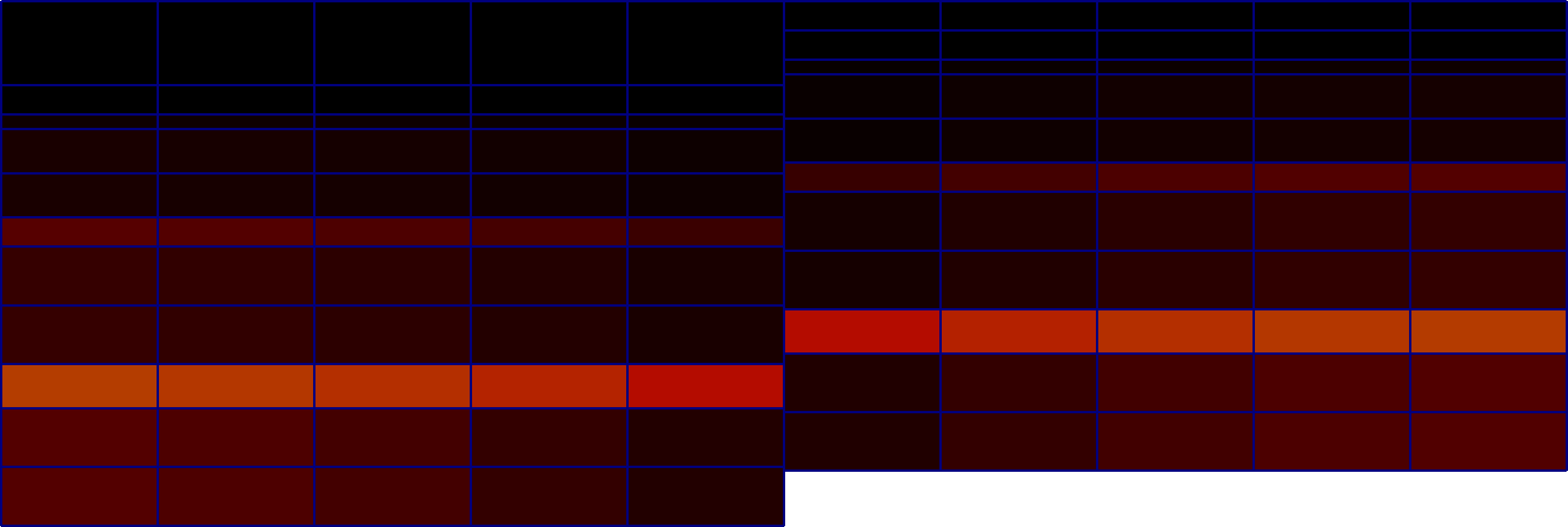}%
        \hspace{0.0015\textwidth}%
    \includegraphics[scale=0.8]{test4_legend_1}
    }
    \caption{Representation of different solution, pressure and Darcy velocity,
        for the conductive fault. The parameter $\chi = 0.02$.}%
    \label{fig:test4_2}%
\end{figure}
In the second and third images of Figure \ref{fig:test4_2} we have the same
phenomena but, since in the last time step the pressure is lower, the Darcy
velocity is higher for $k=9$ than for $k=15$. We see that for the pressure
inside the green cells this behaviour is less evident.

The last test, depicted in Figure \ref{fig:test4_3}, represents an almost
impermeable fault with permeability $\Lambda_f = \diag \br{10^{-17}} m^2$.
The fault is more permeable than $\Omega^{\rm barr}$ but less permeable than
$\Omega \setminus \Omega^{\rm barr}$.
\begin{figure}[htbp]
    \centering%
    \subfloat[Current time $t=0.167My$ and time step number $k=10$.]%
    {
        \includegraphics[width=0.48\textwidth]{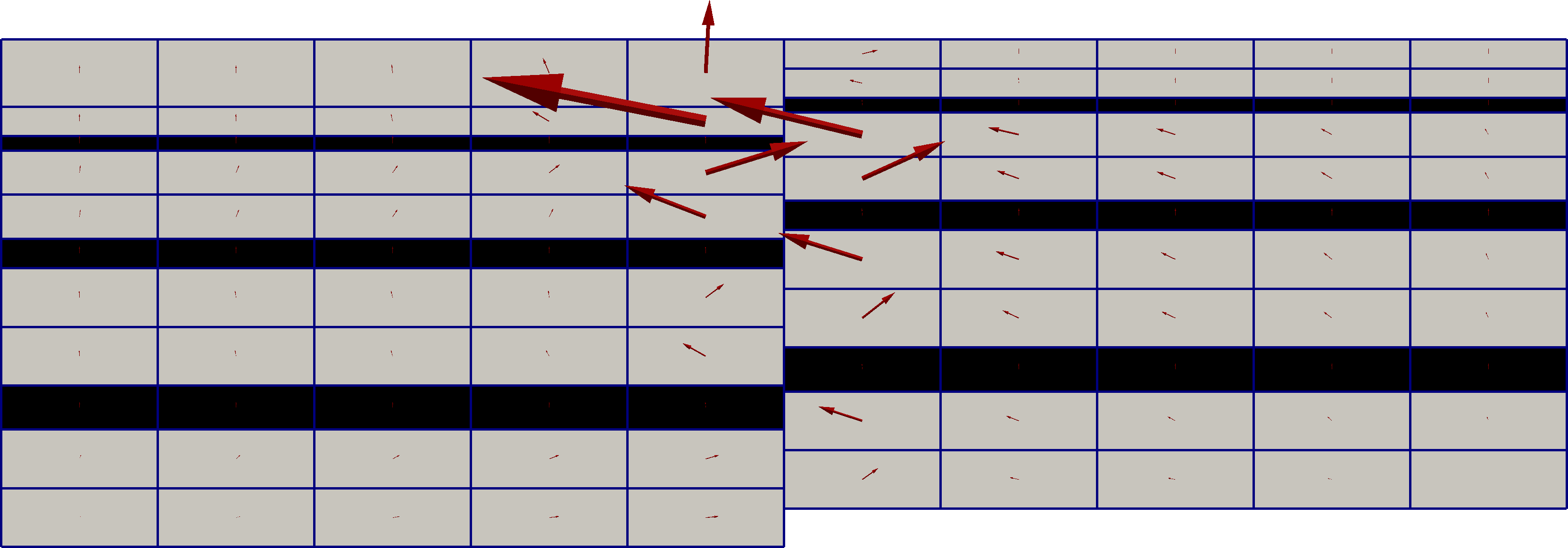}%
        \hspace{0.0025\textwidth}%
        \includegraphics[width=0.48\textwidth]{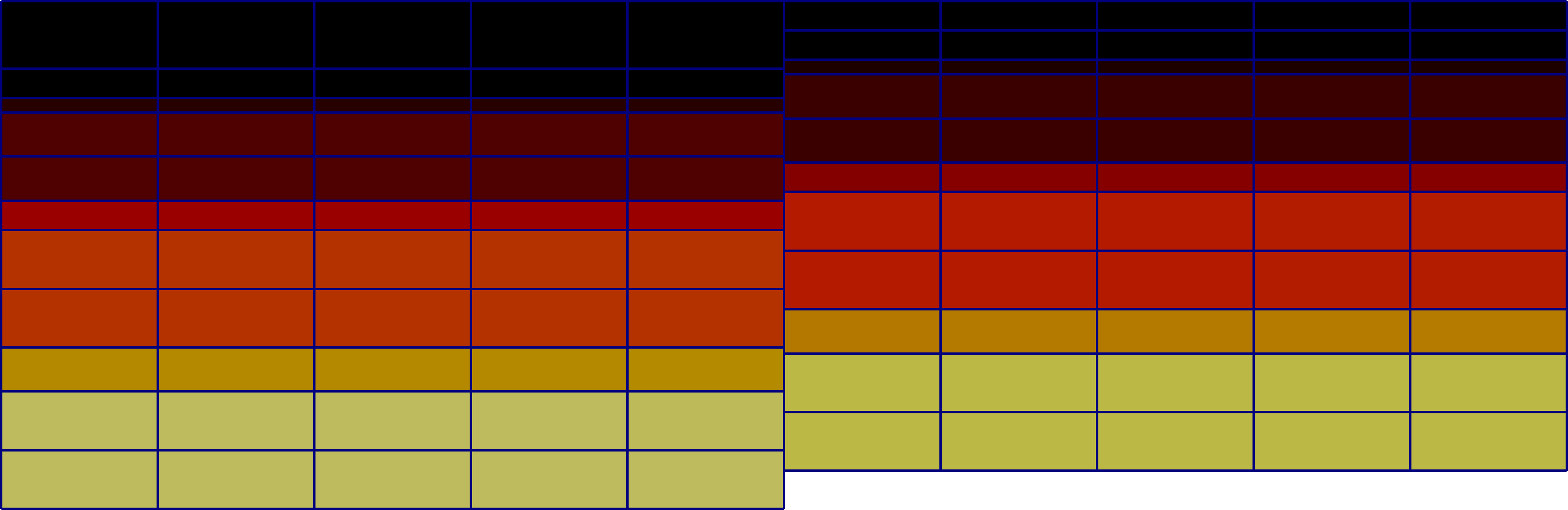}%
        \hspace{0.0015\textwidth}%
    \includegraphics[scale=0.8]{test4_legend_1}
    }
    \\
    \subfloat[Current time $t=T$ and time step number $k=15$.]%
    {
        \includegraphics[width=0.48\textwidth]{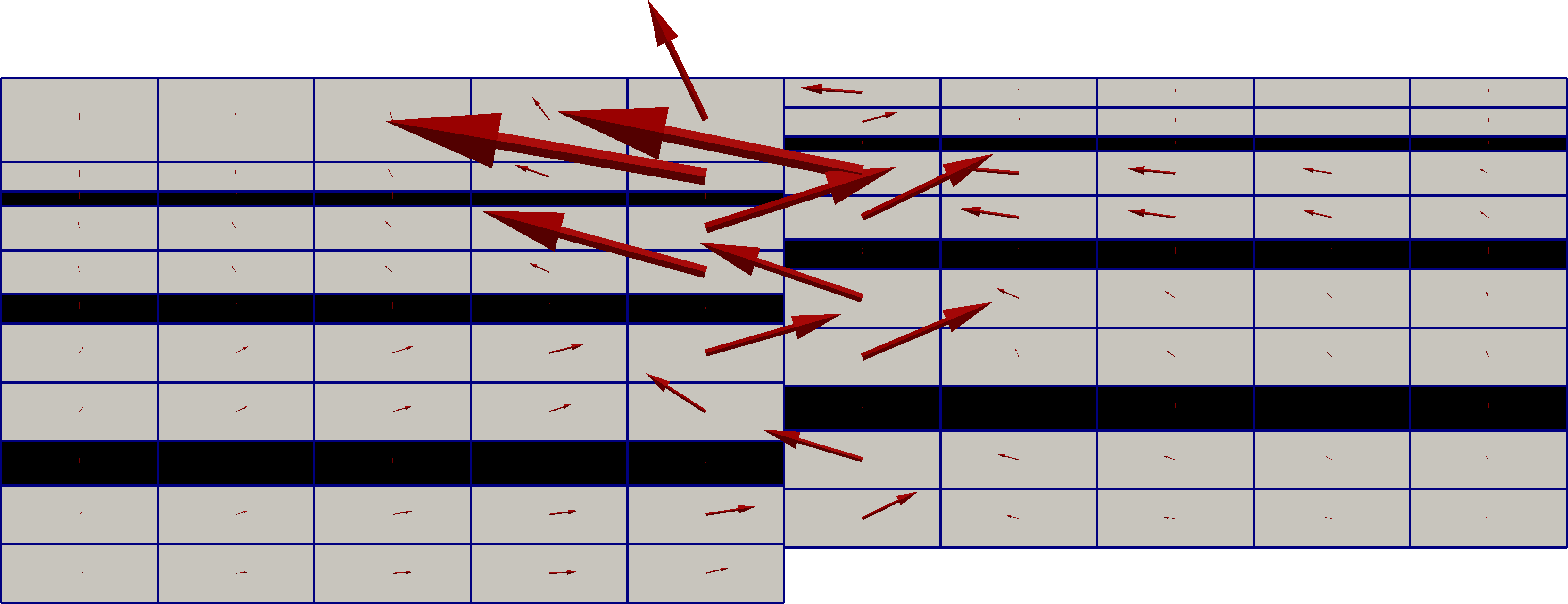}%
        \hspace{0.0025\textwidth}%
        \includegraphics[width=0.48\textwidth]{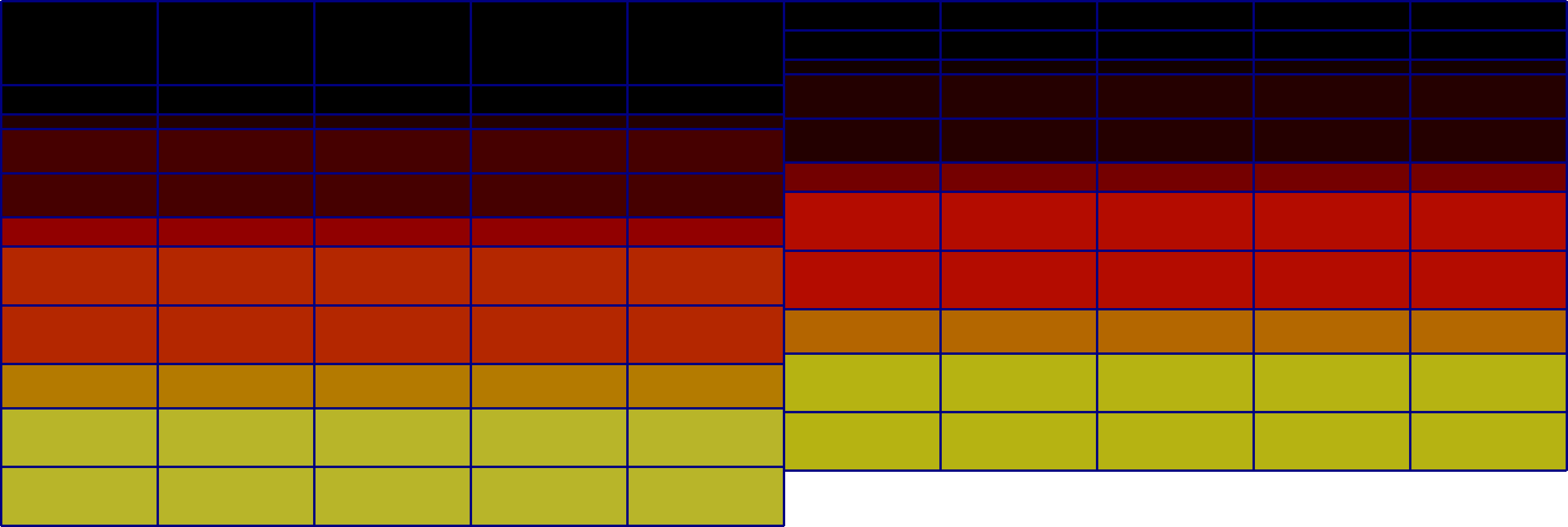}%
        \hspace{0.0015\textwidth}%
    \includegraphics[scale=0.8]{test4_legend_1}
    }
    \caption{Representation of different solution, pressure and Darcy velocity,
        for the almost impermeable fault. The parameter $\chi=0.21$.}%
    \label{fig:test4_3}%
\end{figure}
We have a very slow movement of the pressure during the simulation due to the
nature of the fault. Since the permeability of the fault is in between the
permeabilities of the porous media, once one layer of $\Omega^{\rm barr}$ is
opened the fluid starts to flow up. Contrary to the first test case considered
the end pressure is higher and the factor $\chi$ is six times higher bigger.

\begin{remark}[Maximum principle]
    We have tried to decrease the fault thickness until $d=10^{-1}m$. In this
    case we have noticed that, for the last time step, the maximum principle is
    evidently violated. The maximum of the pressure, which is reached in the
    bottom cells, is a little bigger then $10MPa$. Even if the value of the
    thickness is unphysical for our applications, this behaviour is a limitation of the proposed
    scheme. A possible explanation is the presence of the small cells
    with non-matching neighbours cells.
\end{remark}




\section{Conclusion} \label{sec:conclusion}


In this work we have derived and analysed a RM for single-phase flows
in presence of faults, which can act as low permeable strata or channels. We
consider faults that cut entirely the domain dividing the latter in disjoint
parts. The proposed model allows to handle a domain in which one part can slip,
along the fault past to the other. To easily handle the aforementioned properties
we consider a mesh for each part of the domain independent from the meshes of
other parts.  The derivation of the mathematical model is similar to
\cite{Martin2005,Tunc2012}, yet here we have used a different discretization
scheme: the hybrid finite volume scheme \cite{Eymard2010}, where one of the
advantages is the presence of degrees of freedom on faces which help the
approximation of the interface terms.  Well posedness analyses has been given
for the continuous problem as well as a convergence result for the discrete
solution to the exact one. We have also shown several numerical experiments to
estimate the convergence rates of the errors for both the porous medium and the
fault. The examples highlight also the capability of the proposed method to
handle different data configurations as well as the robustness with respect to
the mesh size ratio between different parts of the domain.



\section{Acknowledgements}


The authors warmly thank
J\'er\^ome Jaffr\'e
and
Jean E. Roberts
for many fruitful
discussions.



\begin{thebibliography}{10}

\bibitem{adler1999fractures}
Pierre~M. Adler and Jean-Fran\c{c}ois Thovert.
\newblock {\em Fractures and fracture networks}.
\newblock Springer, 1999.

\bibitem{adler2012fractured}
Pierre~M. Adler, Jean-Fran\c{c}ois Thovert, and Valeri~V. Mourzenko.
\newblock {\em Fractured Porous Media}.
\newblock Oxford University Press, 2012.

\bibitem{Alboin2002}
Clarisse Alboin, J{\'e}r{\^o}me Jaffr{\'e}, Jean~E. Roberts, and Christophe
  Serres.
\newblock Modeling fractures as interfaces for flow and transport in porous
  media.
\newblock In {\em Fluid flow and transport in porous media: mathematical and
  numerical treatment ({S}outh {H}adley, {MA}, 2001)}, volume 295 of {\em
  Contemp. Math.}, pages 13--24. Amer. Math. Soc., Providence, RI, 2002.

\bibitem{Alboin2000}
Clarisse Alboin, J{\'e}r{\^o}me Jaffr{\'e}, Jean~E. Roberts, Xuewen Wang, and
  Christophe Serres.
\newblock {\em Domain decomposition for some transmission problems in flow in
  porous media}, volume 552 of {\em Lecture Notes in Phys.}, pages 22--34.
\newblock Springer, Berlin, 2000.

\bibitem{Amir2005}
Laila Amir, Michel Kern, Vincent Martin, and Jean~E. Roberts.
\newblock D{\'e}composition de domaine et pr{\'e}conditionnement pour un
  mod{\`e}le 3{D} en milieu poreux fractur{\'e}.
\newblock In {\em Proceeding of JANO 8, 8\textsuperscript{th} conference on
  {N}umerical {A}nalysis and {O}ptimization}, December 2005.
\newblock 2005.

\bibitem{Angot2003}
Philippe Angot.
\newblock A model of fracture for elliptic problems with flux and solution
  jumps.
\newblock {\em Comptes Rendus Mathematique}, 337(6):425--430, 2003.

\bibitem{Angot2009}
Philippe Angot, Franck Boyer, and Florence Hubert.
\newblock Asymptotic and numerical modelling of flows in fractured porous
  media.
\newblock {\em M2AN Math. Model. Numer. Anal.}, 43(2):239--275, 2009.

\bibitem{Balay2013}
Satish Balay, Jed Brown, Kris Buschelman, Victor Eijkhout, William~D. Gropp,
  Dinesh Kaushik, Matthew~G. Knepley, Lois~Curfman McInnes, Barry~F. Smith, and
  Hong Zhang.
\newblock {PETS}c users manual.
\newblock Technical Report ANL-95/11 - Revision 3.4, Argonne National
  Laboratory, 2013.

\bibitem{bear1993flow}
Jacob Bear, Chin-Fu Tsang, and G~de~Marsily.
\newblock {\em Flow and contaminant transport in fractured rock}.
\newblock Academic Press, San Diego, 1993.

\bibitem{Berkowitz2002}
Brian Berkowitz.
\newblock Characterizing flow and transport in fractured geological media: A
  review.
\newblock {\em Advances in Water Resources}, 25(8-12):861--884, 2002.

\bibitem{Brezzi1991}
Franco Brezzi and Michel Fortin.
\newblock {\em Mixed and {H}ybrid {F}inite {E}lement {M}ethods}, volume~15 of
  {\em Computational {M}athematics}.
\newblock Springer Verlag, Berlin, 1991.

\bibitem{DAngelo2011}
Carlo D'Angelo and Anna Scotti.
\newblock A mixed finite element method for {D}arcy flow in fractured porous
  media with non-matching grids.
\newblock {\em Mathematical {M}odelling and {N}umerical {A}nalysis},
  46(02):465--489, 2012.

\bibitem{Droniou2010}
J{\'e}r{\^o}me Droniou, Robert Eymard, Thierry Gallou{\"e}t, and Rapha{\`a}ele
  Herbin.
\newblock A unified approach to mimetic finite difference, hybird finite volume
  and mixed finite volume methods.
\newblock {\em Mathematical Models and Methods in Applied Sciences},
  20(02):265--295, 2010.

\bibitem{Elyes2015}
Ahmed Elyes, Jaffr{\'e} J{\'e}r{\^o}me, and Jean~E. Roberts.
\newblock A 3-{D} reduced fracture model for two-phase flow in porous media
  with a global pressure formulation.
\newblock In {\em {MAMERN VI}}, Pau, France, June 2015.

\bibitem{Ern2004}
Alexandre Ern and Jean-Luc Guermond.
\newblock {\em Theory and {P}ractice of {F}inite {E}lements}.
\newblock Applied mathematical sciences. Springer, 2004.

\bibitem{Eymard2010}
Robert Eymard, Thierry Gallout, and Rapha{\`e}le Herbin.
\newblock Discretization of heterogeneous and anisotropic diffusion problems on
  general nonconforming meshes sushi: a scheme using stabilization and hybrid
  interfaces.
\newblock {\em IMA Journal of Numerical Analysis}, 30(4):1009--1043, 2010.

\bibitem{Faille2002}
Isabelle Faille, Eric Flauraud, Fr{\'e}d{\'e}ric Nataf, Sylvie
  P{\'e}gaz-Fiornet, Fr{\'e}d{\'e}ric Schneider, and Fran\c{c}oise Willien.
\newblock A {N}ew {F}ault {M}odel in {G}eological {B}asin {M}odelling.
  {A}pplication of {F}inite {V}olume {S}cheme and {D}omain {D}ecomposition
  {M}ethods.
\newblock In {\em Finite volumes for complex applications, {III}
  ({P}orquerolles, 2002)}, pages 529--536. Hermes Sci. Publ., Paris, 2002.

\bibitem{Faille2014a}
Isabelle Faille, Alessio Fumagalli, J{\'e}r{\^o}me Jaffr{\'e}, and Jean~E.
  Roberts.
\newblock Model reduction and discretization using hybrid finite volumes of
  flow in porous media containing faults.
\newblock {\em Computational Geosciences}, 20(2):317--339, 2016.

\bibitem{Falgout2002}
Robert~D. Falgout and Ulrike~Meier Yang.
\newblock hypre: A library of high performance preconditioners.
\newblock In PeterM.A. Sloot, AlfonsG. Hoekstra, C.J.Kenneth Tan, and JackJ.
  Dongarra, editors, {\em Computational Science - ICCS 2002}, volume 2331 of
  {\em Lecture Notes in Computer Science}, pages 632--641. Springer Berlin
  Heidelberg, 2002.

\bibitem{Formaggia2012}
Luca Formaggia, Alessio Fumagalli, Anna Scotti, and Paolo Ruffo.
\newblock A reduced model for {D}arcy's problem in networks of fractures.
\newblock {\em {ESAIM}: {M}athematical {M}odelling and {N}umerical {A}nalysis},
  48:1089--1116, 7 2014.

\bibitem{Frih2011}
Najla Frih, Vincent Martin, Jean~E. Roberts, and Ai~Sa{\^a}da.
\newblock Modeling fractures as interfaces with nonmatching grids.
\newblock {\em Computational Geosciences}, 16(4):1043--1060, 2012.

\bibitem{Fumagalli2012d}
Alessio Fumagalli and Anna Scotti.
\newblock A numerical method for two-phase flow in fractured porous media with
  non-matching grids.
\newblock {\em Advances in Water Resources}, 62, Part C(0):454--464, 2013.
\newblock Computational Methods in Geologic CO2 Sequestration.

\bibitem{Fumagalli2012a}
Alessio Fumagalli and Anna Scotti.
\newblock A {R}educed {M}odel for {F}low and {T}ransport in {F}ractured
  {P}orous {M}edia with {N}on-matching {G}rids.
\newblock In Andrea Cangiani, Ruslan~L. Davidchack, Emmanuil Georgoulis,
  Alexander~N. Gorban, Jeremy Levesley, and Michael~V. Tretyakov, editors, {\em
  Numerical Mathematics and Advanced Applications 2011}, pages 499--507.
  Springer Berlin Heidelberg, 2013.

\bibitem{Fumagalli2012g}
Alessio Fumagalli and Anna Scotti.
\newblock An {E}fficient {XFEM} {A}pproximation of {D}arcy {F}lows in
  {A}rbitrarily {F}ractured {P}orous {M}edia.
\newblock {\em {O}il and {G}as {S}ciences and {T}echnologies - {R}evue d'{IFP}
  {E}nergies {N}ouvelles}, 69(4):555--564, April 2014.

\bibitem{Gong2011}
Bin Gong, Guan Qin, Craig Douglas, and Shiyi Yuan.
\newblock {Detailed Modeling of the Complex Fracture Network of Shale Gas
  Reservoirs}.
\newblock {\em SPE Reservoir Evaluation \& Engineering}, 2011.

\bibitem{Grospellier2009}
Gilles Grospellier and Benoit Lelandais.
\newblock The arcane development framework.
\newblock In {\em Proceedings of the 8th Workshop on Parallel/High-Performance
  Object-Oriented Scientific Computing}, POOSC '09, pages 4:1--4:11, New York,
  NY, USA, 2009. ACM.

\bibitem{Jaffre2002}
J{\'e}r{\^o}me Jaffr{\'e}, Vincent Martin, and Jean~E. Roberts.
\newblock Generalized cell-centered finite volume methods for flow in porous
  media with faults.
\newblock In {\em Finite volumes for complex applications, {III}
  ({P}orquerolles, 2002)}, pages 343--350. Hermes Sci. Publ., Paris, 2002.

\bibitem{Jaffre2011}
J{\'e}r{\^o}me Jaffr{\'e}, Mokhles Mnejja, and Jean~E. Roberts.
\newblock A discrete fracture model for two-phase flow with matrix-fracture
  interaction.
\newblock {\em Procedia Computer Science}, 4:967--973, 2011.

\bibitem{Karimi-Fard2004}
Mohammad Karimi-Fard, Louis~J. Durlofsky, and Khalid Aziz.
\newblock An {E}fficient {D}iscrete-{F}racture {M}odel {A}pplicable for
  {G}eneral-{P}urpose {R}eservoir {S}imulators.
\newblock {\em SPE Journal}, 9(2):227--236, 2004.

\bibitem{Martin2005}
Vincent Martin, J{\'e}r{\^o}me Jaffr{\'e}, and Jean~E. Roberts.
\newblock Modeling {F}ractures and {B}arriers as {I}nterfaces for {F}low in
  {P}orous {M}edia.
\newblock {\em SIAM J. Sci. Comput.}, 26(5):1667--1691, 2005.

\bibitem{Quarteroni1994}
Alfio Quarteroni and Alberto Valli.
\newblock {\em Numerical Approximation of Partial Differential Equations},
  volume~23 of {\em Springer Series in Computational Mathematics}.
\newblock Springer-Verlag, Berlin, 1994.

\bibitem{Tunc2012}
Xavier Tunc, Isabelle Faille, Thierry Gallou{\"e}t, Marie~Christine Cacas, and
  Pascal Hav{\'e}.
\newblock A model for conductive faults with non-matching grids.
\newblock {\em Computational Geosciences}, 16:277--296, 2012.

\end{thebibliography}


\end{document}